\documentclass[]{amsart}
\usepackage[margin=3 cm]{geometry}
\usepackage{orcidlink}
\usepackage{tikz-cd}
\usepackage{supertabular}
\usepackage{amsfonts,latexsym,rawfonts,amsmath,amssymb, amsthm}
\usepackage{longtable,setspace}
\usepackage{algorithm2e}
\usepackage{caption}
\usepackage[T1]{fontenc}
\usepackage{latexsym,lscape,rawfonts}
\usepackage{graphicx}
\usepackage{array}
\newcolumntype{P}[1]{>{\centering\arraybackslash}p{#1}}
\newcolumntype{M}[1]{>{\centering\arraybackslash}m{#1}}

\usepackage{mathrsfs, enumitem}
\usepackage[title]{appendix}
\usepackage{hyperref}

\usepackage{galois}
\usepackage[all,cmtip]{xy}
\usepackage{extarrows,color}
\usepackage{times}
\usepackage{thmtools}

\newtheorem{thm}{Theorem}[section]
\newtheorem{manualthmInner}{Theorem}
\newenvironment{manualthm}[1]{%
  \renewcommand\themanualthmInner{#1}%
  \manualthmInner
}{%
  \endmanualthmInner
}

\newtheorem*{fixed point criterion}{Fixed point criterion}
\newtheorem{cor}[thm]{Corollary}
\newtheorem{lem}[thm]{Lemma}
\newtheorem{prop}[thm]{Proposition}
\newtheorem{claim}{Claim}

\theoremstyle{definition}
\newtheorem{defn}[thm]{Definition}
\newtheorem{exam}[thm]{Example}
\newtheorem{ques}[thm]{Question}
\theoremstyle{remark}
\newtheorem{rem}[thm]{Remark}
\numberwithin{equation}{section}

\newcommand{\Z}{\mathbb Z}

\newcommand{\N}{\mathbb N}
\newcommand{\RR}{\mathcal R}
\newcommand{\cl}{\mathrm{CL}}

\newcommand{\W}{\mathcal{W}}
\newcommand{\nf}{\mathfrak{nf}}
\newcommand{\rsp}{\mathfrak{rp}}
\newcommand{\ciw}{\mathfrak{CI}}
\newcommand{\llfr}{\mathrm{LLFR}}
\newcommand{\swl}{\tau}
\newcommand{\gs}{\mathcal{S}}
\newcommand{\groupelement}{\mathcal{E}}
\newcommand{\generalconj}{\mathcal{C}}
\newcommand{\primitiveconj}{\mathcal{P}}

\title{Word Length Formulae, Normal Forms, conjugation and root-finding algorithms in Surface Groups}
	
	\author[Ke Wang]{Ke Wang \orcidlink{0009-0005-7108-3725}}
	\address{School of Mathematics and Statistics, Xi'an Jiaotong University, Xi'an 710049, P. R. China}
	\email{keqiyehuopo@stu.xjtu.edu.cn}
	
	\author[Qiang Zhang]{Qiang Zhang \orcidlink{0000-0001-6332-5476}}
	\address{School of Mathematics and Statistics, Xi'an Jiaotong University, Xi'an 710049, P. R. China}
	\email{zhangq.math@mail.xjtu.edu.cn}
	
	\author[Xuezhi Zhao]{Xuezhi Zhao \orcidlink{0000-0002-6057-7237}}
	\address{School of Mathematical Sciences, Capital Normal University, Beijing 100048, P. R. China}
	\email{zhaoxve@mail.cnu.edu.cn}
	
	\thanks{The authors are partially supported by National Natural Science Foundation of China (No. 12471066 and 12331003) and the Shaanxi Fundamental Science Research Project for Mathematics and Physics (No. 23JSY027).}
	
	
	\date{\today}
	
\begin{document}

\begin{abstract}
In this paper, we mainly study the following symmetric presentation of the surface group $$\pi_1(\Sigma_g)=\left\langle c_1,\dots, c_{2g}\mid c_1\cdots c_{2g}c_1^{-1}\cdots c_{2g}^{-1}\right\rangle.$$  For every nontrivial element $x\in \pi_1(\Sigma_g)$ and $k\geq 2$, we obtain a uniform representative of the normal forms $\nf(x^k)$ of $x^k$ under the length-lexicographical order:
$$\nf(x^k) = \overline{LW^{k-2}R}.$$
Building on this result, we establish a new relation among these normal forms, and then derive the following three formulae related to the word length: \begin{enumerate}\item $|x^2|>|x|$;\item $|x^k|=(k-1)(|x^2|-|x|)+|x|$;\item $\lim_{k\to\infty}\frac{|x^k|}{k}=|x^2|-|x|$.\end{enumerate} Furthermore, we extend these results to obtain a coarser analogue  for every minimal geometric presentation. We then define normal forms of conjugacy classes in $\pi_1(\Sigma_g)$ and provide a criterion for determining the conjugacy of group elements. As a consequence, we provide efficient algorithms for solving the root-finding and conjugacy problems. Finally, we present applications to the computation of several growth rates.
\end{abstract}
\maketitle
\tableofcontents

\section{Introduction}\label{sect introduction}

In the realm of geometric group theory, surface groups and word length stand as two fundamental concepts that have spurred an abundance of related research. For example, Grigorchuk and Nagnibeda \cite{GN97} investigated the word length of surface groups and obtained the complete growth functions of surface groups; Calegari \cite{Ca08} gave a lower bound for the word lengths of power elements with respect to characteristic generating sets; Erlandsson \cite{Er19} gave the asymptotic growth of the number of immersed curves of bounded word length on a hyperbolic surface; Malestein and Putman \cite{MP24} provided lower bounds for the word length of nontrivial elements in the lower central series groups of surface groups. In this paper, we mainly focus on the word length of power elements in surface groups.

\subsection{Motivation}
Given a finitely generated free group $F$ with a  freely generating set, it is easy to get the word length of $a^n$ for any element $1\neq a\in F$: cyclically reduce the element $a$ and obtain a cyclically freely reduced form $b$, then the word length formula is
$$ |a^n|=|b|\cdot(n-1)+|a|.$$
Moreover, we have $|a^2|>|a|$ and $|b|=|a^2|-|a|$. On the other hand, every conjugacy class in $F$ can be uniquely represented by a freely reduced cyclic word. It means that the conjugacy problem is solvable in free groups. Since free homotopy classes of curves on a bordered surface are in one-to-one correspondence with conjugacy classes in the fundamental group of this surface, these representations of conjugacy classes often serve as a bridge between the algebra and geometry of the bordered surface. These results are well-known and can be found in some textbooks \cite{LS01,MKS04}. 

Now, when considering the (oriented closed) surface groups, the situation seems more complicated because there is an additional relator. It is not easy to tell whether the above three formulae still hold for surface groups. If we approach from the perspective of Gromov hyperbolic groups \cite{Gr87}, the word length of power elements of surface groups should follow the following formula:
$$|a^n|=k\cdot n+c(n),$$
where $c(n)$ is uniformly bounded and $k$ is a constant. However, since hyperbolicity is a coarse property, it is difficult to obtain precise results through this method. 

All hyperbolic groups have the "bound conjugacy diagrams" property (see \cite[Section 3]{AC17}), that is, two elements are conjugate if and only if their cyclic permutations are conjugated by an element of universally bounded length. However, since this property is also a coarse property, we cannot obtain exact representations of every conjugacy class as in the cases of free groups.

In this paper, we show that in surface groups, the word length formulae are as precise as in free groups, with exceptions arising for conjugacy classes.

\subsection{Word length formulae for power elements}

In this paper, a \emph{surface group} means the fundamental group $\pi_1(\Sigma_g)$ of a closed orientable surface $\Sigma_g$ of genus $g\geq 2$. 

Let $G$ be a surface group with the generating set $\gs=\{c_1,\dots,c_{2g}\}$ and the following \emph{symmetric presentation}
\begin{equation}\label{symmetric presentation}
G=\pi_1(\Sigma_g)=  \left\langle c_1,\dots , c_{2g}\mid c_1\cdots c_{2g}c_1^{-1}\cdots c_{2g}^{-1} \right\rangle .
\end{equation}
If we give the following order of the generators
\begin{equation}\label{order of generators}
    c_{2g}^{-1}\succ \cdots \succ c_2^{-1}\succ c_1^{-1}\succ c_1\succ c_2\cdots\succ  c_{2g},
\end{equation}
there will be a length-lexicographical order on the word set $\W(\gs)$ of $G$. Then, for an element $x\in G$, there is a unique \emph{normal form} $\nf(x)$ of $x$, which is the word with the minimal order that represents $x$, see Definition \ref{normal form}.

The starting point is the following structural theorem.  It is a simplified form of the more precise Theorem \ref{thm normal forms of power elements (original main thm1)}, where the words appearing below are described explicitly. We establish a relation among the normal forms of all powers of the given word, meaning that the word length behavior of power elements exhibits a high degree of regularity. On the other hand, this theorem also shows every nontrivial element acts by translation along a geodesic axis in the Cayley graph. 

\begin{thm}\label{main thm normal forms of powers}
    Let $G$ be a surface group with the symmetric presentation (\ref{symmetric presentation}) and the length-lexicographical order (\ref{order of generators}). For any element $1\neq x\in G$ and $k\geq 2$, the normal form of $x^k$ admits an expression that depends uniformly on $k$:
    $$\nf(x^k)=\overline{LW^{k-2}R}.$$
\end{thm}

Given a word metric $|\cdot|_\gs$ for the group $G$ with respect to a generating set $\gs$, each $x \in G$ can be assigned such a limit
$$\swl_\gs(x):=\lim_{n\to \infty} \frac{|x^n|_\gs}{n},$$
which is said to be the \emph{stable word length} (or \emph{translation number}) of $x$, denoted by $\swl_\gs(x)$.

We always omit the subscript $\cdot_\gs$ if the generating set $\gs$ is clear. There has been a lot of research on the translation number, such as \cite{Gr87,GS91ann,Ka97,Ye25}. It is noted in \cite[Section 6]{GS91ann} that this limit always exists. Kapovich \cite{Ka97} proved that for some small cancellation groups, their translation numbers are strongly discrete. We, on the other hand, will prove that in the symmetric presentation, all translation numbers are integers. The following word length formulae can be obtained by Theorem \ref{main thm normal forms of powers}.

\begin{thm}\label{main thm word length formulae}
	 Let $G$ be a surface group with the symmetric presentation (\ref{symmetric presentation}). Then, for any $1\neq x\in G$, the following three formulae hold:
	\begin{enumerate}
		\item $|x^2|> |x|$;
		\item $|x^k|=(k-1)(|x^2|-|x|) +|x|, ~\forall k\geq 1$;
        \item $\swl(x)=|x^2|-|x|$.
	\end{enumerate}
\end{thm}

We also consider other presentations of surface groups. In Section \ref{sect from sym. pre. to any pre.}, we first study isomorphisms between Cayley graphs of any \emph{minimal geometric presentation} (see Definition \ref{def minimal geometric presentation}) and the symmetric presentation (\ref{symmetric presentation}), and then obtain a coarser analogue.

\begin{thm}\label{main thm4 word length formulae more rough}
Let $\pi_1(\Sigma_g)$ be a surface group with a minimal geometric presentation $P$. Then there exists a positive integer $t$ dividing $4g$ such that for any $1\neq x\in \pi_1(\Sigma_g)$, the following formulae (with respect to $P$) hold:
	\begin{enumerate}
		\item $|x^{2t}|> |x^{t}|$;
		\item $|x^{kt}|=(k-1)(|x^{2t}|-|x^{t}|) +|x^{t}|, ~\forall k\geq 1$;
        \item $\tau(x)=\frac{|x^{2t}|-|x^{t}|}{t}$.
	\end{enumerate}
In particular, the canonical presentation  $\left\langle a_1, \ldots, a_{2g} \mid\prod_{i=1}^g[a_{2i-1},a_{2i}]\right\rangle$ is minimal geometric and $t$ can be taken as $2g$.
\end{thm}

\begin{ques}
	For the canonical presentation of a surface group, do the formulae in Theorem \ref{main thm word length formulae} hold? More generally, do those formulae hold for every minimal geometric presentation?
\end{ques}

\subsection{Normal forms of conjugacy classes}

The periodic block ``$\overline{W}$'' in Theorem \ref{main thm normal forms of powers} leads naturally to a distinguished representative of the conjugacy class. Our next theorem is about conjugacy classes and shows when two elements are conjugate. For any conjugacy class $[x]$, we define the \emph{normal form $\nf([x])$ of the conjugacy class $[x]$} as the word with the minimal order in it, see Definition \ref{normal form}. After seeing Theorem \ref{thm normal forms of power elements (original main thm1)}, what readily comes to mind is whether the word $\overline{W}$ in Eq. (\ref{eq. nf(xk)}) is the normal form $\nf([x])$. The answer turns out to be negative, because such a $\overline{W}$ is not unique. In fact, different $\overline{W}$'s differ by a cyclic permutation. Then, is it feasible to obtain $\nf([x])$ by cyclically permuting $\overline{W}$? The answer remains negative. A simple counterexample (see \cite[Page 13]{GZ22}) is that $c_3c_4c_1^{-1}$ and $c_4c_3c_1^{-1}$ are conjugate
in the surface group $\pi_1(\Sigma_2)=\langle c_1,c_2,c_3,c_4\mid c_1c_2c_3c_4c_1^{-1} c_2^{-1}c_3^{-1}c_4^{-1}\rangle$,  
$$c_2^{-1}c_1^{-1}(c_4c_3c_1^{-1})c_1c_2=c_3c_4c_1^{-1}c_2^{-1}c_2=c_3c_4c_1^{-1},$$ but the words $\overline{c_3c_4c_1^{-1}}$ and $\overline{c_4c_3c_1^{-1}}$ (we always overline a word to distinguish it from a group element) remain distinct after cyclic permutations.
Therefore, the situation seems more complicated than we imagined. 

In fact, this is the only type of exception, as detailed in item (1b) of the following theorem. 
\begin{restatable}{thm}{mainthmconjclass}\label{main thm3 conj. class}
 Let $G$ be a surface group with the symmetric presentation (\ref{symmetric presentation}) and the length-lexicographical order (\ref{order of generators}). Then
    \begin{enumerate}
        \item Two nontrivial elements $x, y\in G$ are conjugate if and only if one of the following holds:
    \begin{enumerate}
        \item [(a)] $\ciw(x)=\overline{x_1\cdots x_m}$ and $\ciw(y)= \overline{x_{i+1}\cdots x_mx_1\cdots x_i}$ for some $1\leq i\leq m$;
        \item [(b)] $\ciw(x)=\overline{(b_{i+1}\cdots b_{2g-1}b_1\cdots b_i)^t}$ and $\ciw(y)=\overline{(b_{j}\cdots b_1b_{2g-1}\cdots b_{j+1})^t}$ for some $t\geq 1$, $ 1\leq i,j\leq 2g-1$ and a cyclic permutation $b_1\cdots b_{4g}$ of the relator $c_1\cdots c_{2g}c_1^{-1}\cdots c_{2g}^{-1}$ or its inverse.
    \end{enumerate}
    \item If $\ciw(x)$ has a form as in the above case (b), then 
    $$\nf([x])=\min\{\overline{(b_{j+1}\cdots b_{2g-1}b_1\cdots b_j)^t}, ~\overline{(b_{j}\cdots b_1b_{2g-1}\cdots b_{j+1})^t}\mid 1\leq j\leq 2g-1\}.$$ Otherwise 
$$\nf([x])=\min\{\overline{W}\mid \overline{W}\mbox{ is a cyclic permutation of }\ciw(x)\}.$$
    \end{enumerate}
\end{restatable}
Here, $\ciw(x)$ is conjugate to $x$ and called the \emph{cyclically irreducible core} of $x$, which is a specific selection of $\overline{W}$ from Eq. (\ref{eq. nf(xk)}) in Theorem \ref{thm normal forms of power elements (original main thm1)}:
  $$ \nf(x^k)=\overline{Y\nf(X^k)Y^{-1}}=\overline{Yx_1\cdots x_{i}(x'_{1}\cdots x'_{j} W^{k-2}x'_{j+1}\cdots x'_{n'})x_{i+1}\cdots x_nY^{-1}}.$$
See Eq. (\ref{def ciw}) and Corollary \ref{cor cyclically irreducible word notation} for further details.

\begin{rem}\label{rem for main thm3 conj. class}
Thus the classical free-group description of conjugacy by cyclic permutations fails for surface groups in only one explicitly classified way. In particular, every conjugacy class contains at most two cyclically irreducible representatives up to cyclic permutation.

In a free group with a freely generating set, clearly, there is a bijection between the conjugacy class set and the freely reduced cyclic word set. However, in a surface group even with a minimal generating set and a length-lexicographical order, it is hard to say there is such a bijection as shown by the counterexample by Gu and Zhao before Theorem \ref{main thm3 conj. class}. Note that, for any total order on the word set, it is easy to define the normal form of a conjugacy class as the minimal word representing an element in it. Consequently, the most significant contribution of Theorem \ref{main thm3 conj. class} lies not in providing the definition of the normal form of a conjugacy class, but rather in specifically elaborating on the concrete differences between cyclically irreducible words and normal forms of conjugacy classes.
\end{rem}

Note that the conjugacy problem is a fundamental issue in combinatorial group theory and has therefore been extensively studied. For example, as early as the beginning of the last century, Dehn \cite{De11, De12, De87} solved the conjugacy problem for surface groups. In \cite{De12}, he first presented an algorithm for the word problem, and then reduced the conjugacy problem to finitely many word problems. Furthermore, there were many studies on the conjugacy problem for small cancellation groups, hyperbolic groups, and biautomatic groups \cite{Ep92}. In this paper, we will provide \textbf{Algorithm \ref{algorithm: normal form of an element}} for the word problem, and also provide \textbf{Algorithm \ref{algorithm: find normal form of conjugacy class}},  which is used to solve the conjugacy problem and compute a conjugator when the answer is affirmative. In addition, the root-finding \textbf{Algorithm \ref{algorithm: root-finding}} will be derived from Theorem \ref{thm normal forms of power elements (original main thm1)}, along with a commensurability problem algorithm as stated in the following corollary. The complexity of these algorithms is about $O(n^2)$, where $n$ is the word length of the input word. For the differences between our algorithms and Dehn's algorithms, see Subsection \ref{subsect Dehn algorithm}.

\begin{restatable}{cor}{corconjpoweralgo}\label{cor conj-power algo.}
   Given elements $x,y$ in a surface group $G$, there is an algorithm to decide whether there is an element $z\in G$ and nonzero integers $m,n$ such that $x^m=zy^nz^{-1}$ and, in the affirmative case, compute $m, n$ and $z$. 
\end{restatable}

\subsection{Applications and further directions}

The exact periodicity of power elements above suggests several applications, where the
normal form retains information that is invisible to coarse hyperbolic
geometry. More precisely, we have the following applications and further directions.

\begin{itemize}
\item \textbf{Complete conjugacy growth function}. The first possible application concerns the conjugacy growth series of surface groups. Antol\'in and Ciobanu \cite{AC17} proved that the conjugacy growth series of every non-elementary hyperbolic group is transcendental. Nevertheless, transcendence does not preclude explicit arithmetic formulae for its coefficients. For free groups, Rivin \cite{Ri10} obtained such formulae by relating conjugacy classes to cyclically reduced words through the action of cyclic permutation and Burnside's lemma; primitive classes can then be separated from proper powers by M\"obius inversion. Our results in Section \ref{subsect conjugacy growth} provide the corresponding combinatorial ingredients for surface groups. 
A sufficiently precise enumeration of cyclically irreducible normal forms can lead to explicit coefficient formulae for the conjugacy growth series of surface groups, analogous to Rivin's formulae for free groups.

\item \textbf{Conjugator length function}. Given a group $G$ with a finite generating set, for any conjugate elements $u\sim v$ in $G$, define $\cl(u,v)$ as follows:
$$\cl(u,v):=\min\{|w|\mid u=w^{-1}vw, ~w\in G\}.$$
The \emph{conjugator length function} $\cl:\N\to\N$ is defined to be the least integer $N$ with $\cl(u,v)\leq N$ for all conjugate elements $u$ and $v$ whose lengths satisfy $|u|+|v|\leq n$. See \cite{BRS26} for more details. In our subsequent work \cite{WZ26}, for the symmetric presentation of $\pi_1(\Sigma_g)$, by using the detailed structure of conjugation reductions arising from the present normal forms, we obtain $$n-1\leq \mathrm{CL}(2n)=\mathrm{CL}(2n+1)\leq n+8g-1.$$
Consequently, $$\mathrm{CL}(N)=\frac{N}{2}+O_g(1),$$
and the coefficient $1/2$ is optimal. This gives a quantitative refinement of the conjugacy problem and shows that the normal-form and reduction techniques developed here can be used not only to decide conjugacy and construct conjugators, but also to obtain sharp control on the minimal length of conjugators.

\item \textbf{Asymptotic distribution of intersection numbers}. We also have an ongoing work \cite{WZZ26} concerning the geometric intersection number on a surface. Chas and Lalley proved in \cite{CL12} that, for an orientable surface with boundary, the self-intersection number of a random free homotopy class of combinatorial length $n$, after suitable normalization, is asymptotically Gaussian. Their argument relies essentially on a canonical coding by cyclically reduced words together with a combinatorial formula for self-intersection numbers. For closed surfaces, the presence of the surface relation makes such a coding substantially more delicate.

The Gaussian limit for closed surfaces was later established by Wroten \cite{Wr14} using symbolic coding and Markov-chain methods. Meanwhile, Gu and Zhao \cite{GZ22} developed a direct combinatorial procedure for computing geometric intersection numbers from cyclically irreducible words on closed orientable surfaces. The normal forms of conjugacy classes obtained in the present paper complement this by providing a canonical representative for every free homotopy class in the original generating alphabet, while Theorem \ref{main thm3 conj. class} gives explicit control of the possible ambiguity among cyclically irreducible representatives.

These results suggest a direct combinatorial framework for studying intersection statistics on closed surfaces. Beyond recovering asymptotic normality, one may investigate finer questions such as moments, large deviations, local limit behavior, and counting curves subject to prescribed self-intersection bounds. The same framework may also be useful for joint statistics involving word length, geometric length, and intersection number.

\end{itemize}

Overall, these applications indicate that the normal forms developed in this paper provide not only explicit solutions to several problems on powers and conjugacy, but also a useful combinatorial framework for further quantitative and asymptotic questions concerning surface groups.

\subsection{Organization of the paper}
This paper is organized as follows. In Section \ref{sect D-reduced and S-reduced}, we first introduce some basic definitions and notations, especially the Gr\"{o}bner-Shirshov basis $D$ and the rewriting system $S$, and then prove that they are equivalent. In Section \ref{sect cyclically irreducible words}, to prepare for the subsequent sections, we demonstrate some cyclically irreducible words. In Section \ref{sect proof of main thms}, we prove the key Theorem \ref{thm normal forms of power elements (original main thm1)} and Theorem \ref{main thm word length formulae}. Considering the extensive discussions in the proof of Proposition \ref{prop classification of LLFR of X^2}, we first introduce the ideas and procedures of our proof in Section \ref{sect. proof of key Proposition}, and then place the detailed proof in Appendix \ref{sect appendix for some detailed proofs of props}, which is through a lengthy process of case-by-case analysis. In Section \ref{sect classification of reductions}, we give a complete classification of reductions. For the sake of completeness and readability, we list them in Appendix \ref{sect appendix tables for reductions}. In Section \ref{sect normal form of conjugacy classes}, we mainly discuss the normal forms of conjugacy classes, present several algorithms and prove Corollary \ref{cor conj-power algo.}, and analyze the differences between our conjugation algorithm and Dehn's conjugation algorithm. Then in Section \ref{sect from sym. pre. to any pre.}, we generalize our formulae from the symmetric presentation to any minimal geometric presentation and prove Theorem \ref{main thm4 word length formulae more rough}. Finally, in Section \ref{sect remark}, we show an application on calculating some growths in a surface group.

\section{Preliminaries}\label{sect D-reduced and S-reduced}
Throughout this paper, unless it is specifically stated otherwise, $G:=\pi_1(\Sigma_g) (g\geq 2)$ is a surface group with the following symmetric presentation 
\begin{equation*}
		G=\left\langle c_1,\dots, c_{2g}\mid c_1\cdots c_{2g}c_1^{-1}\cdots c_{2g}^{-1}\right\rangle.
	\end{equation*}	 
We first introduce some notations:

\begin{longtable}{ll}
%
\hline
		Notation & Explanation \\
		\hline
		$\overline{abc\cdots xyz},\overline{X}$ & words \\
		$\W(\gs)$ & set of all words in the alphabet $\gs^{\pm}$ \\
		$\overline{W'}\subset (\not\subset)\overline{W}$ &$\overline{W'}$ is (not) a subword of the word $\overline{W}$ \\
		$b_i\in(\notin) \overline{W}$ & $\overline{W}$ (does not) contains the letter $b_i$\\
		$|\overline{X}|_\gs,|X|_\gs$ & word lengths of $\overline{X}$ and $X$\\
        $\swl_\gs(x), \swl(x)$ &  stable word length / translation number of $x$ with respect to $\gs$\\
		$\nf(x), \nf([x]), \nf(W)$ & normal forms of the element $x$, the conjugacy class $[x]$ and the word $\overline{W}$\\
        $\ciw(x), \ciw(X)$ &  cyclically irreducible cores of $x$ and $X$\\
        $\rsp(\overline{W_1},\overline{W_2})$ & reducing-subword pair of $\overline{W_1}$ and $\overline{W_2}$\\
		$\RR$ & set of all cyclic permutations of the relator $\overline{c_1\cdots c_{2g}c_1^{-1}\cdots c_{2g}^{-1}}$ and its inverse \\
		$\overline{W_1}\xrightarrow{D_{(i)}, ~S_{(i)}}\overline{W_2}$ &
        $\overline{W_1}$ is reduced to $\overline{W_2}$ by $D_{(i)}, ~S_{(i)}$\\
		$\llfr$ & locally longest fractional relator \\
		\hline
\end{longtable}

Now we give basic definitions and then present our main tool: the Gr\"{o}bner-Shirshov basis. 

\subsection{Basic definitions}
Let $\gs=\{c_1,\dots,c_{2g} \}$ be a generating set, and $\gs^\pm=\{c_1^\pm,\dots,c_{2g}^\pm \}$. Then each element of $G$ can be represented by a finite word $x_1\cdots x_n$, where $x_i\in \gs^{\pm}$. To avoid ambiguity, we overline every word $\overline{x_1\cdots x_n}$. For convenience, if $\overline{W},\overline{y_1\cdots y_n}$ are words, we always use the general ones $W,y_1\cdots y_n$ to denote the corresponding elements in $G$. We denote the set of all the words in the alphabet $\gs^{\pm}$ by
	$$\W(\gs):=\{\overline{x_1\cdots x_{n}}\mid x_i\in \gs^{\pm},n\geq 0\}.$$
Especially, if $n=0$, we call it an empty word and denote it by $1$. If $\overline{W'}$ is a subword of $\overline{W}$, we denote this by $\overline{W'}\subset\overline{W}$. Similarly, if a letter $b_i\in \gs^{\pm}$ appears in $\overline{W}$, we denote this by $b_i\in \overline{W}$.

  For a better explanation, in a word, we always use lowercase letters to represent letters in the alphabet $\gs^\pm$ and use uppercase letters to represent subwords. For instance, in the word $\overline{x_1x_2W_1y_1Vzs}$, we have $x_1,x_2,y_1,z,s\in\gs^\pm$ and $\overline{W_1},\overline{V}\in\W(\gs)$. 

For a word $\overline{W}=\overline{x_1\cdots x_n}\in \W(\gs)$, we denote the length of $\overline{W}$ by $|\overline{W}|_\gs=n$. The \emph{(word) length} of an element $x\in G$ with respect to the generating set $\gs$ is defined as
	$$|x|_\gs:=\min\{n\mid x=x_1\cdots x_n\in G, ~x_i\in \gs^\pm\}.$$
When the generating set is clear, we always omit the subscript $\cdot_\gs$.

Note that an element may be represented  by several different words with the same minimal length. For example, in the surface group $G$ with the symmetric presentation (\ref{symmetric presentation}), two distinct words $\overline{c_1\cdots c_{2g}}\neq \overline{c_{2g} \cdots c_1}\in\W(\gs)$ have the same minimal length but represent the same element $$x=c_1\cdots c_{2g}=c_{2g} \cdots c_1\in G.$$ Now, we introduce a total order $\succ$ on $ \W(\gs) $, which help us get the unique minimal presentation for every element in $G$. 

\begin{defn}[Length-lexicographical Order]\label{def order in word set}
Given an order $\succ$ on the alphabet $\gs^{\pm}$ as follows:
\begin{equation*}
		c_{2g}^{-1}\succ \cdots \succ c_2^{-1}\succ c_1^{-1}\succ c_1\succ c_2\cdots\succ  c_{2g}.
	\end{equation*}
Then a \emph{length-lexicographical order} on $\W(\gs)$ and $\W(\gs)\times \W(\gs)$ is defined as follows:
\begin{equation*}
\begin{array}{ccccl}
 \overline{X} &\succ& \overline{Y} &\Longleftrightarrow&\left\{\begin{array}{ll}|\overline{X}|>|\overline{Y}|,~\mathrm{or}~  \\ |\overline{X}|=|\overline{Y}| ~\mathrm{with}~ 
 \overline{X}=\overline{AxB}, ~\overline{Y}=\overline{AyC}
~\mathrm{for}~  x \succ y\in \gs^{\pm}\end{array}\right.,\\[10pt]
(\overline{X_1},\overline{X_2})&\succ& (\overline{Y_1},\overline{Y_2})&\Longleftrightarrow& \overline{X_1}\succ \overline{Y_1}, \mbox{or} ~\overline{X_1}=\overline{Y_1}~\mbox{with}~\overline{X_2}\succ\overline{Y_2}.\end{array}
\end{equation*}
Moreover, the related notations ``$\prec, ~\preceq, ~\succeq$'' are defined conventionally. 
\end{defn}

Using the length-lexicographical order, we have:

\begin{defn}[Normal Form and Irreducible Word]\label{normal form}
Let $x\in G$. The \emph{normal form} $\nf(x)$ of the element $x$ is defined to be the minimal word representing $x$:
$$\nf(x):=\min\{\overline{X}\in \W(\gs)\mid X=x\in G\}.$$
A word $\overline{X}\in \W(\gs)$ is called \emph{irreducible} if $\overline{X}= \nf(X)$. Otherwise, $\overline{X}$ is called \emph{reducible}. Furthermore, $\overline{X}=\overline{x_1\cdots x_n}~(x_i\in\gs^\pm)$ is called \emph{cyclically irreducible} if all of its cyclic permutations $\overline{x_i\cdots x_nx_1\cdots x_{i-1}}$ ($1\leq i\leq n$) are irreducible. The \emph{normal form $\nf([x])$ of the conjugacy class $[x]$} is defined to be the minimal word in the conjugacy class $[x]$ of $x$:
    $$\nf([x]):=\min\{\overline{X}\in\W(\gs)\mid X ~\mbox{is conjugate to} ~x\in G\}.$$

\end{defn}

\subsection{Gr\"{o}bner-Shirshov basis}

Since we have defined the normal form of each element in $G$, we need a criterion to decide whether a word is a normal form of an element. Fortunately, we have: 

\begin{thm}[Gu-Zhao, \cite{GZ22}]\label{thm GZ}
	For the surface group  $\pi_1(\Sigma_g)$ with the symmetric presentation (\ref{symmetric presentation}) and the length-lexicographical order (\ref{order of generators}), there is a Gr\"{o}bner-Shirshov basis $D$, consisting of the following:
	\begin{enumerate}
		\item $D_{(1,j,t)}=\overline{c_j(c_{j-1}\cdots c_1c_{2g}^{-1}\cdots c_{j+1}^{-1})^tc_{j}^{-1}} - \overline{(c_{j+1}^{-1}\cdots c_{2g}^{-1}c_1\cdots c_{j-1})^t}\mbox{ for }j=2,\dots,2g\mbox{ and }t=1,2,\dots$;
		\item $ D_{(2,j,t)}=\overline{c_j(c_{j+1}\cdots c_{2g}c_{1}^{-1}\cdots c_{j-1}^{-1})^tc_{j}^{-1}} - \overline{(c_{j-1}^{-1}\cdots c_{1}^{-1}c_{2g}\cdots c_{j+1})^t}\mbox{ for }j=2,\dots,2g\mbox{ and }t=1,2,\dots$;
		\item $D_{(3)}=\overline{c_{2g}^{-1}\cdots c_{1}^{-1}}-\overline{c_1^{-1}\cdots c_{2g}^{-1}}$; 
		\item $D_{(4)}=\overline{c_{1}\cdots c_{2g}}-\overline{c_{2g}\cdots c_{1}}$; 
		\item $D_{(5,i)}=\overline{c_i^{-1}\cdots c_{2g}^{-1}c_1\cdots c_{i-1}}-\overline{c_{i-1}\cdots c_{1}c_{2g}^{-1}\cdots c_i^{-1}}\mbox{ for }i=2,\dots,2g$;
		\item $D_{(6,i)}=\overline{c_{i-1}^{-1}\cdots c_{1}^{-1}c_{2g}\cdots c_{i}}-\overline{c_{i}\cdots c_{2g}c_{1}^{-1}\cdots c_{i-1}^{-1}}\mbox{ for }i=2,\dots,2g$;
		\item $D_{(7,i)}=\overline{c_{i}^{-1}c_i}-1\mbox{ for }i=1,\dots,2g$;
		\item $D_{(8,i)}=\overline{c_ic_{i}^{-1}}-1\mbox{ for }i=1,\dots,2g$.
	\end{enumerate}
\end{thm}

 To make clear what this theorem does, please refer to \cite[Section 3]{GZ22} and \cite[Section 2]{BV06}. For more details on Gr\"{o}bner-Shirshov basis, see \cite{Sh99,Uf98}.

Here we can interpret the $D$-basis as a rewriting system. Every leading word of a polynomial in $D$-basis is called a \emph{forbidden word}. A word that does not contain any forbidden subword is called \emph{$D$-reduced}. Given a word $\overline{X}\in \W(\gs)$, we can always obtain the normal form $\nf(X)$ of the element $X$ by rewriting $\overline{X}$: replacing its subword of the form as the leading word of $D_{(i)}$ with the second word of $D_{(i)}$ until it becomes irreducible. More precisely, if $\overline{ABC}$ is a reducible word, where $\overline{B}$ is the leading term of $D_{(i)}=\overline{B}-\overline{E}$, then we have $$\overline{ABC}\xrightarrow{\mbox{replace }\overline{B}\mbox{ with }\overline{E}}\overline{AEC},$$ that is, $D_{(i)}$ reduces $\overline{ABC}$ to $\overline{AEC}$. Denote this process as 
$$\overline{ABC}\xrightarrow{D_{(i)}}\overline{AEC}.$$
The procedure of reducing by $D$ is a descending order process, and $\nf(x)$ is the unique minimal word representing $x\in G$. 

\subsection{\texorpdfstring{$S$-reduction}{S-reduction}} 
In this subsection, we shall define several kinds of reducible subwords and the corresponding operations based on Theorem \ref{thm GZ}, which facilitates the discussions in the subsequent sections. First, we need a definition of \emph{fractional relator}.

\begin{defn}[Fractional Relator and $\llfr$]\label{fractional relator}
For the symmetric presentation (\ref{symmetric presentation}), we denote the set of all cyclic permutations of the relator $c_1\cdots c_{2g}c_{1}^{-1}\cdots c_{2g}^{-1}$ and its inverse by
$$\RR:=\left\{\overline{d_{i}\cdots d_{4g}d_1\cdots d_{i-1}}, ~\overline{d_{i}\cdots d_{1}d_{4g}\cdots d_{i+1}}\mid i=1,\dots,4g\right\},$$
where $\overline{d_1\cdots d_{4g}}=\overline{c_1\cdots c_{2g}c_{1}^{-1}\cdots c_{2g}^{-1}}$.
For any $2\leq k\leq 4g$, a word $\overline{b_1\cdots b_{k}}$  is called a $k$-\emph{fractional relator} if it is a subword of some $\overline R\in\RR$. Moreover,  if  a $k$-fractional relator $\overline{b_1\cdots b_{k}}$ is a subword of $$\overline{W}=\overline{w_1\cdots w_{m}b_1\cdots b_{k}w'_1\cdots w'_n}\in\W(\gs),$$  such that  neither $\overline{ w_{m}b_1\cdots b_{k}}$ nor $\overline{b_1\cdots b_{k}w'_1}$ is a $(k+1)$-fractional relator, then we call  $\overline{b_1\cdots b_k}$ a \emph{locally longest fractional relator} (\emph{LLFR}) of $\overline{W}$. Note that a $4g$-fractional relator is always an $\llfr$.
\end{defn}
We now recombine the Gr\"{o}bner-Shirshov basis $D$ and obtain four sets of polynomials. It is worth noting that the union of these four sets is not a basis but a generating set. 

\begin{defn}[$S$-reduced]\label{def S-reduced}
	Let $G,\RR$ be defined as before, and let 
	\begin{enumerate}
		\item $S_{(1)}=\{\overline{b_1b_1^{-1}}-1\mid \overline{b_1\cdots b_{4g}}\in \RR\}$;
		\item $S_{(2,k)}=\{\overline{b_1\cdots b_{k}}-\overline{b^{-1}_{4g}\cdots b^{-1}_{k+1}}\mid \overline{b_1\cdots b_{4g}}\in \RR\}, ~2g+1\leq k\leq 4g$; 
		\item $S_{(3,t)}=\{\overline{b_1(b_2\cdots b_{2g})^tb_{2g+1}}-\overline{(b_{2g}\cdots b_{2})^t}\mid \overline{b_1\cdots b_{4g}}\in \RR\}, t\geq 2$;
		\item $S_{(4,t)}=\{\overline{b_1(b_2\cdots b_{2g})^t}-\overline{(b_{2g}\cdots b_{2})^tb_1}; ~\overline{(b_1\cdots b_{2g-1})^tb_{2g}}-\overline{b_{2g}(b_{2g-1}\cdots b_{1})^t}\mid \overline{b_1\cdots b_{4g}}\in \RR, b_1\succ b_{2g}\},t\geq 1$.
	\end{enumerate}
We will denote $S_{(i,k)}$ by $S_{(i)}$ ($i=2,3,4$) if there is no need to emphasize the second subscript $k$.
We call a word $\overline{V}$\emph{ of type} $S_{(i)}$ if it is the leading term of a polynomial in the set $S_{(i)}$. Similarly, we can also interpret this $S$-set (i.e. $\cup S_{(i)}$) as a rewriting system to define \emph{$S$-reduced} and $S_{(i)}$ reducing $\overline{W_1}$ to $\overline{W_2}$:
$$\overline{W_1}\xrightarrow{S_{(i)}}\overline{W_2}.$$
\end{defn}

Recall that a word $\overline{X}=\overline{x_1x_2\cdots x_n}\in\W(\gs)$ is called \emph{freely reduced} if $x_i\neq x^{-1}_{i+1}$ for all $i=1,\ldots, n-1$, and is called \emph{cyclically freely reduced} if in addition $x_1\neq x_n^{-1}$. Therefore, we have:

\begin{lem}
A word is freely reduced if and only if it contains no subword of type $S_{(1)}$.
\end{lem}

Furthermore,  we shall prove that $S$-set can also be used to determine whether a word is $D$-reduced or not, i.e., $S$-reduced is equivalent to $D$-reduced. However, we shall see that $S$-set is not minimal and $D$-basis is actually a proper subset of $S$-set.

\begin{lem}\label{lem four operations}
Let $\overline{W}\in \W(\gs)$ be a word in the presentation (\ref{symmetric presentation}) of  a surface group $G$. Then the following are equivalent:
\begin{enumerate}
    \item $\overline{W}$ is the normal form of the element $W$, i.e., $\overline{W}$ is irreducible;
    \item $\overline{W}$ is $D$-reduced;
    \item $\overline{W}$ does not contain a subword of types $S_{(1)}, S_{(2,2g+1)}, S_{(3,t)}(t\geq 2)\mbox{ or }S_{(4,1)}$;
    \item $\overline{W}$ is $S$-reduced.
\end{enumerate}
\end{lem}

\begin{proof}
(1) $\Leftrightarrow$ (2), (2) $\Rightarrow$ (3) and (3) $\Rightarrow$ (4) hold according to \cite[Theorem 7]{Uf98} and Definition \ref{def S-reduced}. We only need to prove  (4) $\Rightarrow$ (2), that is, if $\overline{W}$ does not contain a subword of types $S_{(i)}$~($1\leq i\leq 4$), then it does not contain a subword of type $D_{(k)}$ ( $1\leq k\leq 8$).
The proof is as follows.

Suppose $\overline{W}$ contains a subword $\overline{W'}$ of type $D_{(k)}$ for some $1\leq k\leq 8$. We will prove $\overline{W'}$ is of type $S_{(i)}$ for some $1\leq i\leq 4$, see Table \ref{tab: the correspondence}.
\begin{enumerate}
	\item[$1^\circ$] If $\overline{W'}$ is of type $D_{(1,j,t)}$, then $$\overline{W'}=\overline{c_j(c_{j-1}\cdots c_1c_{2g}^{-1}\cdots c_{j+1}^{-1})^tc_{j}^{-1}}~(j=2,3,\dots,2g;~t=1,2,\dots).$$
	Note that $\overline{c_j\cdots c_1c_{2g}^{-1}\cdots c_1^{-1}c_{2g}\cdots c_{j+1}}\in \RR$ ,  $c_j \prec c_{j-1}^{-1}$ and $c_{j-1}\prec c_{j}^{-1}$.  So $\overline{W'}$ is of type $S_{(2,2g+1)}$ (and hence of type $S_{(2)}$) if $t=1$, and $\overline{W'}$ is of type $S_{(3,t)}$ (and hence of type $S_{(3)}$)  if $t\geq 2$. If $\overline{W'}$ is of type $D_{(2,j,t)}$, then the argument is similar to the case $D_{(1,j,t)}$'s.
	 \item[$2^\circ$] If $\overline{W'}$ is of type $D_{(3)}$, then $$\overline{W'}=\overline{c_{2g}^{-1}\cdots c_{1}^{-1}}.$$ Since $\overline{c_{2g}^{-1}\cdots c_{1}^{-1}c_{2g}\cdots c_1}\in \RR$ and $c_{2g}^{-1}\succ c_{1}^{-1}$, we have that $\overline{W'}$ is of type $S_{(4,1)}$ (and thus of type $S_{(4)}$). If $\overline{W'}$ is of type $D_{(4)},D_{(5,i)}\mbox{ or }D_{(6,i)}$, then the argument is similar to the case $D_{(3)}$'s.
	 \item[$3^\circ$] If $\overline{W'}$ is of type $D_{(7,i)}$ or  $D_{(8,i)}$, then $\overline{W'}$ is of type $S_{(1)}$ clearly.
\end{enumerate}

\begin{table}[ht]
    \begin{center}
    \begin{tabular}{|c|c|c|c|c|c|c|c|c|}
    \hline
         & $D_{(1)}$  & $D_{(2)}$ & $D_{(3)}$ & $D_{(4)}$ & $D_{(5)}$ & $D_{(6)}$ & $D_{(7)}$ & $D_{(8)}$\\
         \hline
        $S_{(1)}$ &  &  &  &  &  &  & $D_{(7)}$ & $D_{(8)}$\\
         \hline
        $S_{(2)}$ & $D_{(1,j,1)}$ & $D_{(2,j,1)}$ &  &  &  &  &  & \\
         \hline
        $S_{(3)}$ & $D_{(1,j,t)},t\geq 2$ & $D_{(2,j,t)},t\geq 2$ &  &  &  &  &  & \\
         \hline
        $S_{(4)}$ &  &  &    $D_{(3)}$ & $D_{(4)}$ & $D_{(5)}$ & $D_{(6)}$ & & \\
         \hline
    \end{tabular}
\end{center}
\caption{The correspondence between $S$-type and $D$-type.}\label{tab: the correspondence}
\end{table}
In conclusion, we have finished the proof of ``(4) $\Rightarrow$ (2)'' and hence Lemma \ref{lem four operations} holds. 
\end{proof}

Since ``$D$-reduced'', ``$S$-reduced'' and ``irreducible'' are equivalent, ``cyclically $D$-reduced'', ``cyclically $S$-reduced'' and ``cyclically irreducible'' are also equivalent. We will always say ``irreducible'' and ``cyclically irreducible'' in the subsequent sections. Moreover, Gu and Zhao  \cite[Corollary 4.2]{GZ22} gave an equivalent condition of cyclically irreducible words as follows.

\begin{prop}[Gu-Zhao, \cite{GZ22}]\label{prop D cyclically irreducible= DD irreducible}
    A word $\overline{D}\in\W(\gs)$ is cyclically irreducible if and only if $\overline{D^2}$ is irreducible.
\end{prop}


\begin{rem}\label{rem of maximality}
Note that there are some cases, for instance, $\overline{ABC}$ is of type $S_{(i_1)}$ and $\overline{B}$ is of type $S_{(i_2)}$. Then, there is a conflict that we cannot do $S_{(i_2)}$ if we do $S_{(i_1)}$ first or vice versa. Then, we need the following precedence of reducing operations: 
$$S_{(1)}>S_{(2,k+1)}>S_{(2,k)}>S_{(3,t)}>S_{(4,t+1)}>S_{(4,t)}.$$
This precedence comes from the idea ``using less operations to get a normal form''. If two or more operations are equivalent to one operation, we choose to do the one directly. However, the precedence is merely a general setting for the discussions in the subsequent sections but not necessary, since there are some situations where it cannot be used to decide.	
\begin{exam}
In the surface group $\pi_1(\Sigma_3)$ with the symmetric presentation (\ref{symmetric presentation}), let $\overline{b_1\cdots b_{12}}\in \RR$. The word $\overline{b_1\cdots b_{8}}$ has two ways to be irreducible:
	\begin{eqnarray*}
		\overline{b_1\cdots b_8}&\xrightarrow{S_{(2,8)}}& \overline{b_{12}^{-1}b_{11}^{-1}b_{10}^{-1}b_9^{-1}}  ~~~\mbox{ or } \\
		\overline{b_1\cdots b_8}&\xrightarrow{S_{(2,7)}}&\overline{b_{12}^{-1}b_{11}^{-1}b_{10}^{-1}b_9^{-1}b_{8}^{-1}b_{8}}\xrightarrow{S_{(1)}}\overline{b_{12}^{-1}b_{11}^{-1}b_{10}^{-1}b_9^{-1}}.
	\end{eqnarray*}
	Then we choose $S_{(2,8)}$ directly. Furthermore, if $b_1\succ b_6$, then the word $\overline{b_1(b_2\cdots b_{6})^2}$ has two ways to be irreducible:
	\begin{eqnarray}\label{rem of max S4}
		\overline{b_1(b_2\cdots b_6)^2}&\xrightarrow{S_{(4,2)}}& \overline{(b_6\cdots b_2)^2b_1}  ~~~\mbox{ or } \notag\\
		\overline{b_1(b_2\cdots b_6)^2}&\xrightarrow{S_{(4,1)}}&\overline{b_6\cdots b_1 (b_2\cdots b_6)}\xrightarrow{S_{(4,1)}}\overline{(b_6\cdots b_2)^2b_1}.
	\end{eqnarray}
We choose $S_{(4,2)}$ directly. Besides, a word of type $S_{(2,4g)}$ can be irreducible by one $S_{(2,k)}$ and $ (4g-k) $ $S_{(1)}$. In this case, we choose to do $S_{(2,4g)}$ directly. 

Therefore, whenever the $S_{(2)}$ or $S_{(4)}$-reduction is used, we always use the \emph{maximal} reducing operation. Namely, we take the longest subword $\overline{W'}$ of type $S_{(2,k)}$ or $S_{(4,t)}$ with the maximal $k$ or $t$ (i.e, there is no subword $\overline{W''}$ of types $S_{(2,k+1)},S_{(4,t+1)}$ such that $\overline{W'}\subset \overline{W''} \subset \overline{W}$). 
    \end{exam}
\end{rem}




\subsection{Some facts on reduction }
 Before the main argument, we need some easy but frequently used observations. 
\begin{lem}\label{lem some observations}
Let 
$\overline{b_1\cdots b_{4g}}\in \RR$. We have the following facts (with subscript$\mod 4g$).
	\begin{enumerate}
		\item $b_{k}=b_{k\pm2g}^{-1}$ for any $k=1,2,\dots,4g$.
		\item $b_{k}b_{k+1}\cdots b_{k+2g-1}=b_{k+2g-1}\cdots b_{k+1}b_{k}$ for any $k=1,2,\dots,4g$.
		\item $b_{m}\prec b_{n}$ if and only if $b_{m}^{-1}\succ b_{n}^{-1}$ for any $ ~m, n=1,2,\dots,4g$.
		\item If $\overline{b_1\cdots b_{k}}$ is an $\llfr$ of a word $\overline{W}$ and $2\leq k\leq 2g-2$, then any subword of $\overline{W}$ containing $\overline{b_1\cdots b_k}$ cannot be of types $S_{(i)}~(i=1,2,3,4)$.
        \item  $b_1(b_2\cdots b_{2g})^tb_{2g+1}=(b_{2g}\cdots b_2)^t$ for any $t\geq 0$, and hence the word $$\overline{b_1(b_2\cdots b_{2g})^tb_{2g+1}}$$ is always reducible.
		
        \item Let $\{\overline{X_i}=\overline{x_{i_1}\cdots x_{i_{n_i}}}\in \W(\gs)\mid n_i\geq 1,i=1,\dots,m\}$ be a finite set of irreducible words. If the product $\overline{X_1\cdots X_m}$ is reducible, then every reducible subword of it must contain some $\overline{x_{i_{n_i}}x_{(i+1)_1}}$.

        \item Let $\overline{U}=\overline{u_1\cdots u_n}$ be of type $S_{(3, t)}(t\geq 2)$. If $u_i=u_j\in\overline{U}$ for some $1< i<j<n$, then
        $$\overline{U'}:=\overline{u_1\cdots u_{i-1}u_j\cdots u_n}$$ is of type $S_{(2,2g+1)}$ or $S_{(3, t')}(2\leq t'<t)$, and hence is reducible.

\end{enumerate}
\end{lem}
\begin{proof}
Since items (1)--(5) are obvious, we only give the proofs of items (6) and (7) here.
  
 (6) If $\overline{W}\subset\overline{X_1\cdots X_m}$ does not contain any $\overline{x_{i_{n_i}}x_{(i+1)_1}}$, then $\overline{W}$ is a subword of some irreducible word $\overline{X_i}$ and hence it is again irreducible.

 (7) Since $\overline{U}$ is of type $S_{(3, t)}(t\geq 2)$, we can suppose $\overline{U}=\overline{b_1(b_2\cdots b_{2g})^t b_{2g+1}}.$
There are exactly $(2g-1)$ letters $b_2,\dots, b_{2g}$ appearing $t\geq 2$ times in $\overline{U}$. Note that $\overline{U'}$ can be obtained by the following process:
\begin{eqnarray*}
    \overline{U}&\xlongequal{\qquad\quad}&\overline{b_1(b_2\cdots b_{2g})^{t_1}b_2\cdots b_{i-1}\underbrace{b_{i}\cdots b_{2g}(b_2\cdots b_{2g})^{t_2}b_2\cdots b_{i-1}}_{A}b_i\cdots b_{2g}(b_2\cdots b_{2g})^{t_3}b_{2g+1}}\\
    &\xrightarrow{\mbox{delete }\overline{A}}&\overline{b_1(b_2\cdots b_{2g})^{t_1}b_2\cdots b_{i-1}b_{i}\cdots b_{2g}(b_2\cdots b_{2g})^{t_3}b_{2g+1}}=:\overline{U'}.
\end{eqnarray*}
 Hence, $\overline{U'}$ is of type $S_{(2,2g+1)}$ when $t_1+t_3=0$ and of type $S_{(3)}$ when $t_1+t_3\geq 1$.
\end{proof}

The following facts are frequently used in reduction discussions in Section \ref{sect. proof of key Proposition} and Appendix \ref{sect appendix for some detailed proofs of props}.     

\begin{lem}\label{lem frequently used}
Let $\overline{X}=\overline{x_1\cdots x_{n}},\overline{Y}=\overline{y_1\cdots y_m}\in \W(\gs)$ be both irreducible and $\overline{b_1\cdots b_{4g}}\in\RR$.
   \begin{enumerate}
\item Suppose 
        $$\overline{x_{n-r+1}\cdots x_{n}y_1\cdots y_s}=\overline{b_1\cdots b_{r+s}}$$ containing $\overline{x_ny_1}$ is an $\llfr$ of $\overline{XY}$ with length $r+s\in \{2g+1,\dots,4g-2\}$. Then $$\overline{XY}\xrightarrow{S_{(2,r+s)}}\overline{x_1\cdots x_{n-r}b_{2g}\cdots b_{r+s-2g+1}y_{s+1}\cdots y_{m}}=:\overline{Z}.$$
Moreover, if $\overline{Z}$ is reducible, then every potential reducible subword of $\overline{Z}$ of type $S_{(i)}$ must be of the form $\overline{x_{r'}\cdots x_{n-r}b_{2g}}$ or $\overline{b_{r+s-2g+1}y_{s+1}\cdots y_{s+s'}}$.

		\item Suppose there is a reducible subword of $\overline{XY}$ of type $S_{(3,t)}~(t\geq 2)$. Then $$\overline{XY}\xrightarrow{S_{(3,t)}}\overline{x_1\cdots x_{n-r}(b_{2g}\cdots b_2)^ty_{s+1}\cdots y_m}=:\overline{Z}.$$ Moreover, if $\overline{Z}$ is reducible, then every potential reducible subword of $\overline{Z}$ of type $S_{(i)}$ must be of the forms $\overline{x_{r'}\cdots x_{n-r}b_{2g}}$ or $\overline{b_{2}y_{s+1}\cdots y_{s+s'}}$.

        \item Suppose there is a reducible subword of $\overline{XY}$ of type $S_{(4,t)}~(t\geq 1)$. Then $$\overline{XY}\xrightarrow{S_{(4,t)}}\overline{x_1\cdots x_{n-r}(b_{2g}\cdots b_2)^tb_1y_{s+1}\cdots y_m}=:\overline{Z}.$$ Moreover, if $\overline{Z}$ is reducible, then every potential reducible subword of $\overline{Z}$ of type $S_{(i)}$ must be of the forms $\overline{x_{r'}\cdots x_{n-r}b_{2g}}$ or $\overline{b_{1}y_{s+1}\cdots y_{s+s'}}$.
	\end{enumerate}
\end{lem}
\begin{proof}
Since the proofs of the three items are similar, we only provide the detailed proof of item (1) in the following.
Suppose $\overline{V}$ is a reducible subword of $\overline{Z}$ of type $S_{(i)}$. Since $$\overline{x_{n-r+1}\cdots x_{n}y_1\cdots y_s}=\overline{b_1\cdots b_{r+s}}$$ is an $\llfr$ of $\overline{XY}$, we have $$x_{n-r}\neq b_{4g},b_{2g+1}; \quad y_{s+1}\neq b_{r+s+1},b_{r+s-2g}.$$ 
Hence, $\overline{Z}$ is freely reduced and $\overline{V}$ is not of type  $S_{(1)}$ and thus is of $S_{(i)}~(i=2,3,4)$. Then every $\llfr$ in $\overline{V}$ has length $\geq 2g+1$ (resp. $2g-1$) if $\overline{V}$ is of type $S_{(2)}$ (resp. $S_{(i)}(i=3,4)$). By Lemma \ref{lem some observations}(6), we obtain that $\overline{V}$ must contain $\overline{x_{n-r}b_{2g}}$ or $\overline{b_{r+s-2g+1}y_{s+1}}$. 

If $\overline{V}\supset\overline{x_{n-r}b_{2g}}$, note that  $\overline{b_{2g}\cdots b_{r+s-2g+1}}$ with length $4g-(r+s)\in\{2,\dots,2g-1\}$ is an $\llfr$ of $\overline{Z}$, then there exists $1\leq r'\leq n-r$ such that $\overline{V}$  is the following three possible cases: 
    
Case (a) $\overline{V}=\overline{x_{r'}\cdots x_{n-r}b_{2g}}$;

Case (b) $\overline{V}=\overline{x_{r'}\cdots x_{n-r}b_{2g}\cdots b_{2g-k+1}} (2\leq k\leq 4g-r-s\leq 2g-1) $;

Case (c) $\overline{V}=\overline{x_{r'}\cdots x_{n-r}b_{2g}\cdots b_{r+s-2g+1}y_{s+1}\cdots y_{s+l}}$.

Indeed, both Case (b) and Case (c) are impossible:

When $2g+2\leq r+s\leq 4g-2$,  $\overline{V}$ contains an $\llfr$ $\overline{b_{2g}\cdots b_{2g-p+1}}$ of length $p\leq 2g-2$, contradicting that $\overline{V}$ is of type $S_{(i)}$. When $r+s=2g+1$, $\overline{V}$ contains an $\llfr$ with length $\leq 2g-1$ and hence $\overline{V}$ cannot be of types $S_{(2)}$ or $S_{(4,1)}$. Furthermore, if $\overline{V}$ is of type $S_{(3,t)}$ or $ S_{(4,t)}~(t\geq2)$, then we always have
	$$\overline{x_{r'}\cdots x_{n-r}}=\overline{b_{2g+1}(b_{2g}\cdots b_2)^{t'}}(t'\geq 1),$$ 
and hence
$\overline{x_{r'}\cdots x_{n-r+1}}=\overline{b_{2g+1}(b_{2g}\cdots b_2)^{t'}b_1}$ is a reducible subword of $\overline{X}$. It contradicts the assumption that $\overline{X}$ is irreducible. Therefore, only Case (a) is possible. 

	If $\overline{V}\supset\overline{b_{r+s-2g+1}y_{s+1}}$, by the same discussion as above, we can get $\overline{V}=\overline{b_{r+s-2g+1}y_{s+1}\cdots y_{s+s'}}$.
    
Therefore,  we conclude that all the potential reducible subwords of $\overline{Z}$ of type $S_{(i)}$ must be of the form $\overline{x_{r'}\cdots x_{n-r}b_{2g}}$ or $\overline{b_{r+s-2g+1}y_{s+1}\cdots y_{s+s'}}$.
\end{proof}

\section{Cyclically irreducible words}\label{sect cyclically irreducible words}
In this section, we will provide some sufficient conditions for words to be cyclically irreducible. 

\subsection{Lemmas related to normal forms}

First, we have the following lemma, which helps us obtain normal forms of higher powers in Section \ref{sect proof of main thms}. 

\begin{lem}\label{lem from 2-nd power to n-th power}
If $\overline{ADE},\overline{D^2}\in\W(\gs)$ are irreducible, then $\overline{AD^kE}$ is always irreducible for every $k\geq1$.
\end{lem}
\begin{proof}
The case when $\overline{D}=1$ is trivial. So we let 
\begin{eqnarray*}
\overline{D}&=&\overline{d_1\cdots d_n}~(n\geq 1),\\
\overline{A}&=&\overline{a_1\cdots a_m}~(m\geq 0),\\
\overline{E}&=&\overline{e_1\cdots e_l}~(l\geq 0),
\end{eqnarray*}
where $\overline{A}$ and $ \overline{E}$ are possibly empty. When $k=2$,
$$\overline{AD^kE}=\overline{\underbrace{a_1\cdots a_m}_A\underbrace{d_1\cdots d_n}_D\underbrace{d_1\cdots d_n}_D\underbrace{e_1\cdots e_l}_E}.$$
Since $\overline{ADE}$ and $\overline{D^2}$ are irreducible, $\overline{AD^2E}$ is freely reduced, i.e., it does not contain a subword of type $S_{(1)}$. By Lemma \ref{lem four operations}, to prove $\overline{AD^2E}$ is irreducible, it suffices to show that there is no subword of $\overline{AD^2E}$ of types $S_{(2,2g+1)},S_{(3,t)}(t\geq 2)$ or $S_{(4,1)}$. 

Now we assume there is such a reducible subword $\overline{V}\subset \overline{AD^2E}$. Note that $\overline{AD},\overline{D^2},\overline{DE}$ are all irreducible. Then by Lemma \ref{lem some observations}(6), we have three possible positions of $\overline{V}$ in $\overline{AD^2E}$.

Case (1). $\overline{V}$ starts in $\overline{A}$ and ends in the right $\overline{D}$;

Case (2). $\overline{V}$ starts in $\overline{A}$ and ends in $\overline{E}$;

Case (3). $\overline{V}$ starts in the left $\overline{D}$ and ends in $\overline{E}$.

Notice that Case (1) and Case (3) are symmetric. Thus, we only need to consider Case (1) and Case (2). Then  $\overline{V}$ has a form
$$\overline{V}=\overline{A'd_1\cdots d_n d_1X},$$
where $1\neq\overline{A'}\subset \overline{A}$ and $|\overline{X}|\geq 0$. Since a word of types $S_{(2,2g+1)}$ or $S_{(4,1)}$ does not allow a letter to appear more than once, we obtain that $\overline{V}$ is of type $S_{(3,t)}(t\geq 2)$. Then, by Lemma \ref{lem some observations}(7), $\overline{A'd_1X}$
is a reducible subword of the irreducible word $\overline{ADE}$, which is contradictory.

In conclusion, $\overline{AD^2E}$ contains no subword of types $S_{(1)},S_{(2,2g+1)},S_{(3,t)}(t\geq 2)$ or $ S_{(4,1)}$ and thus is irreducible by Lemma \ref{lem four operations}.
     
When $k>2$, note that $\overline{AD^{k}E}=\overline{(AD^{k-1})DE}$. Then by induction on $k$, we obtain that $\overline{AD^kE}$ is irreducible for every $k\geq 1$.
\end{proof}

In the following lemma, we give a complete classification of the reduction of $\overline{Xw}$ with $\overline{X}$ irreducible and $w$ a letter in $\gs^{\pm}$.

\begin{lem}\label{lem for Xw types}
    Suppose $1\neq \overline{X}=\overline{x_1\cdots x_n}\in \W(\gs)$ is irreducible, and $w\in \gs^{\pm}$. Then $\nf(Xw)$ can be obtained by at most once reduction of $S_{(i)}$ and has the following five distinct cases for some $\overline{b_1\cdots b_{4g}}\in \RR$, where $\overline{B}:=\overline{b_{2g-1}\cdots b_1}$.
    \begin{center}
    \begin{tabular}{|l|l|M{7.3cm}|}
    \hline
       & $\overline{Xw}\xrightarrow{S_{(i)}}\nf(Xw)$  & Parameter declaration  \\
    \hline
    (1)& $\overline{Xw}\xrightarrow{S_{(1)}}\overline{x_1\cdots x_{n-1}}$ & $\overline{x_nw}=\overline{b_1b_1^{-1}}$  \\
    \hline
    (2)& $\overline{Xw}\xrightarrow{S_{(2,2g+1)}}\overline{x_1\cdots x_{n-2g}B}$ & $\overline{x_{n-2g+1}\cdots x_{n}w}=\overline{b_{4g}b_1\cdots b_{2g-1}b_{2g}}$\\
    \hline
    (3)& $\overline{Xw}\xrightarrow{S_{(3,t)}}\overline{x_1\cdots x_{n-(2g-1)t-1}B^t}$ & $\overline{x_{n-(2g-1)t}\cdots x_{n}w}=\overline{b_{4g}(b_1\cdots b_{2g-1})^tb_{2g}}(t\geq 2)$\\
    \hline
    (4)& $\overline{Xw}\xrightarrow{S_{(4,t)}}\overline{x_1\cdots x_{n-(2g-1)t}b_{2g}B^t}$ & $\overline{x_{n-(2g-1)t+1}\cdots x_{n}w}=\overline{(b_1\cdots b_{2g-1})^tb_{2g}}$, $b_1\succ b_{2g},x_{n-(2g-1)t}\ne b_{4g}$; $t\geq 1$ is maximal. \\
    \hline
    (5)& $\overline{Xw}$ is irreducible &  $\overline{Xw}$ is not in the previous four cases. \\
    \hline
    \end{tabular}
    \label{tab: types of Xw}
\end{center}
For $\overline{wX}$, we also have a symmetric classification, where $\overline{B}:=\overline{b_{2g-1}\cdots b_1}$.
\begin{center}
    \begin{tabular}{|l|l|M{7.3cm}|}
    \hline
       & $\overline{wX}\xrightarrow{S_{(i)}}\nf(wX)$  & Parameter declaration  \\
    \hline
    (1)& $\overline{wX}\xrightarrow{S_{(1)}}\overline{x_2\cdots x_{n}}$ & $\overline{wx_1}=\overline{b_1b_1^{-1}}$  \\
    \hline
    (2)& $\overline{wX}\xrightarrow{S_{(2,2g+1)}}\overline{Bx_{2g+1}\cdots x_{n}}$ & $\overline{wx_1\cdots x_{2g}}=\overline{b_{4g}b_1\cdots b_{2g-1}b_{2g}}$; $b_{1}\prec b_{2g}$\\
    \hline
    (3)& $\overline{wX}\xrightarrow{S_{(3,t)}}\overline{B^tx_{(2g-1)t+2}\cdots x_{n}}$ & $\overline{wx_1\cdots x_{(2g-1)t+1}}=\overline{b_{4g}(b_1\cdots b_{2g-1})^tb_{2g}}(t\geq 2)$\\
    \hline
    (4)& $\overline{wX}\xrightarrow{S_{(4,t)}}\overline{B^tb_{4g}x_{(2g-1)t+1}\cdots x_{n}}$ & $\overline{wx_{1}\cdots x_{(2g-1)t}}=\overline{b_{4g}(b_1\cdots b_{2g-1})^t}$; $b_{4g}\succ b_{2g-1},x_{(2g-1)t+1}\ne b_{2g}$; $t\geq 1$ is maximal. \\
    \hline
    (5)& $\overline{wX}$ is irreducible &  $\overline{wX}$ is not in the previous four cases. \\
    \hline
    \end{tabular}
    \label{tab: types of wX}
\end{center}
\end{lem}

The following corollary can be deduced from Lemma \ref{lem for Xw types} directly and is frequently used to obtain the irreducibility of some words in the proofs in Section \ref{sect proof of main thms}. Incidentally, Lemma \ref{lem for Xw types} can be used to estimate the standard growth rate, see Section \ref{sect remark}.

\begin{cor}\label{very useful cor}
    Let $\overline{b_1\cdots b_{4g}}\in \RR$. Then
    \begin{enumerate}
        \item if $\overline{Ab_1\cdots b_{2g}}$ is irreducible, then $\overline{A(b_{2g}\cdots b_2)^t}(t\geq 1)$ is again irreducible;
        \item if $\overline{b_2\cdots b_{2g+1}A}$ is irreducible, then $\overline{(b_{2g}\cdots b_2)^tA} (t\geq 1)$ is again irreducible.
    \end{enumerate}
\end{cor}
\begin{proof}
Considering the symmetry of items (1) and (2), we will focus our proof on item (1). By applying Lemma \ref{lem for Xw types}(2), we can perform the following reduction: 
	\begin{eqnarray*}
		\overline{Ab_1\cdots b_{2g}b_{2g+1}}\xrightarrow{S_{(2,2g+1)}}\overline{Ab_{2g}\cdots b_2}=\nf(Ab_1\cdots b_{2g}b_{2g+1}).\end{eqnarray*}
Additionally, observe that $\overline{(b_{2g}\cdots b_2)^2}$ is irreducible. Consequently, by applying Lemma \ref{lem from 2-nd power to n-th power}, it follows that $\overline{A(b_{2g}\cdots b_2)^t}$ remains irreducible for any $t\geq 1$.
\end{proof}

\begin{proof}[\textbf{Proof of Lemma \ref{lem for Xw types}}]
Since these two tables are symmetric, we provide a detailed proof for the first one and leave the second as an exercise for the reader. 

(1) Since $\overline{x_nw}=\overline{b_1b_1^{-1}}$ is of type $S_{(1)}$, we have	$$\overline{Xw}\xrightarrow{S_{(1)}}\overline{x_1\cdots x_{n-1}}=\nf(Xw),$$
where the ``='' holds because $\overline{x_1\cdots x_{n-1}}\subset \overline{X}$ is irreducible.

(2) Since
$\overline{x_{n-2g+1}\cdots x_{n}w}=\overline{b_{4g}b_1\cdots b_{2g}}$  is of type $S_{(2, 2g+1)}$,  we have
\begin{eqnarray*}
	\overline{Xw}&=&\overline{x_1\cdots x_{n-2g}b_{4g}b_1\cdots b_{2g-1}b_{2g}}\\
&\xrightarrow{S_{(2,2g+1)}}&\overline{x_1\cdots x_{n-2g}b_{2g-1}\cdots b_{1}}=:\overline{Y},
\end{eqnarray*}
where $x_{n-2g}\neq b_{2g},b_{4g-1}$ (Otherwise, $\overline{X}$ contains a reducible subword $\overline{b_{2g}b_{4g}}$ or $\overline{b_{4g-1}b_{4g}b_1\cdots b_{2g-1}}$, which contradicts the assumption that $\overline{X}$ is irreducible). According to Lemma \ref{lem frequently used}(1), if $\overline{Y}$ is reducible, there is a reducible subword $$\overline{V}=\overline{x_{n-2g-r+1}\cdots x_{n-2g}b_{2g-1}}\subset\overline{Y}$$ of types $S_{(i)}~(i=2,3,4)$. Thus, we have three subcases:
\begin{enumerate}
	\item[$1^\circ$] $\overline{V}=\overline{b_{4g-1}b_{4g}b_1\cdots b_{2g-1}}$ of type $S_{(2,2g+1)}$. Then, $\overline{X}$ contains a reducible subword 
    $$\overline{x_{n-4g+1}\cdots x_{n}}=\overline{b_{4g-1}(b_{4g}b_1\cdots b_{2g-2})^2b_{2g-1}}.$$
    
	\item[$2^\circ$] $\overline{V}=\overline{b_{4g-1}(b_{4g}b_1\cdots b_{2g-2})^tb_{2g-1}}$ of type $S_{(3,t)}$. Then, $\overline{X}$ contains a reducible subword $$\overline{x_{n-(2g-1)t-2g}\cdots x_{n}}= \overline{b_{4g-1}(b_{4g}b_1\cdots b_{2g-2})^{t+1}b_{2g-1}}.$$
    
	\item[$3^\circ$] $\overline{V}=\overline{b_{4g}b_1\cdots b_{2g-1}}~(b_{4g}\succ b_{2g-1})$ of type $S_{(4,1)}$. Then, $\overline{X}$ contains a reducible subword $$\overline{x_{n-2g+1}\cdots x_n}=\overline{b_{4g}b_1\cdots b_{2g-1}}.$$
\end{enumerate}
In all three cases, the irreducible word $\overline{X}$ contains reducible subwords, which is contradictory. Therefore, $\overline{Y}$ is irreducible.

(3) Since $\overline{x_{n-(2g-1)t}\cdots x_{n}w}=\overline{b_{4g}(b_1\cdots b_{2g-1})^tb_{2g}}~(t\geq 2)$ is of type $S_{(3,t)}$, we have
$$\overline{Xw}\xrightarrow{S_{(3,t)}}\overline{x_1\cdots x_{n-(2g-1)t-1}(b_{2g-1}\cdots b_1)^t}=:\overline{Y},$$
where $x_{n-(2g-1)t-1}\neq b_{2g},b_{4g-1}$. Then the remaining proof is very similar to that of Case (2).

(4) Since $\overline{x_{n-(2g-1)t+1}\cdots x_{n}w}=\overline{(b_1\cdots b_{2g-1})^tb_{2g}}(b_1\succ b_{2g})$, we have 
\begin{eqnarray*}
	\overline{Xw}&=&\overline{x_1\cdots x_{n-(2g-1)t}(b_1\cdots b_{2g-1})^tb_{2g}}\\
	&\xrightarrow{S_{(4,t)}}&\overline{x_1\cdots x_{n-(2g-1)t}b_{2g}(b_{2g-1}\cdots b_1)^t}=:\overline{Y}.
\end{eqnarray*}
Since $x_{n-(2g-1)t}\neq b_{4g}$, we have $\overline{Y}$ is freely reduced. We assume that $\overline{Y}$ has a reducible subword $\overline{V}$ of types $S_{(2,2g+1)},S_{(3,t)}~(t\geq 2)$ or $S_{(4,1)}$ containing $\overline{x_{n-(2g-1)t}b_{2g}}$. By Lemma \ref{lem frequently used}(3), we know that $\overline{V}=\overline{x_{n'}\cdots x_{n-(2g-1)t}b_{2g}}$ for some $n'\leq n-(2g-1)t$. Recall that $x_{n-(2g-1)t}\neq b_{2g+1}$, then we have the following possible cases:
\begin{equation*}
    \overline{V}=\left\{\begin{array}{ll}
        \overline{x_{n'}\cdots x_{n-(2g-1)t}b_{2g}}=\overline{b_{4g}b_1\cdots b_{2g-1}b_{2g}} & \mbox{$\overline{V}$ is of type $S_{(2,2g+1)}$} \\
        \overline{x_{n'}\cdots x_{n-(2g-1)t}b_{2g}}=\overline{b_{4g}(b_1\cdots b_{2g-1})^{t'}b_{2g}} &  \mbox{$\overline{V}$ is of type $S_{(3,t')}~(t'\geq 2)$} \\
         \overline{x_{n'}\cdots x_{n-(2g-1)t}b_{2g}}=\overline{b_1\cdots b_{2g-1}b_{2g}}& \mbox{$\overline{V}$ is of type $S_{(4,1)}$}
    \end{array}\right..
\end{equation*}
In either cases, $\overline{Xw}$ has a reducible subword of type $S_{(4,t+1)}$:
$$\overline{x_{n-(2g-1)(t+1)+1}\cdots x_{n}w}=\overline{(b_1\cdots b_{2g-1})^{t+1}b_{2g}}\subset \overline{Xw}.$$
 It contradicts the maximality of $t$ (see Remark \ref{rem of maximality}). Hence, $\overline{Y}$ does not have a reducible subword of types $S_{(1)},S_{(2,2g+1)},S_{(3,t')}~(t'\geq 2)$ or $S_{(4,1)}$ and thus it is irreducible.
	
(5) Since $\overline{Xw}$ is not in the previous four cases, it contains no reducible subword $\overline{V}$ containing $\overline{x_nw}$ of types $S_{(i)}$. Therefore, it is irreducible.
\end{proof}

\subsection{Cyclically irreducible words for Section \ref{sect. proof of key Proposition} and Appendix \ref{Appendix A1}}
As we have seen in Lemma \ref{lem from 2-nd power to n-th power}, a key step in generalizing the normal forms of squares to the normal forms of higher powers is to prove that the squares of some words are irreducible, i.e., the word itself is cyclically irreducible (see Proposition \ref{prop D cyclically irreducible= DD irreducible}). Thus, in what follows, we will prove in advance that the squares of words satisfying certain conditions are irreducible. Consequently, readers can skip the following two lemmas and return to them when they are cited in the proofs in Section \ref{sect. proof of key Proposition} and Appendix \ref{Appendix A1}.

\begin{lem}\label{lem types of W}
Let $\overline{b_1\cdots b_{4g}}\in\RR$. Then $\overline{A^k}$ is irreducible for every irreducible word $\overline{A}$ in Table \ref{tab: types of cyclically irreducible words 1} and any $k\geq 1$.
\begin{table}[ht]
    \centering
    \begin{tabular}{|l|c|M{5cm}|}
    \hline
       & Irreducible word $\overline{A}$  &Parameter declaration  \\
    \hline
    (1)&$\overline{x_1\cdots x_n(b_{2g}\cdots b_{2})^t}$& $\overline{Ax_1\cdots x_{n}}$ is irreducible;$\quad$
    $x_{1}\neq b_{1}$; $x_{n}\ne b_{2g+1}$\\
    \hline
    (2)&$\overline{x_1\cdots x_n(b_{2g-1}\cdots b_1)^{t}b_{2g-1}\cdots b_2(b_{2g}\cdots b_{2})^{t'}}$& $\overline{Ax_1\cdots x_{n}}$ is irreducible; $\quad$ $x_1\neq b_1$; $x_n\neq b_{2g}$ \\
    \hline
    (3)&$\overline{x_1\cdots x_n(b_{2g}\cdots b_2)^tb_1}$& $\overline{Ax_1\cdots x_{n}}$ is irreducible;$\quad$ $x_1\neq b_{4g}, b_{2g+1}$; $x_n\neq b_{2g+1},b_{4g}$; $b_{1}\succ b_{2g}$; $\overline{x_1\cdots x_n}\neq \overline{b_{2}\cdots b_{2g-1}}$\\
    \hline
    (4)&$\overline{b_{2g-1}\cdots b_2(b_{2g}\cdots b_2)^{t}b_1}$ &   $b_1\succ  b_{2g}$\\
    \hline
    (5)&$\overline{x_1\cdots x_nb_{2g}(b_{2g-1}\cdots b_1)^{t}b_{2g-1}\cdots b_2(b_{2g}\cdots b_{2})^{t'}}$& $\overline{Ax_1\cdots x_{n}}$ is irreducible;$\quad$ $x_1\neq b_1\succ b_{2g}$;  $x_n\neq b_{2g+1}$ \\
    \hline
    (6)&$\overline{x_1\cdots x_nb_{2g}(b_{2g-1}\cdots b_1)^{t}b_{2g-1}\cdots b_2(b_{2g}\cdots b_2)^{t'}b_1}$& $\overline{Ax_1\cdots x_{n}}$ is irreducible;$\quad$  $b_1\succ  b_{2g}$;  $\overline{x_1\cdots x_n}\neq \overline{b_2\cdots b_{2g-1}}$ \\
    \hline
    (7)&$\overline{(b_{2g-1}\cdots b_1)^tb_{2g-1}\cdots b_2(b_{2g}\cdots b_2)^{t'}b_1}$ & $b_1\succ b_{2g}$\\
    \hline
    \end{tabular}
\newline
	\caption{In all the cases, $n\geq 0$ and $t, t'\geq 1$. In particular, if $n=0$, the conditions regarding $x_i$ is satisfied by default.} \label{tab: types of cyclically irreducible words 1}
    \end{table}
\end{lem}

\begin{proof}
If $\overline{A^2}$ is irreducible, then by Lemma \ref{lem from 2-nd power to n-th power}, $\overline{A^k}$ is irreducible for any $k\geq 1$. We now prove the irreducibility of $\overline{A^2}$ in the following.

Note that in items (4) and (7), $t,t'\geq 1$ and $b_1\succ b_{2g}$, then $\overline{A^2}$ is irreducible because it is of the following forms respectively and does not have a subword of types $S_{(i)} (i=1,\ldots, 4)$:
\begin{enumerate}
    \item[(4)] $\overline{A^2}=\overline{b_{2g-1}\cdots b_2(b_{2g}\cdots b_2)^{t}b_1b_{2g-1}\cdots b_2(b_{2g}\cdots b_2)^{t}b_1}$; 
    \item[(7)] $\overline{A^2}=\overline{(b_{2g-1}\cdots b_1)^tb_{2g-1}\cdots b_2(b_{2g}\cdots b_2)^{t'}b_1(b_{2g-1}\cdots b_1)^tb_{2g-1}\cdots b_2(b_{2g}\cdots b_2)^{t'}b_1}.$ 
\end{enumerate}

For the remaining items (1), (2), (3), (5) and (6), we uniformly denote $\overline{X}:=\overline{x_1\cdots x_n}(n\geq 0)$ and $\overline{A}:=\overline{XY}$. Since $\overline{XYX}$ and $\overline{XY}$ are irreducible, $\overline{XYXY}$ is freely reduced. Furthermore, if there is a reducible subword 
$$\overline{V}\subset\overline{A^2}=\overline{XYXY}$$
of types $S_{(2,2g+1)},S_{(3)}$ or $S_{(4,1)}$, then $\overline{V}$ has two possible positions: 
\begin{enumerate}
    \item [(i)] starts in the left $\overline{X}$ and ends in the right $\overline{Y}$, or 
    
\item [(ii)] starts in the left $\overline{Y}$ and ends in the right $\overline{Y}$.
\end{enumerate}
However, in case (i), $\overline{X}\neq 1$ and $\overline{V}=\overline{x_{i}\cdots x_nYx_1\cdots x_ny_1\cdots y_j}$ contains a repeating letter $x_n$ thus is of types $S_{(3)}$. By applying Lemma \ref{lem some observations}(7), $\overline{x_i\cdots x_n y_1\cdots y_j} \subset \overline{XY}$ is reducible, which contradicts the irreducibility of $\overline{XY}$.  Therefore, the case (i) is impossible.
In the following, we shall prove that the case (ii): \emph{$\overline{V}$ starts in the left $\overline{Y}$ and ends in the right $\overline{Y}$}, is also impossible by deriving a contradiction for each item individually.

(1) Since $x_{1}\neq b_{1}$ and $x_{n}\ne b_{2g+1}$, we have
\begin{eqnarray*}
		\overline{V}=\overline{\cdots b_2X b_{2g}\cdots }
		\subset \overline{\underbrace{(b_{2g}\cdots b_3b_2)^t}_Y\underbrace{x_1\cdots x_n}_X \underbrace{(b_{2g}b_{2g-1}\cdots b_2)^t}_Y}
	\end{eqnarray*}
cannot be of type $S_{(2,2g+1)}$, $S_{(3)}$ or $S_{(4,1)}$. Otherwise, $\overline{V}$ must be of the following forms respectively:
$$\overline{b_2b_3\cdots b_{2g}b_{2g+1}b_{2g+2}}, \quad \overline{b_2(b_3\cdots b_{2g+1} )^{t'}b_{2g+2}}, \quad\overline{b_2b_3\cdots b_{2g+1}},$$ 
which implies $b_{2g+1}\in \overline{Y}$, a contradiction.

(2) Since $x_1\neq b_1$ and $x_n\neq b_{2g}$, we have
\begin{eqnarray*}
	\overline{V}&=&\overline{\cdots b_2Xb_{2g-1}\cdots }\\
	&\subset &\overline{\underbrace{(b_{2g-1}\cdots b_1)^{t}b_{2g-1}\cdots b_2(b_{2g}\cdots b_2)^{t'}}_Y\underbrace{x_1\cdots x_n}_X \underbrace{(b_{2g-1}\cdots b_1)^{t}b_{2g-1}\cdots b_2(b_{2g}\cdots b_2)^{t'}}_Y}.
\end{eqnarray*}
Then after a discussion similar to that of item (1), we can also get a contradiction.

(3) Since $b_{1}\succ b_{2g}$, $x_1\neq b_{4g}, b_{2g+1}$, $\overline{x_1\cdots x_n}\neq \overline{b_{2}\cdots b_{2g-1}}$ and $x_n\neq b_{2g+1},b_{4g}$, we have
\begin{eqnarray*}
	\overline{V}=\overline{\cdots b_{1} Xb_{2g} \cdots}
\subset \overline{\underbrace{(b_{2g}\cdots b_2)^tb_1}_Y \underbrace{x_1\cdots x_n}_X\underbrace{(b_{2g}\cdots b_2)^tb_1}_Y }.
\end{eqnarray*}
If $\overline{V}$ is of type $S_{(2,2g+1)}$,  $S_{(3)}$ or $S_{(4,1)}$, then it is of the following forms respectively:
$$\overline{b_1b_2\cdots b_{2g-1}b_{2g}b_{2g+1} },\quad \overline{b_1(b_2\cdots b_{2g})^{t'}b_{2g+1} },\quad \overline{b_1b_2\cdots b_{2g-1}b_{2g} }.$$
Note that the first two cases imply $b_{2g+1}\in \overline{Y}$, and the last one implies $\overline{x_1\cdots x_n}=\overline{b_2\cdots b_{2g-1}}$, a contradiction.

(5) Since $x_1\neq b_1\succ b_{2g}$ and $x_n\neq b_{2g+1}$, we have
\begin{eqnarray*}
	\overline{V}&=&\overline{\cdots b_{2} Xb_{2g} \cdots}\\
    &=&\overline{\underbrace{b_{2g}(b_{2g-1}\cdots b_1)^{t}b_{2g-1}\cdots b_2(b_{2g}\cdots b_{2})^{t'}}_Y \underbrace{x_1\cdots x_n}_X\underbrace{b_{2g}(b_{2g-1}\cdots b_1)^{t}b_{2g-1}\cdots b_2(b_{2g}\cdots b_{2})^{t'}}_Y }.
\end{eqnarray*}
Then after a similar discussion to that of item (1), we can obtain a contradiction.

(6) Since $b_1\succ  b_{2g}$ and $\overline{x_1\cdots x_n}\neq \overline{b_2\cdots b_{2g-1}}$, we have
\begin{eqnarray*}
	\overline{V}&=&\overline{\cdots b_{1} Xb_{2g} \cdots}\\
	&\subset &\overline{ \underbrace{b_{2g}(b_{2g-1}\cdots b_1)^{t}b_{2g-1}\cdots b_2(b_{2g}\cdots b_2)^{t'}b_1}_Y\underbrace{x_1\cdots x_n}_X\underbrace{b_{2g}(b_{2g-1}\cdots b_1)^{t}b_{2g-1}\cdots b_2(b_{2g}\cdots b_2)^{t'}b_1}_Y}.
\end{eqnarray*}
Then by repeating the same discussion outlined in item (3), we can again derive a contradiction. 
\end{proof}

\subsection{\texorpdfstring{Words of type $\overline{U}(i,j)$ and cyclically irreducible words for Appendix \ref{Appendix A2}}{Words of Type overline{U}(i,j) and cyclically irreducible words for Appendix \ref{Appendix A2}}}
The following lemma mainly serves the proof of Proposition \ref{prop classification of LLFR of X^2}(\ref{XX reducible})($|\overline{T}|\in\{2g+1,\dots,4g-2\}$), see Appendix \ref{Appendix A2}. So, readers may skip it and revisit it when necessary. First, we need a definition as follows.

\begin{defn}[Word of Type $\overline{U}(i,j)$]\label{defn irreducible word pair of type (i,j)}
 Let $\overline{b_1\cdots b_{4g}}\in\RR$ and let $\overline{U}=\overline{u_1\cdots u_n}~(n\geq 0)$ be an irreducible word in $\W(\gs)$. We say $\overline{A}\in \W(\gs)$ is \emph{of type $\overline{U}(i,j)$} ($2 \leq i\leq j\leq 5$) if it satisfies the following conditions:
$$\overline{A}=\overline{UL_i'MR_j'},$$
where $\overline{M}=\overline{b_{2g-1}b_{2g-2}\cdots b_{k-2g+2}}$~($2g+1\leq k \leq 4g-2$, and $\overline{M}=1$ when $k=4g-2$) and $\overline{L_i'},~\overline{R_j'}$ are of the forms corresponding to $i,j$ in the following tables, respectively. In particular, $\overline{U}\neq 1$ (i.e. $n> 0$) when $k=4g-2$ and $ 2\leq i\leq j\leq 3$. The conditions regarding $u_1, u_n$ are satisfied by default when $n=0$.
    \begin{center}
    \begin{tabular}{|M{0.3cm}|M{3.7cm}|M{5cm}|}
    \hline
       $i$ & $\overline{L_i'}$ & Parameters    \\
    \hline
    $2$&$\overline{b_{2g-1}\cdots b_1}$ &$u_n\neq b_{2g} $ \\
    \hline
    $3$ &$\overline{(b_{2g-1}\cdots b_1)^t}$ &$u_n\neq b_{2g}; ~t\geq 2$ \\
    \hline
    $4$ &$\overline{b_{2g}(b_{2g-1}\cdots b_1)^t} $ &$u_n\neq b_{2g+1}; ~b_1\succ b_{2g}; ~t\geq 1$ \\
    \hline
    $5$ &$\overline{b_{2g}}$ & $u_n\neq b_{2g+1}, b_{4g}$ \\
    \hline
    \end{tabular}
    \end{center}
    \begin{center}
    \begin{tabular}{|M{0.3cm}|M{3.7cm}|M{5cm}|}
    \hline
       $j$ & $\overline{R_j'}$ & Parameters    \\
    \hline
    $2$ &$\overline{b_k\cdots b_{k-2g+2}} $ &$u_1\neq b_{k-2g+1} $ \\
    \hline
    $3$&$\overline{(b_k\cdots b_{k-2g+2})^{t'}} $ &$u_1\neq b_{k-2g+1}; ~t'\geq 2$\\
    \hline
    $4$&$\overline{(b_k\cdots b_{k-2g+2})^{t'}b_{k-2g+1}} $ &$u_1\neq b_{k-2g}; ~b_{k-2g+1}\succ b_k;~t'\geq 1 $ \\
    \hline
    $5$&$\overline{b_{k-2g+1}}$ &$u_1\neq b_{k-2g},b_{k+1} $ \\
    \hline
    \end{tabular}    
    \end{center}
\end{defn}

Then, we can obtain the following lemma.
\begin{lem}\label{lem types of W for lengh =2g+1 to 4g-2}
If $\overline{A}\in \W(\gs)$ is of type $\overline{U}(i,j)$ and
$\overline{AU}$ is irreducible, then $\overline{A^k}$ is irreducible for any $k\geq 1$.
\end{lem}	
\begin{proof}
To prove  that $\overline{A^k}$ is irreducible for any $k\geq 1$, by Lemma \ref{lem from 2-nd power to n-th power}, it suffices to prove $\overline{A^2}$ is irreducible. Since $\overline{A}$ is of type $\overline{U}(i,j)$ and $\overline{AU}$ is irreducible, $\overline{A}$ is cyclically freely reduced and has a uniform expression $$\overline{A}=\overline{\underbrace{u_1\cdots u_n}_U\underbrace{L_i'MR_j'}_W} ~(n\geq 0, ~ 2\leq i\leq j\leq 5).$$ 
If $\overline{A^2}=\overline{UWUW}$ is reducible, then by Lemma \ref{lem four operations}, it contains a reducible subword $\overline{V}$ of types $S_{(2,2g+1)},S_{(3)}$ or $S_{(4,1)}$. Furthermore, since $\overline{AU}=\overline{UWU}$ is irreducible, by Lemma \ref{lem some observations}(6), $\overline{V}$ must start in the left $\overline{UW}$ and end in the right $\overline{W}$. However, if $\overline{V}$ starts in the left $\overline{U}$ and ends in the right $\overline{W}$, then $$\overline{V}=\overline{u_i\cdots u_nWu_1\cdots u_nW_0}$$ contains a repeating letter $u_n$ and thus is of type $S_{(3)}$. Then, by Lemma \ref{lem some observations}(7), there is a reducible subword $\overline{u_i\cdots u_nW_0}\subset\overline{UW}$, which contradicts the irreducibility of $\overline{UW
}$. Therefore, $$\overline{V}\subset\overline{A^2}=\overline{UWUW}$$ can only start in the left $\overline{W}$ and end in the right $\overline{W}$. 

In the subsequent discussion, we will demonstrate that such reducible subword $\overline{V}$ cannot be found, thereby indicating that $\overline{A^2}$ is irreducible. It is important to note that we will repeatedly rely on the fact ``every $\llfr$  $\overline{T}\subset\overline{V}$ has length $|\overline{T}|\geq 2g-1$''. To facilitate our discussion, we need to divide these cases as follows.

	\textbf{Case (1).} $k=4g-2$ and $i,j\neq 5$. Then $\overline{M}=1$ and hence $$\overline{W}=\overline{\underbrace{\cdots b_{2g-1}\cdots b_2b_1}_{L_i'}\underbrace{b_{4g-2}b_{4g-3}\cdots b_{2g}\cdots }_{R_j'}}.$$ Since $\overline{b_1b_{4g-2}}\subset \overline{W}$ is not a subword of any word of types $S_{(2,2g+1)},S_{(3)}$ or $S_{(4,1)}$, we have $\overline{b_1b_{4g-2}}\notin\overline{V}$, and hence $\overline{V}\subset \overline{WUW}=\overline{L'_iR'_jUL'_iR'_j}$ must start in the left $\overline{R_j'}$ and end in the right $\overline{L_i'}$.

Subcase (1.1). $i=2,3;j=2,3$. Then $\overline{U}\neq 1$ and
		\begin{eqnarray*}
			\overline{V}=\overline{\cdots b_{2g}Ub_{2g-1}\cdots }
			\subset \overline{\underbrace{(b_{4g-2}\cdots b_{2g})^{t'}}_{R_j'}\underbrace{u_1\cdots u_n}_{U\neq 1} \underbrace{(b_{2g-1}\cdots b_{1})^t}_{L_i'}},
		\end{eqnarray*}
		where $u_1\neq b_{2g-1};u_n\neq b_{2g}$ and $t,t'\geq 1$. It implies that $\overline{V}$ cannot be  of types $S_{(2,2g+1)},S_{(4,1)}$ or $S_{(3)}$,  a contradiction.

        Subcase (1.2). $i=2,3;j=4$. Then
		\begin{eqnarray*}
			\overline{V}=\overline{\cdots b_{2g-1}Xb_{2g-1}\cdots }
			\subset \overline{\underbrace{(b_{4g-2}\cdots b_{2g})^{t'}b_{2g-1}}_{R_4'}\underbrace{u_1\cdots u_n}_{U}\underbrace{(b_{2g-1}\cdots b_{1})^t}_{L'_i}},
		\end{eqnarray*}
		where $u_1\neq b_{2g-2};u_n\neq b_{2g}$ and $t,t'\geq 1$. Note that the letter $b_{2g-1}$ appears more than once in $\overline{V}$,  thus 
		$$\overline{V}=\overline{b_{4g-2}(b_{4g-3}\cdots b_{2g-1})^{t_0}b_{2g-2}}$$
is of type $S_{(3)}$,	which leads to $u_n=b_{2g}$, a contradiction.
		
Subcase (1.3). $i=4;j= 4$. Then
		\begin{eqnarray*}
			\overline{V}=\overline{\cdots b_{2g-1}Ub_{2g}\cdots }
			\subset \overline{\underbrace{(b_{4g-2}\cdots b_{2g})^{t'}b_{2g-1}}_{R_4'}\underbrace{u_1\cdots u_n}_{U}\underbrace{b_{2g}(b_{2g-1}\cdots b_{1})^t}_{L_4'}},
		\end{eqnarray*}
		where $u_1\neq b_{2g-2};u_n\neq b_{2g+1}$; $t,t'\geq 1$; $b_{2g-1}\succ b_{4g-2}$ and $ ~b_1\succ b_{2g}$. Since $\overline{R'_4L'_4}$ is irreducible, it contains no reducible subword. Therefore, we have $\overline{U}\neq 1$ and $\overline{V}$ is not of types $S_{(2,2g+1)}$ or $S_{(4,1)}$. Hence $\overline{V}$ is of type $S_{(3)}$ but $$\overline{V}=\overline{b_{2g-1}(b_{2g}\cdots b_{4g-2})^{t_0}b_{4g-1}}\nsubseteq \overline{R_4'UL_4'}.$$
Therefore, in either case, it is contradictory.

\textbf{Case (2).} $k=4g-2$ and $i\neq 5;j=5$. Then $\overline{M}=1$ and $\overline{R_5'}=\overline{b_{2g-1}}$.
	
Subcase (2.1). $i=2,3$. Then 
	\begin{eqnarray*}
    \overline{W}&=&\overline{\underbrace{(b_{2g-1}\cdots b_1)^t}_{L_i'}\underbrace{b_{2g-1}}_{R_5'}},\\
	\overline{V}&=&\overline{\cdots b_{2g-1}Ub_{2g-1}\cdots }
	\subset \overline{\underbrace{(b_{2g-1}\cdots b_1)^tb_{2g-1}}_{W}\underbrace{u_1\cdots u_n}_{U}\underbrace{(b_{2g-1}\cdots b_1)^tb_{2g-1}}_{W}},
	\end{eqnarray*}
where $t\geq 1;\ u_1\neq b_{2g-2},b_{4g-1}$ and $u_n\neq b_{2g}$. Furthermore, by Lemma \ref{lem some observations}(7), if we delete $\overline{b_{2g-1}U}$ from $\overline{V}$ and $\overline{WUW}$ to get $\overline{V'}$ and $\overline{W'}$ respectively, then $\overline{V'}\subset \overline{W'}$ will be of types $S_{(2,2g+1)}$ or $S_{(3)}$. However, $$\overline{W'}=\overline{(b_{2g-1}\cdots b_1)^t(b_{2g-1}\cdots b_1)^tb_{2g-1}}$$
does not contain a reducible subword, which is contradictory.
	
Subcase (2.2). $i=4$. Then 
	\begin{eqnarray*}
    \overline{W}&=&\overline{\underbrace{b_{2g}(b_{2g-1}\cdots b_1)^t}_{L_4'}\underbrace{b_{2g-1}}_{R_5'}},\\
	\overline{V}&=&\overline{\cdots b_{2g-1}Ub_{2g}\cdots }
		\subset \overline{\underbrace{b_{2g}(b_{2g-1}\cdots b_1)^tb_{2g-1}}_{W}\underbrace{u_1\cdots u_n}_{U}\underbrace{b_{2g}(b_{2g-1}\cdots b_1)^tb_{2g-1}}_{W}},
	\end{eqnarray*}
	where $t\geq 1;\ u_1\neq b_{2g-2},b_{4g-1}$ and $u_n\neq b_{2g}$. Clearly, $\overline{V}$ is not of types $S_{(2,2g+1)}$ or $S_{(4,1)}$ and thus is of type $S_{(3)}$. Therefore, $\overline{V}=\overline{b_{2g-1}(b_{2g}\cdots b_{4g-2})^{t_0}b_{4g-1}}\nsubseteq \overline{WUW}$, which is contradictory.
	
\textbf{Case (3).} $i,j=5$. Then 
	\begin{eqnarray*}
    \overline{W}&=&\overline{\underbrace{b_{2g}}_{L_5'}\underbrace{b_{2g-1}\cdots b_{k-2g+2}}_{M}\underbrace{b_{k-2g+1}}_{R_5'}},\\
		\overline{V}&=&\overline{\cdots b_{k-2g+1}U b_{2g}\cdots}
		\subset \overline{\underbrace{b_{2g}\cdots b_{k-2g+1}}_{W}\underbrace{u_1\cdots u_n}_{U}\underbrace{b_{2g}\cdots b_{k-2g+1}}_{W}},
	\end{eqnarray*}
	where $u_1\neq b_{k-2g},b_{k+1};u_n\neq b_{2g+1},b_{4g}$.  Furthermore, if $\overline{V}$ is of types $S_{(2,2g+1)}$ or $S_{(4,1)}$, then 
	$$\overline{V}=\overline{b_{k-2g+1}b_{k-2g+2}\cdots b_{2g-1}b_{2g}}$$
 with length $=4g-k\leq 2g-1$ contradicts $|\overline{V}|=2g+1, 2g$; if $\overline{V}$ is of type $S_{(3)}$, then
	$$\overline{V}=\overline{b_{k-2g+1}(b_{k-2g+2}\cdots b_{k})^{t_0}b_{k+1}}\nsubseteq\overline{WUW},$$
	which is also contradictory.
	
\textbf{Case (4).} $2g+1\leq k\leq 4g-3$ and $i,j\neq 5$. Then $\overline{M}$ has length $\leq 2g-3$.  Since $\overline{M}$ is an $\llfr$ or a single letter that cannot form a fractional relator with adjacent letters in $\overline{WUW}$, we obtain that if $$\overline{V}\subset \overline{WUW}=\overline{L_i'MR_j'UL_i'MR_j'}$$ contains a subword $ \overline{W'_2}\subset\overline{M}$, then $\overline{V}$ cannot be of types $S_{(2,2g+1)},S_{(3)}$ or $S_{(4,1)}$ because $|\overline{W'_2}|\leq |\overline{M}|\leq 2g-3$. Therefore, $\overline{V}$ starts in the left $\overline{R_j'}$ and ends in the right $\overline{L_i'}$.

Subcase (4.1). $i=2,3;j=2,3$. Then
		\begin{eqnarray*}
		\overline{V}=\overline{\cdots b_{k-2g+2}Ub_{2g-1} \cdots}
			\subset \overline{\underbrace{(b_{k}\cdots b_{k-2g+2})^{t'}}_{R_j'}\underbrace{u_1\cdots u_n}_{U}\underbrace{(b_{2g-1}\cdots b_1)^{t}}_{L_i'}},
		\end{eqnarray*}
		where $t,t'\geq 1;u_1\neq b_{k-2g+1};u_n\neq b_{2g}$ and $b_{k-2g+2}\in\{b_{3},\dots ,b_{2g-1}\}$. Clearly, $\overline{V}$ is not of types $S_{(2,2g+1)},S_{(4,1)}$ and thus is of type $S_{(3)}$. Upon a simple check, we obtain that $\overline{V}$ cannot be of type $S_{(3)}$, neither.
		
Subcase (4.2). $i=2,3;j=4$. Then
		\begin{eqnarray*}
			\overline{V}=\overline{\cdots b_{k-2g+1}Ub_{2g-1} \cdots}
			\subset \overline{\underbrace{(b_{k}\cdots b_{k-2g+2})^{t'}b_{k-2g+1}}_{R_4'}\underbrace{u_1\cdots u_n}_{U}\underbrace{(b_{2g-1}\cdots b_1)^{t}}_{L_i'}},
		\end{eqnarray*}
		where $t,t'\geq 1;x_1\neq b_{k-2g};x_n\neq b_{2g}$ and $b_{k-2g+1}\in\{b_{2},\dots ,b_{2g-2}\}$. With the same argument as in Subcase (4.1), we can prove that there is no such reducible subword $\overline{V}$ of types $S_{(2,2g+1)},S_{(3)}$ or $S_{(4,1)}$.
		
Subcase (4.3). $i=4;j=4$. Then
		\begin{eqnarray*}
			\overline{V}=\overline{\cdots b_{k-2g+1}Ub_{2g} \cdots}
			\subset \overline{\underbrace{(b_{k}\cdots b_{k-2g+2})^{t'}b_{k-2g+1}}_{R_4'}\underbrace{u_1\cdots u_n}_{U}\underbrace{b_{2g}(b_{2g-1}\cdots b_1)^{t}}_{L_4'}},
		\end{eqnarray*}
		where $t,t'\geq 1;n\geq 0;x_1\neq b_{k-2g};x_n\neq b_{2g}$ and $b_{k-2g+1}\in\{b_{2},\dots ,b_{2g-2}\}$. The argument is the same as the above two subcases.

\textbf{Case (5).} $2g+1\leq k\leq 4g-3$ and $i=2,3;j= 5$. Then
	 \begin{eqnarray*}
     \overline{W}&=&\overline{\underbrace{(b_{2g-1}\cdots b_{1})^t}_{L_i'}\underbrace{b_{2g-1}\cdots b_{k-2g+2}}_{M}\underbrace{b_{k-2g+1}}_{R_5'}},\\
	 	\overline{V}&=&\overline{\cdots b_{k-2g+1}Ub_{2g-1}\cdots }\\
	 	&\subset &\overline{WUW}=\overline{\underbrace{(b_{2g-1}\cdots b_1)^{t}}_{L_i'}\underbrace{b_{2g-1}\cdots b_{k-2g+1}}_{MR_5'}\underbrace{u_1\cdots u_n}_{U}\underbrace{(b_{2g-1}\cdots b_1)^{t}}_{L_i'}\underbrace{b_{2g-1}\dots b_{k-2g+1}}_{MR_5'}},
	 \end{eqnarray*} 
	where $t\geq 1;u_1\neq b_{k-2g},b_{k+1}$ and $u_n\neq b_{2g}$. Note that the two $\overline{MR_5'}$ have length $2\leq |\overline{MR_5'}|\leq 2g-2$ and are both $\llfr$s in $\overline{WUW}$. Furthermore, if $$\overline{V}=\overline{\cdots b_{k-2g+2}\underbrace{b_{k-2g+1}}_{R_5'}\underbrace{u_1\cdots u_n}_{U}b_{2g-1}\cdots }~~\mbox{ or }~~\overline{\cdots \underbrace{b_{k-2g+1}}_{R_5'} \underbrace{u_1\cdots u_n}_{U}\underbrace{(b_{2g-1}\cdots b_1)^t}_{L_i'}b_{2g-1}b_{2g-2}\cdots },$$
	there will be an $\llfr~(\subset\overline{MR_5'})$  with length $\leq 2g-2$ at the beginning or end of $\overline{V}$, respectively, which contradicts the fact ``every $\llfr$ $\overline{T}\subset\overline{V}$ has length $|\overline{T}|\geq 2g-1$''. Therefore, 
	\begin{eqnarray*}
		\overline{V}=\overline{R_5'Ub_{2g-1}\cdots }\subset \overline{\underbrace{b_{k-2g+1}}_{R_5'} \underbrace{u_1\cdots u_n}_{U}\underbrace{(b_{2g-1}\cdots b_1)^{t}}_{L_i'}b_{2g-1}}.
	\end{eqnarray*}
    Note that $b_{k-2g+1}\in\{b_{2},\dots ,b_{2g-2}\}$, $u_1\neq b_{k-2g}$ and $u_n\neq b_{2g}$.
It is straightforward to verify that there is no such reducible subword $\overline{V}$ of types $S_{(2,2g+1)},S_{(3)}$ or $S_{(4,1)}$.

\textbf{Case (6).} $2g+1\leq k\leq 4g-3$ and $i=4;j= 5$. Then 
\begin{eqnarray*}
\overline{W}&=&\overline{\underbrace{b_{2g}(b_{2g-1}\cdots b_1)^t}_{L_4'}\underbrace{b_{2g-1}\cdots b_{k-2g+2}}_{M}\underbrace{b_{k-2g+1}}_{R_5'}},\\
\overline{V}&\subset&\overline{WUW}=\overline{L_4'MR_5'UL_4'MR_5'},
\end{eqnarray*} 
where $t\geq 1;u_1\neq b_{k-2g},b_{k+1}$ and $u_n\neq b_{2g+1}$. Note that the two $\overline{MR_5'}$ have length $2\leq |\overline{MR_5'}|\leq 2g-2$ and are both $\llfr$s in $\overline{WUW}$. Furthermore, if  $$\overline{V}=\overline{\cdots b_{k-2g+2}\underbrace{b_{k-2g+1}}_{R_5'}\underbrace{u_1\cdots u_n}_Ub_{2g}\cdots }~~\mbox{ or }~~\overline{\cdots \underbrace{b_{k-2g+1}}_{R_5'}\underbrace{u_1\cdots u_n}_U\underbrace{b_{2g}(b_{2g-1}\cdots b_1)^t}_{L_4'}b_{2g-1}b_{2g-2}\cdots },$$
there will be an $\llfr~(\subset\overline{MR_5'})$ with length $\leq 2g-2$ at the beginning or end of $\overline{V}$, respectively, which is contradictory. Therefore, 
\begin{eqnarray*}
	\overline{V}=\overline{R_5'Ub_{2g}\cdots }
	\subset \overline{\underbrace{b_{k-2g+1}}_{R_5'}\underbrace{u_1\cdots u_n}_{U}\underbrace{b_{2g}(b_{2g-1}\cdots b_1)^{t}}_{L_4'}b_{2g-1}}.
\end{eqnarray*}
Note that $b_{k-2g+1}\in\{b_{2},\dots ,b_{2g-2}\}$.
It is straightforward to verify that there is no such reducible subword $\overline{V}$ of types $S_{(2,2g+1)},S_{(3)}$ or $S_{(4,1)}$.

In conclusion, through an examination of the following table, we have proved that for all cases, there does not exist a reducible word $\overline{V}\subset \overline{A^2}=\overline{UWUW}$. Therefore, $\overline{A^2}$ is irreducible and the proof is complete. \end{proof}
\begin{center}
	\begin{tabular}{|c|c|c|}
		\hline
		& $k=4g-2$ & $k=2g+1,\cdots,4g-3$ \\
		\hline
		$i=2,3; ~j=2,3$ & Case (1) & Case (4)  \\
		\hline
		$i=2,3; ~j=4$ & Case (1) & Case (4)  \\
		\hline
		$i=4; ~j=4$ & Case (1)  & Case (4)  \\
		\hline
		$i=2,3;~j=5$ & Case (2)  & Case (5)  \\
		\hline
		$i=4; ~j=5$&  Case (2) & Case (6) \\
		\hline
		$i=5;~j=5$ & Case (3) & Case (3)  \\
		\hline
	\end{tabular}	
	\end{center}

\section{Proofs of Theorem \ref{main thm normal forms of powers} and Theorem \ref{main thm word length formulae}}\label{sect proof of main thms}

In this section, we shall prove Theorem \ref{thm normal forms of power elements (original main thm1)},  and then prove Theorem \ref{main thm normal forms of powers} and Theorem \ref{main thm word length formulae}. 
First, we have the following lemma, which serves as the basis for the case-by-case discussion in the key Proposition \ref{prop classification of LLFR of X^2}.

\begin{lem}\label{prop length classification of LLFR at the joint}
 Let $\overline{U}=\overline{u_1\cdots u_l}$ and $\overline{V}=\overline{v_1\cdots v_m}$ be two nonempty irreducible words in $\W(\gs)$. If $\overline{u_lv_1}$ is a fractional relator, then the $\llfr$ $\overline{T}$ of $\overline{UV}$ that contains $\overline{u_lv_1}$ has length 
 $$|\overline{T}|\in\{2,3,\dots,4g-1\}.$$ Furthermore, for any nonempty irreducible word $\overline{X}=\overline{x_1\cdots x_n}\in\W(\gs)$, if there exists an $\llfr$ of $\overline{X^2}$ that contains $\overline{x_nx_1}$ with length $\leq 2g-2$, then $\overline{X^2}$ is irreducible.
 \end{lem}
\begin{proof}
   To prove $|\overline T|\in\{2,3,\ldots,4g-1\},$ it suffices to prove $|\overline T|\neq 4g$ according to the definition of $\llfr$. Suppose $$\overline T=\overline{u_{l-s+1}\cdots u_lv_1\cdots v_{4g-s}}=\overline{b_1\cdots b_{4g}}\in \RR$$ is an $\llfr$ containing $\overline{u_lv_1}$ with length $4g$. Then
    \begin{eqnarray*}
      \overline{u_{l-s+1}\cdots u_{l}}&=&\overline{b_1\cdots b_{s}}\subset \overline{U},\\
      \overline{v_1\cdots v_{4g-s}}&=&\overline{b_{s+1}\cdots b_{4g}}\subset \overline{V},  
    \end{eqnarray*} are both irreducible, and hence $s,4g-s\leq 2g$. Thus $s=4g-s=2g$ and $$\overline{u_{l-2g+1}\cdots u_{l}}=\overline{b_1\cdots b_{2g}}, ~\quad\overline{v_1\cdots v_{2g}}=\overline{b_{2g+1}\cdots b_{4g}}=\overline{b_{1}^{-1}\cdots b_{2g}^{-1}}.$$
	Since these two words are irreducible, we have $b_1\prec b_{2g}$ and $b_{1}^{-1}\prec b_{2g}^{-1}$, which is contradictory.

For the irreducible word $\overline{X}$, since there exists an $\llfr$ that contains $\overline{x_nx_1}$ and has length $\leq 2g-2$, we have $x_n\neq x_1^{-1}$. Moreover, note that $\llfr$s of reducible words of types $S_{(2,2g+1)},S_{(3,t)}$ or $S_{(4,1)}$ have length $\geq 2g-1$. Thus, $\overline{X^2}$ does not contain a reducible subword of types $S_{(1)},S_{(2,2g+1)}$, $S_{(3,t)}$ or $S_{(4,1)}$, by Lemma \ref{lem four operations}, $\overline{X^2}$ is irreducible.
\end{proof}

Before posting the key Proposition \ref{prop classification of LLFR of X^2}, we need a definition of three special kinds of irreducible, cyclically freely reduced but not cyclically irreducible words.

\begin{defn}[Word of Types $\mathfrak{A},\mathfrak{B}$ or $\mathfrak{C}$]\label{defn word of types ABC}
Let $\overline{X}=\overline{x_1\cdots x_n}\in \W(\gs)$. We say that

\begin{enumerate}
    \item $\overline{X}$ is of \emph{type $\mathfrak{A}$} if for some $\overline{b_1\cdots b_{4g}}\in\RR,~b_1\succ b_{2g}, 1\leq r\leq 2g-1$ and $t_1,t_2\geq 0$, 
    $$\overline{X}=\overline{b_{r+1}\cdots b_{2g}(b_2\cdots b_{2g})^{t_1}b_2\cdots b_{2g-1}(b_1\cdots b_{2g-1})^{t_2}b_1\cdots b_r};$$ 
    \item $\overline{X}$ is of \emph{type $\mathfrak{B}$} if for some
    $\overline{b_1\cdots b_{4g}}\in\RR,~b_1\prec b_{2g}$ and $t\geq 1$,
    $$\overline{X}=\overline{b_1(b_2\cdots b_{2g})^tx_{m_1}\cdots x_{m_2}(b_{2g+2}\cdots b_{4g})^t},$$
    where $x_{m_1}\neq b_{4g},~ x_{m_2}\neq b_2$ and $\overline{b_1x_{m_1}\cdots x_{m_2}}$ is of type $\mathfrak{A}$;
    \item $\overline{X}$ is of \emph{type $\mathfrak{C}$} if for some
    $\overline{b_1\cdots b_{4g}}\in\RR,~b_{2g+1}\prec b_1$ and $t\geq 1$,
    $$\overline{X}=\overline{(b_2\cdots b_{2g})^tx_{m_1}\cdots x_{m_2}(b_{2g+2}\cdots b_{4g})^tb_1},$$
    where $x_{m_1}\neq b_{4g},~x_{m_2}\neq b_2$ and $\overline{x_{m_1}\cdots x_{m_2}b_1}$ is of type $\mathfrak{A}$.
\end{enumerate}  
\end{defn}

\begin{prop}\label{prop classification of LLFR of X^2} Let $G$ be a surface group with the symmetric presentation (\ref{symmetric presentation}) and the order (\ref{order of generators}).
Let $1\neq x\in G$ and $\nf(x)=\overline{YXY^{-1}}$ with $\overline{X}=\overline{x_1\cdots x_n}(n\geq 1)$ cyclically freely reduced and $\overline{Y}=\overline{y_1\cdots y_m}(m\geq 0)$ possibly empty.

 \begin{enumerate}
\item  \label{XX irreducible} If $\overline{X^2}$ is irreducible (i.e., $\overline{X}$ is cyclically irreducible), then $\nf(x^k)=\overline{YX^kY^{-1}}$ for any $k\geq 1$.

\item \label{XX reducible} If $\overline{X^2}$ is reducible, then for any $k\geq 1$,
     $$\nf(x^k)=\overline{Yx_1\cdots x_i W^{k-1}x_{i+1}\cdots x_nY^{-1}}$$
     with cyclically irreducible word $\overline{W}=\nf(x_{i+1}\cdots x_nx_1\cdots x_i)  ~(1\leq i\leq n)$,
    except for the following three special cases:   
(a) $\overline{X}$ is of type $\mathfrak{A}$, (b) $\overline{X}$ is of type $\mathfrak{B}$, (c) $\overline{X}$ is of type $\mathfrak{C}$.
 In each of the cases (a), (b) and (c), for any $k\geq 2$, we obtain
\begin{equation}\label{word of prop 2g}
		\nf(x^k)=\overline{Yx_1\cdots x_i(x'_1\cdots x'_j W^{k-2}x'_{j+1}\cdots x'_{n'})x_{i+1}\cdots x_nY^{-1}} \quad(1\leq i\leq n; ~1\leq j\leq n'),
	\end{equation}
where $|\overline{W}|=n'$ and
$$\overline{x'_1\cdots x'_{n'}}=\nf(x_{i+1}\cdots x_nx_1\cdots x_i), \quad \overline{W}=\nf(x'_{j+1}\cdots x'_{n'}x'_1\cdots x'_j).$$
  \end{enumerate}
 \end{prop}

Since the discussion in the proof of Proposition \ref{prop classification of LLFR of X^2} is extensive, to enhance readability and ensure completeness, we shall introduce the key ideas and procedures of our proof in Section \ref{sect. proof of key Proposition}, with the detailed proof provided in Appendix \ref{sect appendix for some detailed proofs of props}. 

By directly summarizing the items in Proposition \ref{prop classification of LLFR of X^2}, we obtain the following key theorem.

\begin{thm}\label{thm normal forms of power elements (original main thm1)}
Let $G$ be a surface group with the symmetric presentation (\ref{symmetric presentation}) and the order (\ref{order of generators}). Then, for any $1\neq x\in G$ with its normal form $\nf(x)=\overline{YXY^{-1}}$, where $\overline{X}=\overline{x_1\cdots x_n}$ is cyclically freely reduced, we have
\begin{equation}\label{eq. nf(xk)}
   \nf(x^k)=\overline{Y\nf(X^k)Y^{-1}}=\overline{Yx_1\cdots x_{i}(x'_{1}\cdots x'_{j} W^{k-2}x'_{j+1}\cdots x'_{n'})x_{i+1}\cdots x_nY^{-1}} \quad (k\geq 2),
\end{equation}
	where $\overline{x'_1\cdots x'_{n'}}=\nf(x_{i+1}\cdots x_nx_1\cdots x_i)$ and $\overline{W}=\nf(x_{j+1}'\cdots x_{n'}'x_1'\cdots x_j')$ is cyclically irreducible with $|\overline{W}|=n'\geq 1$ for some  $1\leq i\leq n$ and $1\leq j\leq n'$.
\end{thm}

\begin{proof}[\textbf{Proof of Theorem \ref{main thm normal forms of powers}}] 
Theorem \ref{main thm normal forms of powers} holds immediately
by applying Theorem \ref{thm normal forms of power elements (original main thm1)} and setting $$ \overline{L}=\overline{Yx_1\cdots x_{i}x_1'\cdots x_{j}'},\qquad \overline{R}=\overline{x_{j+1}'\cdots x_{n'}x_{i+1}\cdots x_{n}Y^{-1}}.$$    
\end{proof}

Note that in Eq. (\ref{eq. nf(xk)}), we have $x_1'\cdots x_{n'}'$ conjugate to $X$ and $W$ conjugate to $x_1'\cdots x_{n'}'$. Thus, $W$ is conjugate to $X$  and hence conjugate to $x$. 
However, the cyclically irreducible word $\overline{W}$ is not unique and any such two district $\overline{W}$'s differ by a cyclic permutation. For convenience, we define the \emph{cyclically irreducible core} of $x$ as follows:  
\begin{equation}\label{def ciw}
   \ciw(x):= \mbox{the ~word} ~\overline{W} ~\mbox{such ~that the pair} ~(i, j) ~\mbox{is ~maximal ~in ~lexicographical ~order ~in Eq.} ~(\ref{eq. nf(xk)}).
\end{equation} 

We have the following corollary on $\ciw(x)$.

\begin{cor}\label{cor cyclically irreducible word notation}
With the same notations as in Theorem \ref{thm normal forms of power elements (original main thm1)}, and let $x, y\in G$. Then 

\begin{enumerate}
    \item $\ciw(x)$ is a cyclically irreducible word,  and as an element of  $G$, it is conjugate to $x$;
    \item $\ciw(X)=\overline{X}$ provided $\overline{X}$ itself is a cyclically irreducible word;
     \item $|\ciw(x)|=\swl(x)$;
     \item $|\ciw(x^r)|=r\cdot|\ciw(x)|$ and $\ciw(x^r)$ is a cyclic permutation of $\ciw(x)^r$ for every $r\geq 1$ ;
    \item $x$ and $y$ are conjugate $\Longleftrightarrow$ $\ciw(x)$ and $\ciw(y)$ are conjugate.
\end{enumerate}
\end{cor}

\begin{proof}
   Item (1) and (2) follow from the definition of $\ciw(x)$ and item (3) follows from 
   \begin{eqnarray*}
       \swl(x)&=&\lim_{k\to \infty}\dfrac{|x^k|}{k}=\lim_{k\to \infty}\dfrac{|\nf(x^k)|}{k}\\
       &=&\lim_{k\to \infty}\dfrac{2|\overline{Y}|+n+n'+(k-2)\cdot|\overline{W}|}{k}\\
       &=&|\overline{W}|=|\ciw(x)|.
   \end{eqnarray*}
   For item (4), by Theorem \ref{thm normal forms of power elements (original main thm1)}, for any $k\geq 2$, we have
    \begin{eqnarray*}
      \nf(x^{rk})&=&\overline{L_1\ciw(x^r)^{k-2}R_1}\\
      &=&\overline{L_2\ciw(x)^{rk-2}R_2}\\
      &=&\overline{L_2\ciw(x)^{2r-2}\ciw(x)^{r(k-2)}R_2}.
   \end{eqnarray*}
   Since $|\ciw(x^r)|=\swl(x^r)=r\cdot \swl(x)=|\ciw(x)^r|$, we obtain that $\ciw(x^r)$ is a cyclic permutation of $\ciw(x)^r$. Item (5) follows from item (1) clearly.
\end{proof}

\begin{exam}\label{exam calculate cic of b1...bk} Here, we give a simple example for calculating $\ciw(x)$ with $x:= b_1\cdots b_k$($1\leq k<4g$) for some $\overline{b_1\cdots b_{4g}}\in\RR$. For any $t\geq 1$, we have
\begin{equation*}
   \nf(x^t)=\left\{\begin{array}{ll}
        \overline{(b_1\cdots b_k)^t} & k<2g  \\
        \overline{(b_1\cdots b_{2g})^t} &k=2g,~ b_1\prec b_{2g}\\
        \overline{(b_{2g}\cdots b_{1})^t}& k=2g,~ b_1\succ b_{2g}\\
        \overline{(b_{2g}\cdots b_{k-2g+1})^t}& k>2g
    \end{array}\right..
\end{equation*}
Then, pick the maximal $(i,j)$ as in Eq. (\ref{def ciw}), we have 
\begin{equation*}
   \ciw(x)=\left\{\begin{array}{lll}
        \overline{b_1\cdots b_k} &\mbox{ with }(i,j)=(k,k)& k<2g  \\
        \overline{b_1\cdots b_{2g}} &\mbox{ with }(i,j)=(2g,2g)&k=2g,~ b_1\prec b_{2g}\\
        \overline{b_{2g}\cdots b_{1}}&\mbox{ with }(i,j)=(2g,2g)& k=2g,~ b_1\succ b_{2g}\\
        \overline{b_{2g}\cdots b_{k-2g+1}}&\mbox{ with }(i,j)=(4g-k,4g-k)& k>2g
    \end{array}\right.. 
\end{equation*}
\end{exam}

In Theorem \ref{thm normal forms of power elements (original main thm1)} above, by setting $k=2$, we arrive at the following immediate corollary. 

\begin{cor}\label{main cor1}\label{cor reduction of 2-nd power}
  With the same notations as in Theorem \ref{thm normal forms of power elements (original main thm1)},  then there exists a partition $\nf(x)=\overline{X_LX_R}$ such that
  $$\nf(x^2)=\overline{X_L\nf(X_RX_L)X_R}.$$
\end{cor}

\begin{rem}\label{rem 2-nd power not always producing cyclically irreducible form}
For every irreducible word $\overline{X}$, there always exists a partition $\overline{X}=\overline{X_LX_R}$ such that $\nf(X^2)=\overline{X_L\nf(X_RX_L)X_R}$ as in Corollary \ref{cor reduction of 2-nd power}. In most cases, $\nf(X_RX_L)$ is cyclically irreducible. Consequently, by Lemma \ref{lem from 2-nd power to n-th power}, $\nf(\overline{X^k})=\overline{X_L(\nf(X_RX_L))^{k-1}X_R}$. However, in a few cases ( listed in Proposition \ref{prop classification of LLFR of X^2}(\ref{XX reducible}a, \ref{XX reducible}b, \ref{XX reducible}c)), there is no partition $\overline{X}=\overline{X_LX_R}$ such that $\nf(X_RX_L)$ is cyclically irreducible. Fortunately, there is always a desired cyclically irreducible subword of $\nf(X^3)$ for these special words. This is the reason why we need $k\geq 2$ in Proposition \ref{prop classification of LLFR of X^2}(\ref{XX reducible}a, \ref{XX reducible}b, \ref{XX reducible}c) while $k\geq 1$ in Proposition \ref{prop classification of LLFR of X^2}(\ref{XX irreducible}).
\end{rem}


Finally, we obtain the proof of Theorem \ref{main thm word length formulae} as follows.

\begin{proof}[\textbf{Proof of Theorem} \ref{main thm word length formulae}] 
For any $1\neq x\in G$, let $\nf(x)=\overline{YXY^{-1}}$ with $\overline{X}=\overline{x_1\cdots x_n}$ cyclically freely reduced. By Eq. (\ref{eq. nf(xk)}) in Theorem \ref{thm normal forms of power elements (original main thm1)}, we have
$$|x^2|=2|\overline{Y}|+n+n'>2|\overline{Y}|+n=|x|,$$ 
and for any $k\geq 2$, $$|x^k|=2|\overline{Y}|+n+n'+(k-2)\cdot|\overline{W}|=(k-1)(|x^2|-|x|)+|x|.$$ 
Therefore, 
$$\swl(x)=\lim_{k\to \infty}\frac{|x^k|}{k}=|x^2|-|x|.$$
\end{proof}

\section{Ideas and procedures of proof of Proposition \ref{prop classification of LLFR of X^2}}\label{sect. proof of key Proposition}

Considering the extensive discussions in the proof of  Proposition \ref{prop classification of LLFR of X^2}, we first introduce the ideas and procedures of our proof in the following. 

Given an element $1\neq x\in G$, we can always write its normal form as $\nf(x)=\overline{YXY^{-1}}$ with $\overline{X}=\overline{x_1\cdots x_n}\neq 1$ cyclically freely reduced and $\overline{Y}=\overline{y_1\cdots y_m}$ possibly empty. We will first reduce $\overline{(YXY^{-1})^2}$ to obtain $\nf(x^2)$. Then, in most cases, we can obtain $\nf(x^k)(k\geq 1)$ directly by Lemmas \ref{lem from 2-nd power to n-th power}, \ref{lem types of W} and \ref{lem types of W for lengh =2g+1 to 4g-2}; in the remaining few cases (Proposition \ref{prop classification of LLFR of X^2}(\ref{XX reducible}a, \ref{XX reducible}b, \ref{XX reducible}c)), we need to conduct a further analysis. 

Hence, the most important step is reducing $\overline{(YXY^{-1})^2}$ to its normal form $\nf(x^2)$. First, we can reduce $\overline{(YXY^{-1})^2}$ to $\overline{YXXY^{-1}}$ by $|\overline{Y}|$ times $S_{(1)}$:
$$\overline{(YXY^{-1})^2}\xrightarrow{|\overline{Y}|\cdot S_{(1)}}\overline{YXXY^{-1}}=\overline{Yx_1\cdots x_nx_1\cdots x_nY^{-1}}.$$

If $\overline{X^2}$ is irreducible, we have $$\nf(x^k)=\overline{YX^kY^{-1}},$$ for any $k\geq 1$ by Lemma \ref{lem from 2-nd power to n-th power}, that is, \textbf{Proposition \ref{prop classification of LLFR of X^2}(\ref{XX irreducible}) holds}. 

If $\overline{X^2}$ is reducible, then it must have a reducible subword $\overline{V}$ that contains $\overline{x_nx_1}$. By Lemma \ref{prop length classification of LLFR at the joint}, we shall divide the proof of \textbf{Proposition \ref{prop classification of LLFR of X^2}(\ref{XX reducible})} into five cases based on the length of the $\llfr$ $\overline{T}$ such that
$$\overline{x_nx_1}\subset\overline{T}\subset \overline{X^2}.$$
\begin{enumerate}
    \item $\overline{x_nx_1}$ is not a fractional relator;
    \item $|\overline{T}|=2g-1$;
    \item $|\overline{T}|=2g$;
    \item $|\overline{T}|\in\{2g+1,\dots, 4g-2\}$;
    \item $|\overline{T}|=4g-1$.
\end{enumerate}

\begin{proof}[\textbf{Proof of Proposition \ref{prop classification of LLFR of X^2}(\ref{XX reducible}) ($\overline{x_nx_1}$ is not a fractional relator)}]
Since $\overline{X^2}$ is reducible, there exists a reducible subword $\overline{V}$ of $\overline{X^2}$ that contains $\overline{x_nx_1}$ and is of types $S_{(2,2g+1)},S_{(3,t)}$ or $S_{(4,1)}$. In fact, $\overline{V}$ can only be of type $S_{(3,t)}(t\geq 2)$ because ($\overline{x_nx_1}$ and hence) $\overline{V}$ is not a fractional relator. Then  $$\overline{V}=\overline{b_1(b_2\cdots b_{2g})^tb_{2g+1}}\xrightarrow{S_{(3,t)}}\overline{(b_{2g}\cdots b_2)^{t}}, \quad \overline{x_{n}x_1}=\overline{b_{2g}b_2},$$
for some $\overline{b_1\cdots b_{4g}}\in\RR$, and
    $$\overline{YX}=\overline{\cdots b_1(b_2\cdots b_{2g})^{t_1}}(t_1\geq 1), ~~\overline{XY^{-1}}=\overline{(b_2\cdots b_{2g})^{t_2}b_{2g+1}\cdots}(t_2\geq 1), ~t_1+t_2=t.$$
 Since $\overline{YXY^{-1}}=\nf(x)$ is irreducible and
    \begin{equation}\label{eq. 4.1(1)1} b_{1}\notin\overline{(b_2\cdots b_{2g})^{t_2}b_{2g+1}}, \quad b_{2g+1}\notin\overline{b_1(b_2\cdots b_{2g})^{t_1}},
    \end{equation}
we have
\begin{equation}\label{length=0 condition 1}
		\overline{X}=\overline{(b_2\cdots b_{2g})^{t_2}b_{2g+1}x_{n_1}\cdots x_{n_2}b_1(b_2\cdots b_{2g})^{t_1}},
	\end{equation}
where $x_{n_1}\neq b_1,b_{2g+2}$ and $x_{n_2}\neq b_{4g}, b_{2g+1}$ for $n_1\leq n_2.$
It follows that
\begin{eqnarray}\label{length=0 word1}
   \overline{YX^2Y^{-1}}
   &=&\overline{Y(b_2\cdots b_{2g})^{t_2}b_{2g+1}x_{n_1}\cdots x_{n_2}Vx_{n_1}\cdots x_{n_2} b_1(b_2\cdots b_{2g})^{t_1}Y^{-1}}\nonumber\\
&\xrightarrow{S_{(3,t)}}&\overline{Y(b_2\cdots b_{2g})^{t_2}b_{2g+1}\underbrace{x_{n_1}\cdots x_{n_2}(b_{2g}\cdots b_2)^{t}}_{W}x_{n_1}\cdots x_{n_2}b_1(b_2\cdots b_{2g})^{t_1}Y^{-1}}\nonumber\\
     &=&\overline{Yx_1\cdots x_{n_1-1}Wx_{n_1}\cdots x_{n}Y^{-1}}\notag\\
   &=:&\overline{Z}.
\end{eqnarray} 
	Therefore, $\overline{Z}$ is freely reduced clearly.
    

    Below, let us prove that $\overline{Z}$ (and hence $\overline{W}$) is irreducible. Suppose, for the sake of contradiction, that it is reducible. By Lemma \ref{lem frequently used}(2), there exists a reducible subword $\overline{U}:=\overline{u_1\cdots u_l}~(\subset\overline{Z})$ of one of the forms
    $$\overline{\cdots x_{n_2-1} x_{n_2}b_{2g}}(\subset\overline{Yx_1\cdots x_{n_2}b_{2g}})~~\mbox{ or }~~\overline{b_2x_{n_1} x_{n_1+1}\cdots}(\subset \overline{b_2x_{n_1}\cdots x_nY^{-1}}).$$
    Without loss of generality, we assume $\overline{U}=\overline{\cdots x_{n_2-1} x_{n_2}b_{2g}}$. Then, we have the following three cases.
	
	Case (1). $\overline{U}$ is of type $S_{(2,2g+1)}$. Since $x_{n_2}\neq b_{2g+1}$, we have $x_{n_2}=b_{2g-1}$ and $$\overline{U}=\overline{\cdots x_{n_2-1}x_{n_2}b_{2g}}=\overline{b_{4g}b_1\cdots b_{2g-1}b_{2g}}.$$
	Thus, 
	\begin{eqnarray*}
		\overline{u_1\cdots u_{2g}x_{n_2+1}\cdots x_{n}}&=&\overline{b_{4g}(b_1\cdots b_{2g-1})b_{1}(b_2\cdots b_{2g})^{t_1}}\\\notag
		&=&\overline{b_{4g}(b_1\cdots b_{2g-1})^2b_{2g}(b_2\cdots b_{2g})^{t_1-1}}\subset\overline{YX},
	\end{eqnarray*}
	where $t_1\geq 1$. Thus, the irreducible word $\overline{YX}$ contains a reducible subword $\overline{b_{4g}(b_1\cdots b_{2g-1})^2b_{2g}}$, which is contradictory.
	
	Case (2). $\overline{U}$ is of type $S_{(3,t')}$. Since $x_{n_2}\neq b_{2g+1}$, we have $$\overline{U}=\overline{\cdots x_{n_2-1}x_{n_2}b_{2g}}=\overline{b_{4g}(b_1\cdots b_{2g-1})^{t'}b_{2g}}.$$
It leads to a contradiction similar to that of Case (1).

Case (3). $\overline{U}$ is of type $S_{(4,1)}$. Since $x_{n_2}\neq b_{2g+1}$, we have $x_{n_2}= b_{2g-1}$ and 
$$\overline{U}=\overline{\cdots x_{n_2-1} x_{n_2}b_{2g}}=\overline{b_1\cdots b_{2g-1}b_{2g}}.$$
However, $\overline{b_1\cdots b_{2g}}$ is not of type $S_{(4,1)}$ because $$\overline{x_{n_2+1}\cdots x_n}=\overline{b_1(b_2\cdots b_{2g})^{t_1}}\subset \overline{X}$$ is irreducible and thus $b_1\prec b_{2g}$. It's a contradiction.

We conclude that $\overline{Z}=\overline{Yx_1\cdots x_{n_1-1}Wx_{n_1}\cdots x_{n}Y^{-1}}$ in Eq. (\ref{length=0 word1}) does not contain a subword of types $S_{(1)},S_{(2,2g+1)},S_{(3)}$ or $S_{(4,1)}$. Consequently, both $\overline{Z}$ and its subword $\overline{W}$ are  irreducible. Therefore,
$$\overline W=\overline{x_{n_1}\cdots x_{n_2}(b_{2g}\cdots b_2)^t}(t\geq 2)=\nf(x_{n_1}\cdots x_{n}x_1\cdots x_{n_1-1}).$$ 
Then $\overline{W^2}$ is irreducible by Lemma \ref{lem types of W}(1). According to Lemma \ref{lem from 2-nd power to n-th power}, for every $k\geq 1$, 
\begin{eqnarray}\label{length=0 word2}
    \overline{Yx_1\cdots x_{n_1-1}W^{k-1}x_{n_1}\cdots x_{n}Y^{-1}}
\end{eqnarray} is again irreducible, which is exactly $\nf(x^k)$.
\end{proof}

\begin{rem}
    In the above proof, we proved that the freely reduced word $\overline{Z}$ in Eq. (\ref{length=0 word1}) is irreducible by assuming it contains a reducible subword of types $S_{(2,2g+1)},S_{(3)}$ or $S_{(4,1)}$ and deriving a contradiction in each case. This kind of procedure will frequently appear in the subsequent discussions. Moreover, there will always be some conditions similar to Eqs. (\ref{eq. 4.1(1)1})(\ref{length=0 condition 1}) in the later proofs. To streamline the proofs, we may, on occasion, omit these procedures, and readers are expected to verify them independently. 
\end{rem}

\begin{proof}[\textbf{Proof of Proposition \ref{prop classification of LLFR of X^2}(\ref{XX reducible}) ($|\overline{T}|=2g-1$)}]
Since $\overline{X^2}$ is reducible, there exists a reducible subword $\overline{V}\subset \overline{X^2}$ containing $\overline{x_nx_1}$. Moreover, since the $\llfr$ $\overline{T}$ has length $=2g-1$, $\overline{V}$ can only be of type $S_{(3,t)}(t\geq 3)$ and
\begin{eqnarray*}
    \overline{YX}&=&\overline{\cdots b_1(b_2\cdots b_{2g})^{t_1}b_2\cdots b_{r+1}}(t_1\geq 1),\\
\overline{XY^{-1}}&=&\overline{b_{r+2}\cdots b_{2g}(b_2\cdots b_{2g})^{t_2}b_{2g+1}\cdots}(t_2\geq 1),
\end{eqnarray*}
 for some $\overline{b_1\cdots b_{4g}}\in\RR$. Since $\overline{YXY^{-1}}$ is irreducible and $$b_1\notin \overline{b_{r+2}\cdots b_{2g}(b_2\cdots b_{2g})^{t_2}b_{2g+1}}, \quad b_{2g+1}\notin\overline{b_1(b_2\cdots b_{2g})^{t_1}b_2\cdots b_{r+1}},$$ 
we have 
\begin{equation}\label{X in |LLFR|=2g-1}
    \overline{X}=\overline{b_{r+2}\cdots b_{2g}(b_2\cdots b_{2g})^{t_2}b_{2g+1} x_{n_1}\cdots x_{n_2} b_1(b_2\cdots b_{2g})^{t_1}b_2\cdots b_{r+1}},
\end{equation}
where $\overline{x_{n_1}\cdots x_{n_2}}\neq 1$. Then, there is a similar term as in Eq. (\ref{length=0 word1}): 
	\begin{eqnarray*}
		\overline{YX^2Y^{-1}}&=&\overline{Yx_1\cdots x_{n_2}b_1(b_2\cdots b_{2g})^{t_1}b_2\cdots b_{r+1}b_{r+2}\cdots b_{2g}(b_2\cdots b_{2g})^{t_2}b_{2g+1}x_{n_1}\cdots x_nY^{-1}}\\
	&\xrightarrow{S_{(3,t)}}&\overline{Yx_1\cdots x_{n_2}(b_{2g}\cdots b_{2})^tx_{n_1}\cdots x_{n}Y^{-1}}\notag\\
    &=&\overline{Yx_1\cdots x_{n_1-1}Wx_{n_1}\cdots x_{n}Y^{-1}}\notag\\
    &=:&\overline{Z},
	\end{eqnarray*}
	where $\overline{x_nx_1}=\overline{b_{r+1}b_{r+2}}$, $t=t_1+t_2+1$ for $t_1,t_2\geq 1$, and
$$\overline W=\overline{x_{n_1}\cdots x_{n_2}(b_{2g}\cdots b_2)^t}(t\geq 3).$$
    Then after an argument nearly identical to that following Eq. (\ref{length=0 word1}), we can prove that $\overline{Z}$ and $\overline{W}$ are irreducible, and then obtain the same normal form of $x^k~(k\geq 1)$ as in Eq. (\ref{length=0 word2}):
$$\nf(x^k)= \overline{Yx_1\cdots x_{n_1-1}W^{k-1}x_{n_1}\cdots x_{n}Y^{-1}},$$
where $\overline W=\overline{x_{n_1}\cdots x_{n_2}(b_{2g}\cdots b_2)^t}=\nf(x_{n_1}\cdots x_nx_1\cdots x_{n_1-1}),~(t\geq 3)$.
\end{proof}

For the readability and completeness of this paper, we relegate the detailed \textbf{Proof of Proposition} \ref{prop classification of LLFR of X^2}(\ref{XX reducible}) ($|\overline{T}|\in\{2g,\dots,4g-1\}$) to Appendix \ref{sect appendix for some detailed proofs of props} and merely outline the proof strategy below. 

In \textbf{Proof of Proposition} \ref{prop classification of LLFR of X^2}(\ref{XX reducible}) ($|\overline{T}|\in \{2g,2g+1,\dots,4g-2\}$) (see Appendix \ref{Appendix A1} and \ref{Appendix A2}), we will conduct two discussions analogous to the proof of Proposition \ref{prop classification of LLFR of X^2}(\ref{XX reducible}) ($\overline{x_nx_1}$ is not a fractional relator). The outline of the discussion process is as follows. (See Figure \ref{fig: reduction of most cases} for the reduction processes.)

\begin{enumerate}
    \item[Step 1.] According to the length of $\llfr$, we use different $S_{(i)}$'s to reduce $\overline{YX^2Y^{-1}}$. Whether $S_{(2,k)}$ (with maximal $k$) and $S_{(3)}$, or $S_{(4)}$ (not reducible by $S_{(2)}$ and $S_{(3)}$ first) is used first,  there will be no subword of type $S_{(1)}$. These two potential reductions are the step ``replace $\overline{A}$ with $\overline{B}$'' in Figure \ref{fig: reduction of most cases}.
    
    \item[Step 2.] For the junctions of $\overline{X_1}$ and $\overline{B}$, and of $\overline{B}$ and $\overline{X_2}$, we can respectively revisit the above-mentioned discussion and once again perform the reduction process. These potential reductions correspond to ``replace $\overline{C}$ with $\overline{E}$'' and ``replace $\overline{D}$ with $\overline{F}$'' in Figure \ref{fig: reduction of most cases}.

    \item[Step 3.] After the above three potential reductions, we will obtain $\nf(x^2)$ through a straightforward discussion. Then, we may derive the representative of $\nf(x^k)$ by making use of $\nf(x^2)$ and the prepared cyclically irreducible words in Section \ref{sect cyclically irreducible words}. In particular, there still exists a special case as described in Proposition \ref{prop classification of LLFR of X^2}(\ref{XX reducible}a).  
\end{enumerate}

\begin{figure}[ht]
    \centering
    \includegraphics[width=1\linewidth]{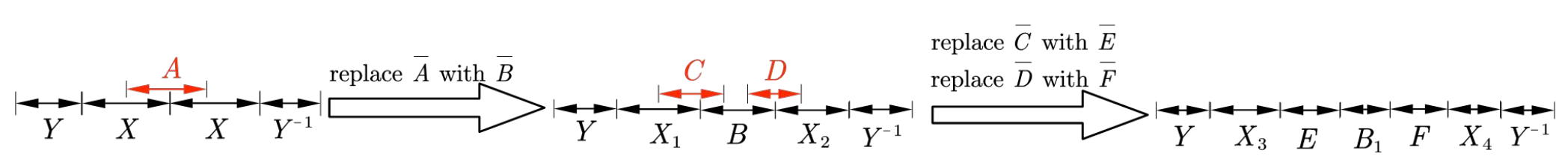}
    \caption{In Proof of Proposition \ref{prop classification of LLFR of X^2}(\ref{XX reducible}) ($\overline{x_nx_1}$ is not a fractional relator, $|\overline{T}|\in\{2g-1,\dots,4g-2\}$), the reduction of $\overline{YX^2Y^{-1}}$ always concludes  after \textbf{at most} three operations at different positions, where each reduction corresponds to one $S_{(i)}$.  It should be noted that the reducible subwords $\overline{C}$ and $\overline{D}$ do not exist in Proposition \ref{prop classification of LLFR of X^2}(\ref{XX reducible}) ($\overline{x_nx_1}$ is not a fractional relator, $|\overline{T}|=2g-1$).}
    \label{fig: reduction of most cases}
\end{figure}

In \textbf{Proof of Proposition} \ref{prop classification of LLFR of X^2}(\ref{XX reducible}) ($|\overline{T}|=4g-1$), the discussion remains similar yet more complicated. Therefore, we also give the sketch of reductions and place the detailed proof in Appendix \ref{Appendix A3}.  (See Figure \ref{fig: reduction of 4g-1} for the reduction processes.)

\begin{enumerate}
    \item[Step 1.] The first reduction is always $S_{(2,4g-1)}$. Then, we may repeat this reduction finitely many times until there is no such subword of type $S_{(2,4g-1)}$. These reductions correspond to the first two rows in Figure \ref{fig: reduction of 4g-1}.
    \item[Step 2.] Then, there will be at most three potential reductions (replacing $\overline{C}$ with $\overline{D}$, replacing $\overline{E}$ with $\overline{H}$ and replacing $\overline{F}$ with $\overline{K}$), which is consistent with what is stated in Proofs of Proposition \ref{prop classification of LLFR of X^2}(\ref{XX reducible}) ($\overline{x_nx_1}$ is not a fractional relator, $|\overline{T}|\in\{2g-1,\dots,4g-2\}$). These reductions correspond to the last row in  Figure \ref{fig: reduction of 4g-1}.
    \item[Step 3.] After the above reductions, we will obtain $\nf(x^2)$ through a complicated discussion. Then, we can derive the presentation of $\nf(x^k)$ by making use of $\nf(x^2)$ and Proposition \ref{prop classification of LLFR of X^2}(\ref{XX reducible}) ($\overline{x_nx_1}$ is not a fractional relator, $|\overline{T}|\in\{2g-1,\dots,4g-2\}$). In particular, there still exist two special cases (as described in Proposition\ref{prop classification of LLFR of X^2}(\ref{XX reducible}b, \ref{XX reducible}c)), which originates from Proposition\ref{prop classification of LLFR of X^2}(\ref{XX reducible}a).
\end{enumerate}

\begin{figure}[ht]
\centering
    \includegraphics[width=1\linewidth]{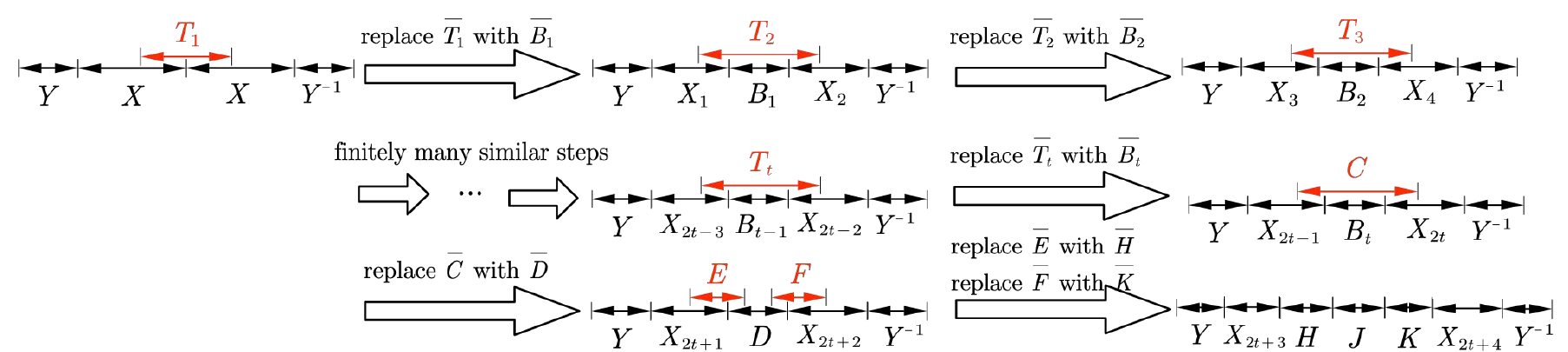}
\caption{In Proof of Proposition \ref{prop classification of LLFR of X^2}(\ref{XX reducible}) ($|\overline{T}|=4g-1$), the reduction of $\overline{YX^2Y^{-1}}$ always contains two steps:  first finitely many reductions of $S_{(2,4g-1)}$ and then some potential reductions as in Proofs of Proposition \ref{prop classification of LLFR of X^2}(\ref{XX reducible}) ($\overline{x_nx_1}$ is not a fractional relator, $|\overline{T}|\in\{2g-1,\dots,4g-2\}$).}
    \label{fig: reduction of 4g-1}
\end{figure}

\section{Classification of reductions}\label{sect classification of reductions}

In this section, we shall classify all potential reductions of any two irreducible words. Let $G,\gs,\W(\gs)$ be defined as before.

Let $\overline{W_1}=\overline{x_1\cdots x_m}$ and $\overline{W_2}=\overline{y_1\cdots y_n}$ be two nonempty irreducible words in $\W(\gs)$ with $x_m\neq y_1^{-1}$. Clearly, the reductions of $\overline{X^2}$ ($\overline{X}$ is an irreducible and cyclically freely reduced word) that we discussed in Section \ref{sect proof of main thms} are special cases of those of $\overline{W_1W_2}$. Therefore, the reduction cases of $\overline{W_1W_2}$ cover those of $\overline{X^2}$. In fact, we will prove that the converse also holds, see Lemma \ref{thm the reduction of W_1W_2 = X^2's}. This result means that we can summarize all the reductions of $\overline{X^2}$ directly to get all the reductions of $\overline{W_1W_2}$. First, we need a standard definition for reduction.

\begin{defn}[Reducing-subword Pair]\label{def for reduction}
Let $\overline{W_1}, \overline{W_2}\in\W(\gs)$ be two irreducible words. The \emph{reducing-subword pair} of 
$(\overline{W_1},\overline{W_2})$ is defined as
$$\rsp(\overline{W_1},\overline{W_2}):=\min\{(\overline{C_1},\overline{C_2})\mid \overline{W_1}=\overline{AC_1}, \overline{W_2}=\overline{C_2B}, \nf(W_1W_2)=\overline{AC'B}\},$$
where ``$\min$'' is in the sense of the length-lexicographical order on $\W(\gs)\times\W(\gs)$ (see Definition \ref{def order in word set}).
\end{defn}

Note that the reducing-subword pair $\rsp(\overline{W_1},\overline{W_2})$ is unique. Moreover, once $\nf(W_1W_2)$ is known, it can be readily determined as follows: select the longest subword $\overline{A}$ of $\overline{W_1}$ such that $\overline{W_1}=\overline{AC_1}$ and $\nf(W_1W_2)=\overline{AD}$; then, pick the longest subword $\overline{B}$ of $\overline{W_2}$ such that $\overline{W_2}=\overline{C_2B}$ and $\nf(W_1W_2)=\overline{AC'B}$.
We now obtain $$\rsp(\overline{W_1},\overline{W_2})=(\overline{C_1},\overline{C_2}).$$

Let $\overline{X}$ be an arbitrary nonempty irreducible and cyclically freely reduced word in $\W(\gs)$. Now, we will summarize all the reductions and $\rsp(\overline{X},\overline{X})$ by checking the proofs of Proposition \ref{prop classification of LLFR of X^2}. Thus, we still classify by the length of the $\llfr$ at the junction and \textbf{omit all the trivial reductions} $\overline{XX^{-1}}$. For the sake of reading, we place these tables of the classification in Appendix \ref{sect appendix tables for reductions}.

Before stating the main result of this section, we list several observations.

\begin{lem}\label{lem observations of reductions}
 Let $\overline{X}\in \W(\gs)$ be irreducible and cyclically freely reduced with $\rsp(\overline{X},\overline{X})=(\overline{C_1},\overline{C_2})$. Then
    \begin{enumerate}
    \item 
   for any $b\in \gs^{\pm}$, we always have $\overline{b^2}\not\subset\overline{C_1C_2}$;
 
    \item if $|\overline{C_1C_2}|\geq 4g+4(2g-1)t$ with $t\geq 2$, then there must be a subword $\overline{(b_1\cdots b_{2g-1})^{t}}\subset\overline{C_1C_2}$ for some $\overline{b_1\cdots b_{4g}}\in\RR$.
    \end{enumerate}
\end{lem}
\begin{proof}
(1) This result holds in all the cases of Tables \ref{tab: reduction class 1}, \ref{tab: reduction class 2g-1}, \ref{tab: reduction class 2g}, \ref{tab: reduction class 2g+1 to 4g-2} in Appendix \ref{sect appendix tables for reductions} directly. In Table \ref{tab: reduction class 4g-1}, note that $x_{p_2}\neq b_{2g+2}$ and $ x_{q_1}\neq b_{2g}$, hence the result also holds. 

Here, we give a proof of the parameter declaration in Table \ref{tab: reduction class 4g-1}:
\begin{enumerate}
    \item[i)] $\overline{x_{p_1}\cdots x_{p_2}}=\overline{x_{q_1}\cdots x_{q_2}}=1$, or
    \item[ii)] $   
        (\overline{x_{p_1}\cdots x_{p_2}}, ~\overline{b_1x_{q_1}\cdots x_{q_2}})$ is a reducing-subword pair in Tables \ref{tab: reduction class 1}, \ref{tab: reduction class 2g-1}, \ref{tab: reduction class 2g}, \ref{tab: reduction class 2g+1 to 4g-2} with $x_{p_2}\in\{ b_{4g}, b_{2g-1}\}$ and $x_{q_1}\in\{1, b_2,b_{2g+3}\}$.
\end{enumerate}

It suffices to consider the following one of the two symmetric cases (see Eqs. (\ref{Eq X with maximal t for LLFR=4g-1})(\ref{|LLFR|=4g-1, condition for n_5,n_6})):
  $$\overline{X}=\overline{x_1\cdots x_n}=\overline{b_1(b_2\cdots b_{2g})^tx_{n_5}\cdots x_{n_6}(b_{2g+2}\cdots b_{4g})^t}\quad\mbox{with }t \mbox{ maximal},$$
  $$ b_1\prec b_{2g},\qquad x_{n_5}\neq b_{2g+1},b_{4g},\qquad x_{n_6}\neq b_{2g+1},b_2.$$
  Then $\overline{X_0}=\overline{b_1x_{n_5}\cdots x_{n_6}}$ (see Eq. (\ref{Eq. A.3 X0})) is irreducible and cyclically freely reduced. Moreover, $\overline{x_{n_6}b_1}$ is not a $2$-fractional relator or the $\llfr$ $\overline{T'}\subset\overline{X_0^2}$ containing $\overline{x_{n_6}b_1}$ has length $\leq 4g-2$ (i.e., $\overline{x_{n_6}b_1}$ is a $2$-fractional relator).
  If $\overline{X_0^2}$ is irreducible, then by Claim \ref{claim 3} in Appendix \ref{Appendix A3}, we have 
  \begin{equation}\label{Eq 0 for 4g-1 remark}
 \rsp(\overline{X}, \overline{X})=(\overline{(b_{2g+2}\cdots b_{4g})^t}, ~\overline{b_1(b_2\cdots b_{2g})^t}).
  \end{equation}
  
  Now we suppose $\overline{X_0^2}=\overline{b_1x_{n_5}\cdots x_{n_6}b_1x_{n_5}\cdots x_{n_6}}$ is reducible. We have two cases.
  
  Case (a). $\overline{x_{n_6}b_1}$ is not a $2$-fractional relator. Since $x_{n_5}\neq b_{4g}$, we have 
  $$\overline{\cdots x_{n_6-1}x_{n_6}b_1x_{n_5}\cdots }=\overline{\underbrace{ b_{4g}(b_1\cdots b_{2g-1})^{t_1}}_{\cdots x_{n_6-1}x_{n_6}}\underbrace{(b_1\cdots b_{2g-1})^{t_2}b_{2g}}_{b_1x_{n_5}\cdots}} ~\mbox{ with }~x_{n_5}=b_2, x_{n_6}=b_{2g-1},$$
and the reducing-subword pair of $(\overline{X},\overline{X})$ is of the form:
\begin{equation}\label{Eq 1 for 4g-1 remark}
    \rsp(\overline{X},\overline{X})=(\overline{\cdots b_{2g-1}(b_{2g+2}\cdots b_{4g})^t}, ~\overline{b_1(b_2\cdots b_{2g})^tb_2\cdots }).
\end{equation}

  Case (b). $\overline{x_{n_6}b_1}$ is a $2$-fractional relator. Then we have $x_{n_6}=b_{4g}$. If $x_{n_5}\in \overline{T'}$, we have $x_{n_5}=b_2$ and 
  \begin{equation}\label{Eq 2 for 4g-1 remark}
      \rsp(\overline{X},\overline{X})=(\overline{\cdots b_{4g}(b_{2g+2}\cdots b_{4g})^t}, ~\overline{b_1(b_2\cdots b_{2g})^tb_2\cdots}).
  \end{equation}
  If $x_{n_5}\notin \overline{T'}$, then $x_{n_5}\neq b_2$. Furthermore, if $\rsp(\overline{X_0},\overline{X_0})=(\overline{\cdots x_{n_6-1}x_{n_6}}, ~\overline{b_1})$, we have
  \begin{equation}\label{Eq 3 for 4g-1 remark}
     \rsp(\overline{X}, \overline{X})=(\overline{\cdots b_{4g}(b_{2g+2}\cdots b_{4g})^t}, ~\overline{b_1(b_2\cdots b_{2g})^t}).
  \end{equation}
Otherwise, we have $$\rsp(\overline{X_0},\overline{X_0})=(\overline{\cdots x_{n_6-1}x_{n_6}}, ~\overline{b_1x_{n_5}\cdots})=(\overline{\cdots b_{4g-1}b_{4g}}, ~\overline{b_1x_{n_5}\cdots}).$$ Note $x_{n_5}\neq b_2,b_{2g+1},b_{4g}$. We check the tables (except Table \ref{tab: reduction class 4g-1}) in Appendix \ref{sect appendix tables for reductions} and obtain $x_{n_5}=b_{2g+3}$. Then,
\begin{equation}\label{Eq 4 for 4g-1 remark}
    \rsp(\overline{X},\overline{X})=(\overline{\cdots b_{4g}(b_{2g+2}\cdots b_{4g})^t}, ~\overline{b_1(b_2\cdots b_{2g})^tb_{2g+3}\cdots}).
\end{equation}

After concluding Eqs. (\ref{Eq 0 for 4g-1 remark})(\ref{Eq 1 for 4g-1 remark})(\ref{Eq 2 for 4g-1 remark})(\ref{Eq 3 for 4g-1 remark})(\ref{Eq 4 for 4g-1 remark}), we have obtained the parameter declaration.

(2) By calculating the length of $\overline{C_1C_2}$ in Appendix \ref{sect appendix tables for reductions} item by item, we can easily obtain $$|\overline{C_1C_2}|< 4g+4(2g-1)t,$$ provided that $\overline{C_1C_2}$ has no subword of form $\overline{(b_1\cdots b_{2g-1})^{t}}$  ( $t\geq 2$) for any $\overline{b_1\cdots b_{4g}}\in\RR$. Therefore,  the conclusion holds.
\end{proof}

\begin{lem}\label{thm the reduction of W_1W_2 = X^2's} 
    For any two irreducible words $\overline{W_1}, \overline{W_2}\in\W(\gs)$ with $\overline{W_1W_2}$ freely reduced,   there is an irreducible and cyclically freely reduced word $\overline{X}\in\W(\gs)$ such that  $$\rsp(\overline{X},\overline{X})=\rsp(\overline{W_1},\overline{W_2}).$$
\end{lem}

\begin{proof}
We omit all the trivial cases. Denote 
$\rsp(\overline{W_1},\overline{W_2})=(\overline{C_1},\overline{C_2})$ for $$\overline{C_1}=\overline{x_1\cdots x_n}\neq 1, \quad \overline{C_2}=\overline{y_1\cdots y_m}\neq 1$$ with
$x_n\neq y_1^{-1}$. For a letter  $b\in\gs\cup \gs^{-1}\backslash\{x_1^{-1},y_{m}^{-1}\}$, denote $$\overline{X}=\overline{C_2y_{m}^2b^2x_1^2C_1}.$$ 
Since no word of type $S_{(i)}$ contains $\overline{b^2}$ for any $b\in\gs^{\pm}$ and $x_n\neq y_1^{-1}$, we have $\overline{X}$ irreducible and cyclically freely reduced.
Suppose $\rsp(\overline{X},\overline{X})=(\overline{C_1'},\overline{C_2'})$. Clearly, $\overline{C_1'}$ ends with $\overline{C_1}$ and $\overline{C_2'}$ begins with $\overline{C_2}$. By Lemma \ref{lem observations of reductions}(1), $\overline{C_1'}$ does not contain $\overline{x_1^2}$ and $\overline{C_2'}$ does not contain $\overline{y_m^2}$. Therefore, $\overline{C_i'}=\overline{C_i}$. The proof is complete.
\end{proof}

Lemma \ref{thm the reduction of W_1W_2 = X^2's}  shows that the classification in Appendix \ref{sect appendix tables for reductions} is equally applicable to every pair of irreducible words $(\overline{W_1},\overline{W_2})$ with $\overline{W_1W_2}$ freely reduced. Hence it is a complete classification, and we have the following lemma by Lemma \ref{lem observations of reductions}(2) directly.

\begin{lem}\label{lem reducing formulae}
    Let $\overline{W_1}$ and $\overline{W_2}$ be two irreducible words with $\overline{W_1W_2}$ freely reduced and $\rsp(\overline{W_1},\overline{W_2})=(\overline{C_1},\overline{C_2})$. If $|\overline{C_1C_2}|\geq 4g+4(2g-1)t$ with $t\geq 2$, then there must be a subword $$\overline{(b_1\cdots b_{2g-1})^{t}}\subset\overline{C_1C_2},$$ for some $\overline{b_1\cdots b_{4g}}\in \RR$.
\end{lem}

\section{Proof of Theorem \ref{main thm3 conj. class} and some algorithms}\label{sect normal form of conjugacy classes}

In this section, let $G,\gs,\W(\gs), \RR$ be defined as before. We first prove Theorem \ref{main thm3 conj. class}, then provide some algorithms for root-finding and the conjugacy problem.

\subsection{Proof of Theorem \ref{main thm3 conj. class}}
Before our proof, we give an observation as follows.

\begin{lem}\label{prop observations of b_1 cdots b_{2g-1} }
  Suppose $\overline{b_1\cdots b_{4g}}, ~\overline{b_1'\cdots b_{4g}'}\in \RR$. If $b_1\cdots b_{k}$ and $b_1'\cdots b_{k}'$ ($k\leq 2g$) are conjugate in the surface group $G$, then $\overline{b_1\cdots b_{k}}=\overline{b_1'\cdots b_{k}'}$ or $\overline{b_1\cdots b_{k}}=\overline{b_{k}'\cdots b_1'}$.  
\end{lem}

\begin{proof}Since $b_1\cdots b_{k}$ and  $b_1'\cdots b_{k}'$ are conjugate in $G$, we have
   $$f(b_1\cdots b_{k})=f(b_1'\cdots b_{k}'),$$
   where $$f:G\to \Z^{2g}, ~c_i\mapsto e_i=(0,\dots,0,1,0,\dots,0)$$
is the abelianization of $G$. Furthermore, note that $b_i\neq b_j^{\pm 1}$ and  $b_i'\neq b_j'^{\pm 1}$ for any $1\leq i<j\leq k\leq 2g$,  we have  
   $$\sum_{i=1}^{k}f(b_i)=\sum_{i=1}^{k}f(b_i').$$
Therefore, $\{b_1,\dots,b_{k}\}=\{b_1',\dots,b_{k}'\}$ and the conclusion holds.
\end{proof}

Recall for any $x\in G$, there is a unique cyclically irreducible core $\ciw(x)$ conjugate to $x$, see Eq. (\ref{def ciw}) and Corollary \ref{cor cyclically irreducible word notation}. Now we can prove Theorem \ref{main thm3 conj. class}.

\mainthmconjclass*

\begin{proof}
Item (2) directly follows from item (1). Therefore, we prove item (1) in the following.

``$\Leftarrow$''. Note that $x_1\cdots x_m$ and $x_{i+1}\cdots x_mx_1\cdots x_i $ are conjugate, $b_{4g}=b_{2g}^{-1}$ and 
$$b_{4g}(b_1\cdots b_{2g-1})^tb_{2g}=(b_{2g-1}\cdots b_1)^t$$ for every $\overline{b_1\cdots b_{4g}}\in \RR$. We can easily deduce that $x$ and $y$ are conjugate.

``$\Rightarrow$''.  Since $x$ and $y$ are conjugate, $\ciw(x)$ and $\ciw(y)$ are also conjugate in $G$. Notice that conjugate elements have the same translation number by \cite[Lemma 6.3]{GS91ann}. Hence
by Corollary \ref{cor cyclically irreducible word notation}, we have the translation number
$$\swl(x)=|\ciw(x)|=\swl(y)=|\ciw(y)|.$$
We now suppose the cyclically irreducible cores
$$\ciw(x)=\overline{X}=\overline{x_1\cdots x_m},\qquad \ciw(y)=\overline{Y}=\overline{y_1\cdots y_m},$$
and suppose $ZXZ^{-1}=Y$ as group elements with $\overline{Z}$ irreducible. In addition, we can suppose $\overline{ZXZ^{-1}}$ is freely reduced. (Otherwise, we can replace $\overline{X}$ and $\overline{Z}$ with a cyclic permutation $\overline{X'}$ of $\overline{X}$ and a subword $\overline{Z'}$ of $\overline{Z}$, respectively. Then $Z'X'Z'^{-1}=Y,~\overline{X'}=\overline{x_{i+1}\cdots x_nx_1\cdots x_i}$ and $\overline{Z'X'Z'^{-1}}$ is freely reduced.) 

For any $\ell\geq 1$, let us consider the reducing-subword pairs:
\begin{eqnarray*}
    \rsp(\overline{Z},\overline{X^\ell}):=(\overline{Z_\ell},\overline{X_\ell}), \qquad \rsp(\overline{Y^\ell},\overline{Z}):=(\overline{Y_\ell},\overline{Z_\ell'}).
\end{eqnarray*}

Case (1). As $\ell \to \infty$, the lengths of $\overline{X_\ell}$ and $\overline{Y_\ell}$ are uniformly bounded. 
Then there is a sufficiently large positive integer $\ell_0$ such that
\begin{eqnarray*}
\rsp(\overline{Z},\overline{X^{\ell_0+k}})=\rsp(\overline{Z},\overline{X^{\ell_0}})= (\overline{Z_0},\overline{X_0})&\quad\mbox{with} &\overline{Z}=\overline{Z_1 Z_0}, ~\overline{X^{\ell_0}}=\overline{X_0X_1}, \\
    \rsp(\overline{Y^{\ell_0+k}},\overline{Z})=\rsp(\overline{Y^{\ell_0}},\overline{Z})=(\overline{Y_0},\overline{Z_0'})&\quad\mbox{with}&\overline{Y^{\ell_0}}=\overline{Y_1Y_0},~\overline{Z}=\overline{Z_0'Z_1'},   
\end{eqnarray*}
for every $k\geq 0$, and hence we have
\begin{equation}
\begin{array}{lllll}
\nf(ZX^{\ell_0+k})&=&\overline{Z_1\nf(Z_0X_0)X_1X^{k}}&=&\overline{\nf(ZX^{\ell_0})X^{k}},\label{Eq. finite rsp X}\\
\nf(Y^{\ell_0+k}Z)&=&\overline{Y^{k}Y_1\nf(Y_0Z'_0)Z'_1}&=&\overline{Y^{k}\nf(Y^{\ell_0}Z)}.
\end{array}
\end{equation} 
Moreover, note that $ZX^{\ell_0+k}=Y^{\ell_0+k}Z$, and the lengths $|\nf(ZX^{\ell_0})|$ and $|\nf(Y^{\ell_0}Z)|$ are bounded and independent of $k$, then for sufficiently large $k$, we have
\begin{eqnarray*}
&&\overline{\nf(ZX^{\ell_0})\underbrace{(x_1\cdots x_m)( x_1\cdots x_m)(x_1\cdots x_m) \cdots (x_1\cdots x_m)(x_1\cdots x_m)(x_1\cdots x_m)}_{X^{k}}}\\
&=&\overline{\underbrace{(y_1\cdots y_m)(y_1\cdots y_m)(y_1\cdots y_m)\cdots (y_1\cdots y_m)(y_1\cdots y_m)(y_1\cdots y_m)}_{Y^{k}}\nf(Y^{\ell_0}Z)}.    
\end{eqnarray*}
It implies $\overline{Y}=\overline{x_{i+1}\cdots x_mx_1\cdots x_{i}}$ for some  $1\leq  i \leq m$.

Case (2). As $\ell\to\infty$, the lengths of
both $\overline{X_\ell}$ and $\overline{Y_\ell}$ are unbounded. 
Then we can pick a sufficiently large positive integer $\ell_0$ such that
$$|\overline{Z_{\ell_0}X_{\ell_0}}|>4g+8100(2g-1)(|\overline{X}|+|\overline{Z}|).$$
By Lemma \ref{lem reducing formulae}, there must be a subword $\overline{(b_1\cdots b_{2g-1})^t}\subset\overline{Z_{\ell_0}X_{\ell_0}}$ with $t=2025(|\overline{X}|+|\overline{Z}|)$ for some $\overline{b_1\cdots b_{4g}}\in\RR$. Then
\begin{eqnarray*}
\overline{Z_{\ell_0}X_{\ell_0}}&=&\overline{Z_{\ell_0}(x_1\cdots x_m)(x_1\cdots x_m)(x_1\cdots x_m)\cdots (x_1\cdots x_m)(x_1\cdots x_m)(x_1\cdots x_m)x_1\cdots x_i}\\
    &=&\overline{\cdots \underbrace{(b_1\cdots b_{2g-1})(b_1\cdots b_{2g-1})(b_1\cdots b_{2g-1})\cdots(b_1\cdots b_{2g-1})(b_1\cdots b_{2g-1})}_{(b_1\cdots b_{2g-1})^t}\cdots},
\end{eqnarray*}
which implies that $\overline{X}=\overline{x_1\cdots x_m}$ is a cyclic permutation of $\overline{(b_1\cdots b_{2g-1})^{t_1}}$ for some $t_1>0$. Similarly, we have $\overline{Y}$ is also a cyclic permutation of $\overline{(b_1'\cdots b_{2g-1}')^{t_2}}$ for some $\overline{b_1'\cdots b_{4g}'}\in\RR$. Since $|\overline{X}|=|\overline{Y}|$, we have $t_1=t_2$, and hence $(b_1\cdots b_{2g-1})^{t_1}$ and $(b_1'\cdots b_{2g-1}')^{t_1}$ are conjugate in the surface group $G$. It follows that $b_1\cdots b_{2g-1}$ and $b_1'\cdots b_{2g-1}'$ are also conjugate in $G$, by the fact that every element in a surface group has a unique $t_1$-th root. Then by Lemma  \ref{prop observations of b_1 cdots b_{2g-1} },  we have
$$\overline{b_1\cdots b_{2g-1}}=\overline{b_1'\cdots b_{2g-1}'}\qquad \mbox{or}\qquad \overline{b_1\cdots b_{2g-1}}=\overline{b_{2g-1}'\cdots b_1'}.$$
Therefore, Theorem \ref{main thm3 conj. class} holds in this case.

Case (3). As $\ell\to \infty$,  the length of  $\overline{X_\ell}$ is uniformly bounded but the length of $\overline{Y_\ell}$ is not. 
Then there is a sufficiently large $\ell_0$ such that Eq. (\ref{Eq. finite rsp X}) holds:
$$\nf(ZX^{\ell_0+k})=\overline{\nf(ZX^{\ell_0})X^{k}},$$
and $Y$ is a cyclic permutation of $\overline{(b_1'\cdots b_{2g-1}')^{t_2}}$ for some $\overline{b_1'\cdots b_{4g}'}\in\RR$ as in Case (2).  Denote $$\rsp(\nf(ZX^{\ell_0+k}), ~\overline{Z^{-1}})=\rsp(\overline{\nf(ZX^{\ell_0})X^k}, ~\overline{Z^{-1}})=(\overline{C_k},\overline{D_k}).$$ 
If the length of $\overline{C_k}$ is not uniformly bounded as $k\to \infty$, then $\overline{X}$ is a cyclic permutation of $\overline{(b_1\cdots b_{2g-1})^{t_1}}$ for some $\overline{b_1\cdots b_{4g}}\in \RR$. We have reduced the discussion to Case (2).
If the length of $\overline{C_k}$ is uniformly bounded as $k\to \infty$,
then there is a sufficiently large $k_0$ such that for every $n\geq 0$, we have $$\nf(ZX^{\ell_0+k_0+n}Z^{-1})=\overline{\nf(ZX^{\ell_0})X^n\nf(X^{k_0}Z^{-1})}=\overline{Y^{\ell_0+k_0+n}}.$$
Now by a same discussion as in Case (1), we have $\overline{Y}=\overline{x_{i+1}\cdots x_mx_1\cdots x_i}$ for some $1\leq i\leq m$. The discussion of Case (3) is complete.

Case (4). As $\ell\to \infty$, the length of $\overline{X_\ell}$ is unbounded but the length of $\overline{Y_\ell}$ is uniformly bounded. The argument is symmetric to that of Case (3).

Therefore, the proof of Theorem \ref{main thm3 conj. class} is complete.
\end{proof}

By combining Theorem \ref{main thm3 conj. class} with Example \ref{exam calculate cic of b1...bk}, we can immediately strengthen Lemma \ref{prop observations of b_1 cdots b_{2g-1} } to the following.

\begin{cor} 
Suppose $\overline{b_1\cdots b_{4g}}, ~\overline{b_1'\cdots b_{4g}'}\in \RR$. Then $b_1\cdots b_{k}$ and $b_1'\cdots b_{k}'$ ($k\leq 2g$) are conjugate in the surface group $G$ if and only if 
\begin{eqnarray*}
\overline{b_1\cdots b_{k}}=\left\{ 
\begin{array}{ll}
    \overline{b_1'\cdots b_{k}'} & k\neq 2g-1, 2g\\
    \overline{b_1'\cdots b_{k}'} ~\mbox{or} ~ \overline{b_{k}'\cdots b_1'} & k=2g-1, 2g
\end{array}\right..
\end{eqnarray*}
\end{cor}


\subsection{Conjugation algorithm and root-finding algorithm}
First, by Theorem \ref{thm GZ} and the $S$-set, we provide Algorithm \ref{algorithm: normal form of an element} for computing the normal form of any given word $x\in\W(\gs)$. In general, the complexity of computing the normal form of $x$ by using Gr\"{o}bner-Shirshov basis is not polynomial, but our case turns out to be simple. Suppose that the input word has length $n$. In Algorithm~\ref{algorithm: normal form of an element}, each letter is scanned once to compare with the leading term of some $S_{(1)}$, and is scanned twice  to compare with the leading term of some $S_{(i)}$ with $i>1$. Note that the letter $b_1$ of the leading term of any $S_{(i)}$ with $i>1$
has $4g$ choices. The total operation (scanning, comparing and replacing) is at most $3\times (n+2n\times 4g+2n\times 4g)$. Hence, the complexity of this algorithm is linear, i.e. $O(n)$.  

Then, according to Theorem \ref{thm normal forms of power elements (original main thm1)} and Theorem \ref{main thm3 conj. class}, we provide Algorithm \ref{algorithm: find normal form of conjugacy class} to compute the normal form of a conjugacy class (i.e., the minimal word in this conjugacy class, see Definition \ref{normal form}). Since two elements $x,y\in G$ are conjugate if and only if $\nf([x])=\nf([y])$, we can decide whether $x$ and $y$ are conjugate by comparing $\nf([x])$ and $\nf([y])$.

\begin{algorithm}
\rule{\textwidth}{0.2pt}
\SetAlgoLined
    
\SetKwInOut{Input}{Input}
\SetKwInOut{Output}{Output}

\Input{A word $\overline{X}\in\W(\gs)$.}
\Output{The normal form $\nf(X)$ of $X$.}

\begin{enumerate}

\item[Step 1.]
Look for all the subwords of $\overline{X}$ which are the leading term of $S_{(1)}$, and remove them from $\overline{X}$, and update $\overline{X}$.

\item[Step 2.]
Set $q=1$, the position of the first letter need to be considered.

\item[Step 3.] 
Starting from the $q$-letter, look for the left most subword $A$ of $\overline{X}$ such that $A$ is a leading term of some $S_{(i)}$ and $Ax$ is not the leading term of any $S_{(i)}$, where $x$ is the letter of $X$ next to $A$. Set $\overline{X}=\overline{B_LAB_R}$. IF there is no such  $\overline{A}$, THEN STOP.

\item[Step 4.] Replace $A$ with corresponding non-leading term $\overline{A'}=\overline{y_1\cdots y_m}$, and update $\overline{X} =\overline{B_LA'B_R}$. 

\item[Step 5.] Starting from the $q$-letter, i.e. the first letter of $\overline{A'}$, look for (from rigth to left) maximal subword $\overline{C}$ of $\overline{B_Ly_1}$ such that $\overline{C}$ is the leading term of some $S_{(i)}$. IF there is such $\overline{C}$, THEN replace $\overline{C}$ with corresponding non-leading term $\overline{C'}$.

\item[Step 6.] Set $q$ to be the position of the letter before $\overline{B_R}$.

\item[Step 7.] GOTO Step 3. 
\end{enumerate}
\rule{\textwidth}{0.2pt}
\caption{Algorithm for computing the normal form of a word.}\label{algorithm: normal form of an element}
\end{algorithm}

\begin{algorithm}[ht]
\rule{\textwidth}{0.2pt}
\SetAlgoLined
    
\SetKwInOut{Input}{Input}
\SetKwInOut{Output}{Output}

\Input{An irreducible word $\overline{X}\in\W(\gs)$, the normal form $\overline{X^{(2)}}:=\nf(X^2)$ of $X^2$ and the normal form $\overline{X^{(3)}}:=\nf(X^3)$ of $X^3$.}

\Output{The normal form $\overline{N}:= \nf([X])$ of the 
conjugacy class $[X]$ of $X$, and the conjugator $z$ from $N$ to $X$, i.e., $N=zXz^{-1}$ in the surface group $G$.}
\begin{enumerate}
    \item[Step 1.] Find the longest common initial subword $\overline{X_L}$  of  $\overline{X}$  and  $\overline{X^{(2)}}$.
    \item[Step 2.] Decompose $\overline{X}$, $\overline{X^{(2)}}$ and $\overline{X^{(3)}}$ into products $\overline{X_LX_R}$, $\overline{X_LX'X_R}$ and $\overline{X_LX''X_R}$.
  \item[Step 3.] Find the longest common initial subword $\overline{X'_L}$  of  $\overline{X'}$  and  $\overline{X''}$.
  \item[Step 4.] Decompose $\overline{X'}$  and  $\overline{X''}$  into products $\overline{X'_LX'_R}$ and $\overline{X'_L W X'_R}$. ($\overline{W}$ is actually $\ciw(X)$.)
\item[Step 5.] Let $\overline{W} = \overline{w_1\cdots w_{|\overline{W}|}}$. Pick $\overline{N}=\overline{w_k w_{k+1}\cdots w_{|\overline{W}|}w_1\cdots w_{k-1}}$ to be the minimal word among all rotations of $\overline{W}$, and pick $z$ to be $w_k\cdots w_{|\overline{W}|}X'_RX_R$.

\item[Step 6.] 
IF the word $\overline{W}$ has the form
$\overline{(b_{i+1}\cdots b_{2g-1}b_1\cdots b_{i})^t}$ for some $t$ and $i$ with $t\geq 1$ and $1\leq i\leq 2g-1$ and $\overline{b_1\cdots b_{4g}}\in\RR$ THEN
\begin{itemize}
    \item Pick $\overline{N'}$ to be the minimal word among the reverses of all rotations of $\overline{W}$.
 \item IF $\overline{N'}$ is smaller than $\overline{N}$ THEN
 \begin{itemize}
 \item Update $\overline{N}$ to be $\overline{N'} = \overline{w_j\cdots w_1w_{|\overline{W}|}\cdots w_{j+1}}$.
 \item Update $z$ to be $w_j\cdots w_1b_i\cdots b_1b_{4g}\cdots b_{2g+1+i}X'_RX_R$.
 \end{itemize}
\end{itemize}
\end{enumerate}
 \rule{\textwidth}{0.2pt}
\caption{Algorithm for computing the normal form of a conjugacy class. See \cite{Zh26} for the Wolfram Mathematica and Python files.}
\label{algorithm: find normal form of conjugacy class}
\end{algorithm}

Let us consider the complexity of Algorithm~\ref{algorithm: find normal form of conjugacy class}, depending on its input $n$:  the length of given word $X$. The first four steps are all comparison among words of length $n$, $2n$ and $3n$. The times of comparison of letters are at most $n+2n=3n$. Step 5 involves a comparison among $n$ words of length $n$. Hence, the times of comparison of letters are at most $n^2$. Note that the number of elements in $\RR$ is fixed, depends only on the genus $g$.
The complexity of Step 6 is similar to that of Step 5. Thus, the total complexity is about $O(n^2)$.  

In addition, we also provide Algorithm \ref{algorithm: root-finding} for computing the \emph{primitive root} (i.e. the root is not a proper power) of an arbitrary element. Clearly, the complexity of this algorithm is similar to that of Algorithm~\ref{algorithm: find normal form of conjugacy class}, and hence is about $O(n^2)$.

\begin{algorithm}[ht]
\rule{\textwidth}{0.4pt}
\SetAlgoLined

\SetKwInOut{Input}{Input}
\SetKwInOut{Output}{Output} 

\Input{An irreducible word $\overline{X}\in\W(\gs)$, the normal form $\overline{X^{(2)}}:=\nf(X^2)$ of $X^2$ and the normal form $\overline{X^{(3)}}:=\nf(X^3)$ of $X^3$.}

\Output{The primitive root $Y$ and its power $r$ with $X=Y^r$ in the surface group $G$.}
\begin{enumerate}
    \item[Step 1.] Find the longest common initial subword $\overline{X_L}$  of  $\overline{X}$  and  $\overline{X^{(2)}}$.
    \item[Step 2.] Decompose $\overline{X}$, $\overline{X^{(2)}}$ and $\overline{X^{(3)}}$ into products $\overline{X_LX_R}$, $\overline{X_LX'X_R}$ and $\overline{X_LX''X_R}$.
  \item[Step 3.] Find the longest common initial subword $\overline{X'_L}$  of  $\overline{X'}$  and  $\overline{X''}$.
  \item[Step 4.] Decompose $\overline{X'}$  and  $\overline{X''}$  into products $\overline{X'_LX'_R}$ and $\overline{X'_L W X'_R}$. ($\overline{W}$ is actually $\ciw(X)$.)

    \item[Step 5.] Let $\overline{W}:=\overline{w_1w_2\cdots w_n}$ with $w_i\in\gs^{\pm}$, and calculate all the positive divisors $\{d_j\mid j=1, 2, \ldots, m \}$ of $n$ with $1=d_1<d_2<\cdots< d_m=n$.
\item[Step 6.]
 Find the minimal $j$ in $\{1,2,\ldots, m\}$ such that $\overline{W}=\overline{(w_1\cdots w_{d_j})^{n/d_j}}$
\item[Step 7.] Pick $Y=X_LX'_L(w_1\cdots w_{d_j})(X'_L)^{-1}X_L^{-1}$ and pick $r=\frac{n}{d_j}$.
\end{enumerate}
\rule{\textwidth}{0.4pt}
\caption{Algorithm for root-finding.}  
\label{algorithm: root-finding}
\end{algorithm}

Finally, by combining Algorithm \ref{algorithm: find normal form of conjugacy class} and Algorithm \ref{algorithm: root-finding}, we obtain Corollary \ref{cor conj-power algo.} regarding the commensurability problem algorithm.
\corconjpoweralgo*

\begin{proof}
  We first use Algorithm \ref{algorithm: root-finding} to obtain the primitive roots $x_1, y_1$ of $x$ and $y$  such that $x_1^{r_1}=x$ and $y_1^{r_2}=y$, respectively. Note that $r_1, r_2\geq 1$ and
  \begin{eqnarray*}
 y^n=zx^mz^{-1}\Leftrightarrow y_1^{nr_2}=(zx_1z^{-1})^{mr_1}\Leftrightarrow y_1=\left\{ \begin{array}{ll}
    zx_1z^{-1} & nr_2=mr_1\\
    zx_1^{-1}z^{-1} & nr_2=-mr_1
\end{array}\right., 
  \end{eqnarray*}
 where the last ``=" holds because $x_1$ and $y_1$ are primitive roots. Then using Algorithm \ref{algorithm: find normal form of conjugacy class}, we can verify whether $y_1=zx_1^{\pm 1}z^{-1}$  for some $z\in G$. If the answer is yes, we can compute the conjugator $z$ and take $m=r_2/d$ and $n=\pm r_1/d$, where $d$ is the greatest common divisor of $r_1$ and $r_2$.  
\end{proof}

\subsection{Dehn's conjugation algorithm\label{subsect Dehn algorithm}} In the early 1900's, Dehn formulated three fundamental decision problems in group theory: the word problem, the conjugacy problem, and the isomorphism problem. Since then, these three problems have had a profound impact on shaping the entire field of group theory. He was also the first person to solve the word problem and the conjugacy problem for surface groups \cite{De11,De12,De87}. 
Now, we illustrate the differences between Dehn's work and our conclusions. Since he gave a more efficient algorithm in \cite{De12} than that in \cite{De11}, we introduce the algorithm in \cite{De12}.

Let $G=\pi_1(\Sigma_g)(g\geq 2)$ have a canonical presentation $\langle\gs_c\mid R_c\rangle$, where $\gs_c=\{a_1,a_2,\dots,a_{2g-1},a_{2g}\}$ and $R_c=[a_1,a_2]\cdots[a_{2g-1},a_{2g}]$. Following the convention of Definition \ref{fractional relator}, we denote by $\RR_c$ the set of cyclic permutations of $R_c$ and its inverse and call $\overline{b_1\cdots b_k}$ a \emph{$k$-fractional relator} if $\overline{b_1\cdots b_{4g}}\in\RR_c$ and $2\leq k\leq 4g$. For the word problem, Dehn presented the following theorem, which was later referred to as ``Dehn's algorithm''.
\begin{thm}[Dehn, \cite{De12}]\label{thm dehn word problem}
    If $\overline{W}\in\W(\gs_c)$ is freely reduced and represents the identity element in $G$, then there exists a $(2g+1)$-fractional relator $\overline{A}\subset\overline{W}$.
\end{thm}
By this theorem, for any $\overline{X}\in\W(\gs_c)$, we can reduce $\overline{X}=:\overline{X^{(0)}}$ to a freely reduced word $\overline{Y}$, and then reduce $\overline{Y}$ to $\overline{X^{(1)}}:=\overline{X_L(b_{2g+2}\cdots b_{4g})^{-1}X_R}$ if $\overline{Y}=\overline{X_L(b_1\cdots b_{2g+1})X_R}$ for some $\overline{b_1\cdots b_{4g}}\in\RR_c$. Repeat this operation $n$ times to obtain $\overline{X^{(n)}}$, until $\overline{X^{(n)}}$ is freely reduced and does not have a $(2g+1)$-fractional relator. Then, $X=1\in G$ if and only if $\overline{X^{(n)}}=1\in\W(\gs_c)$. 

Therefore, we may call $\overline{X}=\overline{x_1\cdots x_n}\in\W(\gs_c)$ a \emph{Dehn-reduced form} if $\overline{X}$ is freely reduced and contains no $(2g+1)$-fractional relator, and call $\overline{X}$ a \emph{cyclically Dehn-reduced form} if $\overline{X_i}:=\overline{x_i\cdots x_nx_1\cdots x_{i-1}}$ is a Dehn-reduced form for every $1\leq i\leq n$. With the above algorithm for word problem, Dehn presented the following theorem to solve the conjugacy problem.

\begin{thm}[Dehn, \cite{De12}]\label{thm dehn conjugacy problem} Let $\overline{U}=\overline{u_1\cdots u_m},\overline{V}=\overline{v_1\cdots v_n}\in\W(\gs_c)$ be any two cyclically Dehn-reduced forms. 
Then $U$ and $V$ are conjugate in $G$ if and only if $U_iV_j^{-1}=1$ or $aU_ia^{-1}V_j^{-1}=1$ in $G$ for some $a\in\gs_c^{\pm},~i\in\{1,\dots,m\}$ and $j\in\{1,\dots,n\}$ where 
$$\overline{U_i}:=\overline{u_i\cdots u_mu_1\cdots u_{i-1}},\qquad \overline{V_j}:=\overline{v_j\cdots v_nv_1\cdots v_{j-1}}.$$
\end{thm}

Hence, for any $\overline{X},\overline{Y}\in\W(\gs_c)$, we can reduce them to cyclically Dehn-reduced forms $\overline{X'},\overline{Y'}$ respectively. Then, the conjugacy problem of $X$ and $Y$ becomes finite word problems.

Although Theorem \ref{thm dehn conjugacy problem} is similar to Theorem \ref{main thm3 conj. class}, there is still an important difference between them. Note that Theorem \ref{thm dehn word problem} solves the word problem but it does not give a normal form for every element in $G$. For a nontrivial element $x\in G$, it may have many Dehn-reduced forms. For the conjugacy class $[x]$ in $G$, there may exist many cyclically Dehn-reduced forms (up to cyclic permutation) representing elements in $[x]$, and the number of cyclically Dehn-reduced forms (up to cyclic permutation) of $[x]$ is unbounded as $x$ traverses the entire group $G$. Consequently, it is difficult to directly define and compute the normal form of a conjugacy class in the way of Dehn. However, for any conjugacy class $[x]$ in $G$, there are at most two cyclically irreducible words (up to cyclic permutation) representing the elements in $[x]$ by Theorem \ref{main thm3 conj. class}. Moreover, the normal form of the conjugacy class $\nf([x])$ is the minimal one of these cyclically irreducible words. This enables us to easily calculate $\nf([x])$. 

\section{From the symmetric presentation to arbitrary minimal geometric presentations}\label{sect from sym. pre. to any pre.}

In this section, we shall consider the length formulae on arbitrary minimal geometric presentation of the surface group $\pi_1(\Sigma_g)(g\geq 2)$.

\begin{defn}\label{def minimal geometric presentation}
A presentation of $\pi_1(\Sigma_g)$ is \emph{minimal} if the number of generators is $2g$, and is \emph{geometric} if its Cayley $2$-complex is a plane. 
\end{defn}

Obviously, the following two presentations of $\pi_1(\Sigma_g)$ are both minimal and geometric:
    \begin{eqnarray*}
 \mbox{Canonical ~ presentation}: \quad P_c&=&\left<a_1,a_2,\dots,a_{2g}\mid [a_1,a_2]\cdots [a_{2g-1},a_{2g}]\right>, ~\gs_c=\{a_1 ,a_2,\dots,a_{2g}\}.\\
  \mbox{Symmetric ~presentation}: \quad P_s&=&\left<c_1,c_2,\dots,c_{2g}\mid c_1\cdots c_{2g}c_1^{-1}\cdots c_{2g}^{-1}\right>, \quad \gs_s=\{c_1, c_2,\dots,c_{2g}\}.
    \end{eqnarray*}


Observe that an orientable closed surface $\Sigma_g ~(g\geq 2)$ admits many representations of a $4g$-sided polygon. Each of these representations corresponds to a minimal geometric presentation. Moreover, their Cayley graphs are isomorphic to each other as graphs. The following result is classical and can be found in \cite[Section 1]{FP87}.


\begin{lem}[Floyd-Plotnick, \cite{FP87}]\label{lem cyclics are preserved under group action}
Let $P=\left<\gs\mid \RR\right>$ be a geometric presentation of $\pi_1(\Sigma_g)$. There exists a cyclic order of $\gs^\pm$ preserved under the group action of $\pi_1(\Sigma_g)$ on the Cayley graph corresponding to $P$. In other words, the order is same at every vertex of the Cayley graph.
\end{lem}
    
For example, Figure \ref{fig: cyclic orders of minimal geometric presentations} illustrates the cyclic orders of the canonical and symmetric presentations $P_c, P_s$ of $\pi_1(\Sigma_2)$. 
\begin{figure}[ht]
    \centering
    \includegraphics[width=0.35\linewidth]{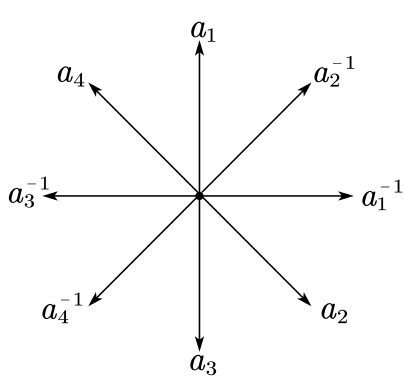}
    \includegraphics[width=0.35\linewidth]{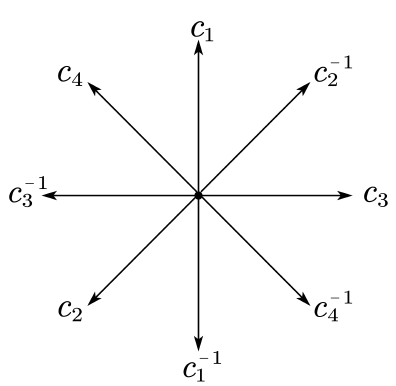}
    \caption{The cyclic orders of $\gs_c^\pm$ and $\gs_s^\pm$ are $(a_1,a_2^{-1},a_1^{-1},a_2,a_3, a_4^{-1},a_3^{-1},a_4)$ and $(c_1,c_2^{-1},c_3,c_4^{-1},c_1^{-1},c_2,c_3^{-1},c_4)$, respectively.\label{fig: cyclic orders of minimal geometric presentations}}
\end{figure}

From now on, we return to the general case.
 Let $P=\langle \gs \mid \RR \rangle$ represent any minimal geometric presentation of $\pi_1(\Sigma_g)$ with the Cayley graph $\Gamma$, and let $P_s=\langle \gs_s \mid \RR_s \rangle$ be the symmetric presentation of $\pi_1(\Sigma_g)$ with the Cayley graph $\Gamma_s$. According to Lemma \ref{lem cyclics are preserved under group action}, it is possible to select two cyclic orders $\mathcal{O}=(d_1,d_2,\dots,d_{4g})$ and $\mathcal{O}_s=(b_1,b_2,\dots,b_{4g})$ 
for $\gs^\pm$ and $\gs_s^\pm$, respectively.
Once these cyclic orders are determined and the initial generators $d_1$ and $b_1$ are assigned as above, we derive a natural bijection $$\sigma_0:\gs^\pm \to \gs_s^\pm,~d_i\mapsto b_i.$$
Notice that $\sigma_0$ functions as a local graph isomorphism from $\mathrm{St}(1)$ to $\mathrm{St}(1_s)$, where $\mathrm{St}(1)$ and $\mathrm{St}(1_s)$ denote the stars of the origins $1\in\Gamma$ and $1_s\in\Gamma_s$, respectively (refer to Figure \ref{fig: cyclic orders of minimal geometric presentations} for example). Furthermore, $\sigma_0$ can be expanded into a graph isomorphism $\phi:\Gamma\to \Gamma_s$ such that $\sigma_0$ serves as the restriction of $\phi$ to $\mathrm{St}(1)\subset\Gamma$. Additionally, for every vertex $v\in\Gamma$, let
\begin{equation}\label{eq sigma v is a restriction}
    \sigma_{v}:\gs^{\pm}\to \gs_s^\pm
\end{equation} be the restriction of $\phi$ to $\mathrm{St}(v)\subset\Gamma$. Observe that $\sigma_v$ preserves the cyclic orders, and it may differ from $\sigma_0$ by a rotation. By restricting $\phi$ to the vertex set of $\Gamma$ and to the set of paths originating at $1\in\Gamma$, we can obtain two bijections (which are not homomorphisms):
    \begin{equation}\label{eq def of f and h}
        f: \pi_1(\Sigma_g)\to \pi_1(\Sigma_g), \quad h:\W(\gs)\to\W(\gs_s).
    \end{equation} 

\begin{figure}[ht]
    \centering
    \includegraphics[width=0.75\linewidth]{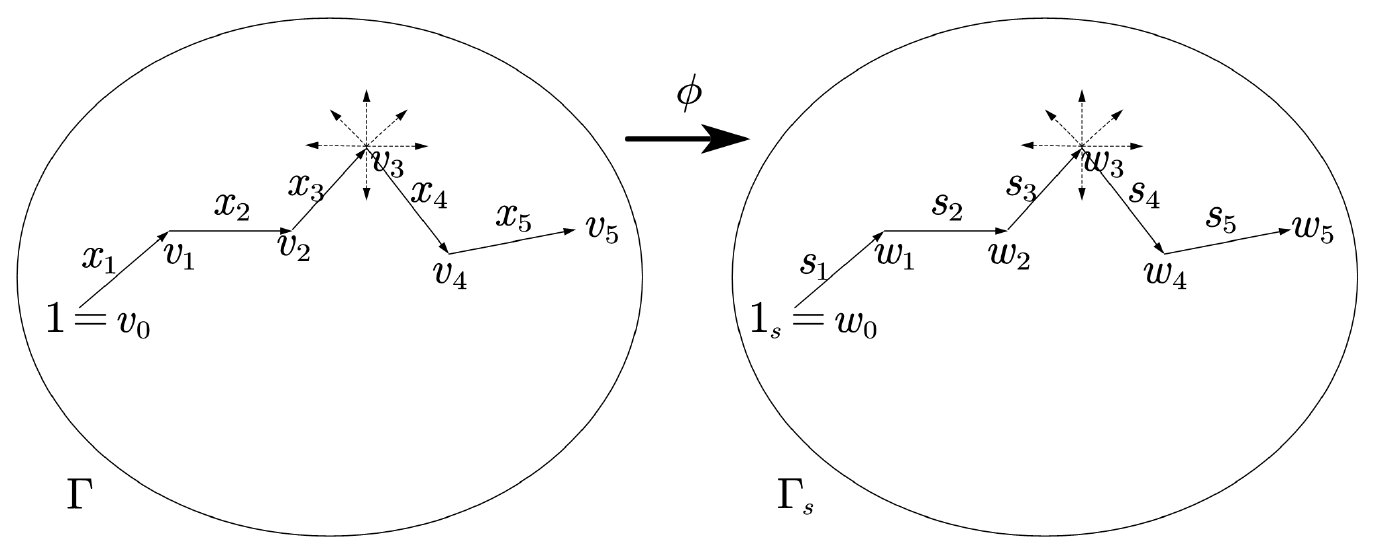}
    \caption{The graph isomorphism $\phi$ induces a bijection preserving the cyclic orders at every two corresponding vertices. For instance, $\sigma_{v_3}: x_3^{-1}\mapsto s_3^{-1},~x_4\mapsto s_4$.}
    \label{fig: graph isomorphism}
\end{figure}

Furthermore, we fix $$\mathcal{O}_s=(b_1,b_2,\dots, b_{4g})=(c_1,c_2^{-1},\dots,c_{2g-1},c_{2g}^{-1},c_1^{-1},c_2,\dots,c_{2g-1}^{-1},c_{2g})$$ illustrated in the right of Figure \ref{fig: cyclic orders of minimal geometric presentations}, and let 
\begin{equation*}
\begin{array}{llll}
  \theta_s:&  \gs_s^\pm&\to&\Z_{4g}\\
 &b_i&\mapsto& i
\end{array}\qquad\mbox{and}\qquad \begin{array}{llll}
 \theta=\theta_s\comp\sigma_0: & \gs^\pm&\to&\Z_{4g}\\
 &d_i&\mapsto& i
\end{array}.  
\end{equation*} 
Note that 
  \begin{equation}\label{eq. cyclic order of sym.}
    \theta_s(b_i)-\theta_s(b_i^{-1})=2g\in\Z_{4g},\qquad\forall ~b_i\in\gs_s^\pm. 
  \end{equation}
To calculate $f$ and $h$ in Eq. (\ref{eq def of f and h}), we need the following definition derived from the properties of group action and graph isomorphism (see Figure \ref{fig: graph isomorphism}).

\begin{defn}[$\mathcal{O}(\overline{X})$ and $\sigma_r$]\label{def graph-iso to word-iso}
For every $x\in\gs^\pm$, define 
$$\mathcal{O}(x):=\theta(x)-\theta(x^{-1})\in \mathbb Z_{4g}.$$
Then for any  $1\neq \overline{X}=\overline{x_1x_2\cdots x_n}\in\W(\gs)$, we define a sequence $$\mathcal{O}(\overline{X})=(\mathcal{O}(\overline{X})_i):=(\mathcal{O}(x_1),\mathcal{O}(x_2),\dots,\mathcal{O}(x_{n}))\in \mathbb Z^n_{4g}.$$  
For any $r\in \Z_{4g}$, we define a rotation $\sigma_r$ derived from $\sigma_0$,
        \begin{eqnarray}\label{restore function}
            \sigma_r: \quad\gs^\pm&\to& \gs_s^\pm\\
            x&\mapsto& \theta_s^{-1}(\theta(x)+r).\notag
        \end{eqnarray}
\end{defn}

We now calculate the desired bijection $h:\W(\gs)\to \W(\gs_s)$ (see Eq. (\ref{eq def of f and h})). 

Let $\overline{X}=\overline{x_1\cdots x_n}\in\W(\gs)$ and $h(\overline{X})=\overline{s_1\cdots s_n}\in  \W(\gs_s)$ with $s_1=\sigma_0(x_1)$. Note that the paths corresponding to these two words $\overline{X}$ and $h(\overline{X})$ in the Cayley graphs have the same shape, it means the numbers of angles  that paths turn through are the same at any two corresponding vertices (see Figure \ref{fig: graph isomorphism}):
   \begin{equation}\label{equation turning angles same}
        \begin{array}{rlcll}
         \theta(x_1)&=&(\theta_s\comp\sigma_0)(x_1)&=&\theta_s(s_1),\\
      \theta(x_j)-\theta(x_{j-1}^{-1})&=&\theta_s(s_j)-\theta_s(s_{j-1}^{-1})&=&\theta_s(s_j)-\theta_s(s_{j-1})+2g,\quad j=2,\dots,n,
        \end{array}
   \end{equation}
where the last ``='' holds by Eq. (\ref{eq. cyclic order of sym.}). Then for any $2\leq k\leq n$, summing both sides of Eq. (\ref{equation turning angles same}), we have 
\begin{eqnarray}\label{eq. s_k to s_k-1}
    \theta_s(s_k)&=&\theta(x_k)+\sum_{j=1}^{k-1}\left(\theta(x_j)-\theta(x_j^{-1})\right)+2g(k-1)\notag\\
     &=&\theta(x_k)+\sum_{j=1}^{k-1}\mathcal{O}(\overline{X})_j+2g(k-1)\in \mathbb Z_{4g}.
\end{eqnarray}

Recall that the map $\sigma_{v_{k}}:\gs^\pm\to\gs_s^\pm$ preserves the cyclic orders at every vertex $v_{k}\in\Gamma$. (See Eq. (\ref{eq sigma v is a restriction}) and Figure \ref{fig: graph isomorphism}. In particular, we have already fixed $\sigma_0$ for the origin $1$ of $\Gamma$.) Since the action of $\pi_1(\Sigma_g)$ on $\Gamma$ preserves the order, we can calculate $\sigma_{v_{k-1}}$ by solving the equation 
$$s_{k}=\sigma_{v_{k-1}}(x_{k})=\sigma_{r_{k-1}}(x_{k}),$$
where $r_{k-1}\in\mathbb Z_{4g}$ of the rotation $\sigma_{r_{k-1}}$ is the required value.
 
 When $k=1$, we have $s_1=\sigma_0(x_1)$ and thus $\sigma_{r_0}=\sigma_0,r_0=0\in\Z_{4g}$.
When $k\geq 2$, by Eq. (\ref{eq. s_k to s_k-1}),
  \begin{eqnarray}\label{eq for sigmar_k-1}
      s_k=\theta_s^{-1}\left(\theta_s(s_k)\right)
      =\theta_s^{-1}\left(\theta(x_k)+\sum_{j=1}^{k-1}\mathcal{O}(\overline{X})_j+2g(k-1)\right)=\sigma_{r_{k-1}}(x_k).
  \end{eqnarray}
Then combining Eq. (\ref{restore function}), we have
  \begin{equation}\label{eq for r_k-1}
      r_{k-1}=\sum_{j=1}^{k-1}\mathcal{O}(\overline{X})_j+2g(k-1), ~k=2,\dots,n.
  \end{equation}

Therefore, we obtain the specific expression of $h:\W(\gs)\to \W(\gs_s)$ as follows:
\begin{equation}\label{eq. specific h}
    h:\overline{x_1\cdots x_n}\mapsto \overline{s_1\cdots s_n}, \qquad s_k=\sigma_{r_{k-1}}(x_k),
\end{equation}
where $r_0=0$ and $r_{k-1}$ ($k=2,\dots, n$) is defined as in Eq. (\ref{eq for r_k-1}).


We now give a simple example to show how to use $\mathcal{O}(\cdot)$ and $\sigma_r$ to calculate the image of a word in $\W(\gs)$.

\begin{exam}
    In $\pi_1(\Sigma_2)$, let $\sigma_0:\gs_c^\pm\to \gs_s^\pm$ be a bijection defined as follows. Then we can get the definitions of $\theta,\theta_s$ by referring to Figure \ref{fig: cyclic orders of minimal geometric presentations}:
    \begin{equation*}
\begin{tikzcd}[row sep=0cm, column sep=2cm]
    (a_1,a_2^{-1},a_1^{-1},a_2,a_3,a_4^{-1},a_3^{-1},a_4) \arrow[dr, |->, "\theta"] \arrow[dd, |->, "\sigma_0"'] & \\
    & (1,2,3,4,5,6,7,8). \\
    (c_1, c_2^{-1}, c_3, c_4^{-1}, c_1^{-1}, c_2, c_3^{-1}, c_4)~~ \arrow[ur,|->, "\theta_s"'] & 
\end{tikzcd}
    \end{equation*}
    Then we have $h(\overline{a_1})=\overline{\sigma_0(a_1)}=\overline{c_1}$. For the word $\overline{a_1a_1}$, we have $\mathcal{O}(\overline{a_1a_1})=(6,6)$ and by Eq. (\ref{eq. specific h}),
    $$r_1=6+4=2\in\Z_{8},\quad \sigma_{r_1}(a_1)=\sigma_{2}(a_1)=\theta_s^{-1}(\theta(a_1)+2)=c_3.$$ Thus, $$h(\overline{a_1a_1})=\overline{c_1c_3}\neq \overline{c_1c_1}=h(\overline{a_1})^2.$$ 
    Similarly, for the word $\overline{a_1a_1^{-1}}$, we can obtain $\mathcal{O}(\overline{a_1a_1^{-1}})=(6,2)$, $h(\overline{a_1a_1^{-1}})=\overline{c_1c_1^{-1}}$ and $$h(\overline{(a_1a_1^{-1})^2})=\overline{(c_1c_1^{-1})^2}=h(\overline{a_1a_1^{-1}})^2.$$ 
    Furthermore,  for any $m\geq 1$, we have $$h(\overline{(a_1a_1^{-1})^m})=h(\overline{a_1a_1^{-1}})^m.$$ 
\end{exam}

In general, the bijections $f:\pi_1(\Sigma_g)\to \pi_1(\Sigma_g)$ and $h:\W(\gs)\to \W(\gs_s)$ are not homomorphisms. However, we still have the following proposition.

\begin{prop}\label{prop property of graph isomorphism}
   For a word $\overline{X}=\overline{x_1\cdots x_n}\in\W(\gs)$, if $\sum_{k=1}^{n}\mathcal{O}(\overline{X})_k=2gn\in\Z_{4g}$, then for any $m\geq 1$,
   $$h(\overline{X^m})=h(\overline{X})^m,\qquad f(X^m)=f(X)^m.$$
\end{prop}

\begin{proof}
  It suffices to show the case of $m=2$ because the case when $m\geq 3$ is totally same. Let $h(\overline{X})=\overline{s_1\cdots s_{n}}$. By Eq. (\ref{eq. specific h}) and Definition \ref{def graph-iso to word-iso}, we can assume $h(\overline{X^2})=\overline{s_1\cdots s_ns_{n+1}\cdots s_{2n}}$. Note that $\mathcal{O}(\overline{X^2})=(\mathcal{O}(\overline{X}),\mathcal{O}(\overline{X}))$ and 
  \begin{equation}\label{eq O^2=OO}
      \mathcal{O}(\overline{X^2})_{n+k}=\mathcal{O}(\overline{X^2})_{k}=\mathcal{O}(\overline{X})_k,\qquad k=1,\dots,n.
  \end{equation}
Then we have
\begin{eqnarray*}
         s_{n+k}&\xlongequal{\mbox{\tiny{Eq. (\ref{eq. specific h})}}}&\sigma_{r_{n+k-1}}(x_k)\\
         &\xlongequal{\mbox{\tiny{Eq. (\ref{eq for sigmar_k-1})}}}&\theta_s^{-1}\left(\theta(x_k)+\sum_{j=1}^{n+k-1}\mathcal{O}(\overline{X^2})_j+2g(n+k-1)\right)\\
         &\xlongequal{\mbox{\tiny{Eq. (\ref{eq O^2=OO})}}}&\theta_s^{-1}\left(\theta(x_k)+\sum_{j=1}^{k-1}\mathcal{O}(\overline{X})_j+2g(k-1)+\sum_{j=1}^{n}\mathcal{O}(\overline{X})_j+2gn\right)\\
         &\xlongequal{\mbox{\tiny{Eq. (\ref{eq for r_k-1})}}}&\theta_s^{-1}\left(\theta(x_k)+r_{k-1}+4gn\right)\\
         &\xlongequal{\mbox{\tiny{Eq. (\ref{restore function})}}}&\sigma_{r_{k-1}}(x_k)\\
         &\xlongequal{\mbox{\tiny{Eq. (\ref{eq. specific h})}}}&s_k.
     \end{eqnarray*}
     Therefore, $h(\overline{X^2})=h(\overline{X})^2$ and thus $f(X^2)=f(X)^2$. We have finished the proof.
 \end{proof}

Note that every minimal geometric presentation $P$ gives a cyclic order $\mathcal{O}$ of its generating set $\gs^\pm$ and thus induces a function $\mathcal{O}(\cdot)$ as in Definition \ref{def graph-iso to word-iso}. 
Then, using Proposition \ref{prop property of graph isomorphism}, we obtain the following similar formulae as in Theorem \ref{main thm word length formulae}. Here, $\gcd(\cdot)$ denotes the greatest common divisor.

\begin{manualthm}{\ref*{main thm4 word length formulae more rough}}
    For any minimal geometric presentation $P=\left<\gs\mid R\right>$ of a surface group $\pi_1(\Sigma_g)$ with cyclic order $\mathcal{O}$, 
  let $t=\dfrac{4g}{\gcd(2g,p)}$ where $1\leq p=\gcd\{\mathcal{O}(d)\mid d\in \gs^\pm\}\leq 4g$. Then for any $1\neq x\in \pi_1(\Sigma_g)$ and $m\geq 1$, we have
    \begin{enumerate}
        \item $|x^{2t}|_{\gs}>|x^t|_{\gs}$;
        \item $|x^{tm}|_{\gs}=(m-1)(|x^{2t}|_{\gs}-|x^t|_{\gs})+|x^t|_{\gs}$;
        \item $\swl_{\gs}(x)=\frac{1}{t}(|x^{2t}|_{\gs}-|x^t|_{\gs})$.
    \end{enumerate}
    In particular, for the canonical presentation $P_c$, we can take $t=2g$.   
\end{manualthm}

\begin{proof} 
Let $\overline{X}=\overline{x_1\cdots x_n}$ be a word in $\W(\gs)$ representing $1\neq x\in \pi_1(\Sigma_g)$ and $\mathcal{O}(\overline{X})=(\mathcal{O}(x_1),\dots,\mathcal{O}(x_n))$. Then 
$$\mathcal{O}(\overline{X^{t}})=(\underbrace{\mathcal{O}(x_1),\dots,\mathcal{O}(x_n),\dots,\mathcal{O}(x_1),\dots,\mathcal{O}(x_n)}_{t\mbox{ times }(\mathcal{O}(x_1),\dots,\mathcal{O}(x_n))}).$$
 Note $tp=0=2gt\in \mathbb Z_{4g}$, thus for the word $\overline{X^{t}}$, we have $$\sum_{k=1}^{tn} \mathcal{O}(\overline{X^{t}})_k=t\cdot\sum_{k=1}^{n}\mathcal{O}(x_k)=t p\cdot\sum_{k=1}^{n}\dfrac{\mathcal{O}(x_k)}{p}=0=2gtn\in\Z_{4g}.$$
  According to Proposition \ref{prop property of graph isomorphism}, we have 
  $h(\overline{X^{tm}})=h(\overline{X^{t}})^m$ and $f(x^{tm})=f(x^{t})^m$ for every $m\geq 1$. 
  Note that $x^{tm}$ and $f(x^{tm})$ are located at the corresponding positions of $\Gamma=\Gamma(\pi_1(\Sigma_g),\gs)$ and the Cayley graph $\Gamma_s$ of the symmetric presentation, respectively. Thus, they have the same word length $$|x^{tm}|_{\gs}=|f(x^{tm})|_{\gs_s}=|f(x^{t})^m|_{\gs_s}.$$
  Therefore, combining Theorem \ref{main thm word length formulae}, the formulae in this theorem hold:
  
  (1) $|x^{2t}|_{\gs}=|f(x^t)^2|_{\gs_s}>|f(x^t)|_{\gs_s}=|x^{t}|_{\gs}$;
  
  (2) for any $m\geq 1$,
  \begin{eqnarray*}
       |x^{tm}|_{\gs}&=&|f(x^t)^m|_{\gs_s}\\
       &=&(m-1)(|f(x^t)^2|_{\gs_s}-|f(x^t)|_{\gs_s})+|f(x^t)|_{\gs_s}\\
       &=&(m-1)(|x^{2t}|_{\gs}-|x^t|_{\gs})+|x^t|_{\gs};   
      \end{eqnarray*}

(3) $\swl_{\gs}(x)=\lim_{m\to \infty}\frac{|x^{tm}|_{\gs}}{tm}=\frac{1}{t}(|x^{2t}|_{\gs}-|x^t|_{\gs}).$

\noindent In particular, for the canonical presentation, we have $t=4g/2=2g$ because $\mathcal{O}(d)=\pm 2$ for any $d\in \gs_c^\pm.$
\end{proof}


\section{A remark on several growths}\label{sect remark}
The study of standard growth \cite{AJLM16,CW92,FP87,GN97,Lo14}, conjugacy growth and primitive  (i.e., not a proper power) conjugacy growth \cite{AC17,CK02,CK04,GS10,GY22,Ri04,Ri10} provides a deep understanding of the structure of groups by quantifying these diverse algebraic and geometric invariants in geometric group theory. The standard growth rate is a classic research object, for example,
\begin{thm}[Alsed\`{a}-Juher-Los-Ma\~{n}osas, \cite{AJLM16}]\label{thm for standard growth rate}
    For $n>2$, let $G$ be an (orientable or non-orientable) surface group of rank $n$ with a minimal geometric presentation $P$. Then, the volume entropy of $G$ with respect to the presentation $P$ is $\log(\lambda_n)$. Moreover, for $n\geq 4$, $\lambda_n$ satisfies:
$$2n-1-\dfrac{1}{(2n-1)^{n-2}}<\lambda_n<2n-1.$$
\end{thm}
In fact, the standard growth rate was proven to be the maximal real root of the following polynomial, see \cite{AJLM16,CW92}:
$$x^n-2(n-1)\sum_{i=1}^{n-1}x^i+1.$$

In this section, we provide rough estimates of standard growth, conjugacy growth and primitive conjugacy growth for the surface group $G=\pi_1(\Sigma_g) ~(g\geq 2)$ under the symmetric presentation (\ref{symmetric presentation}). These estimates have already been established for a broader class of groups: the standard growth of surface groups \cite[Theorem 1.1]{AJLM16}, the (primitive) conjugacy growth of acylindrically hyperbolic groups \cite[Theorem 1.2]{AC17} and groups admitting an statistically convex-cocompact action on some proper geodesic metric space \cite[Main Theorem]{GY22}.

We say that the \emph{length} $|[x]|$ of a conjugacy class $[x]\subset G$ is the length of its normal form, that is
$$|[x]|:=|\nf([x])|=\min\{|yxy^{-1}| \mid y\in G\}.$$
Denote the \emph{number of group elements} (resp. \emph{conjugacy classes} and \emph{primitive conjugacy classes}) of \emph{length} $n$ in $G$ by
\begin{eqnarray*}
   \groupelement(n)&:=&\#\{x\in G\mid |x|=n\}=\#\{\overline{X}\in \W(S)\mid \overline{X} ~\mathrm{is ~irreducible~ and} ~|\overline{X}|=n\}, \\
    \generalconj(n)&:=&\#\{[x]\subset G\mid |[x]|=n\},\\
    \primitiveconj(n)&:=&\#\{[x]\subset G\mid |[x]|=n ~\mathrm{and} ~x ~\mathrm{is ~not ~a ~proper ~power} \},      
\end{eqnarray*}
respectively, and let
$$E(n):=\sum_{i=0}^n\groupelement(i), \quad C(n):=\sum_{i=0}^n\generalconj(i), \quad P(n):=\sum_{i=0}^n\primitiveconj(i).$$
Then the \emph{standard growth rate}, the \emph{conjugacy growth rate} and the \emph{primitive conjugacy growth rate} are defined as follows, respectively:
$$\limsup_{n\to\infty}\sqrt[n]{E(n)},\qquad\limsup_{n\to\infty}\sqrt[n]{C(n)},\qquad\limsup_{n\to\infty}\sqrt[n]{P(n)}.$$

\subsection{Standard growth rate}
We first provide a rough estimate for the standard growth rate using Lemma \ref{lem for Xw types}.

  Given an irreducible word $\overline{X}=\overline{x_1\cdots x_n}\in\W(\gs)$ with $n\geq 1$, there are five distinct cases for $\overline{Xw}$ by Lemma \ref{lem for Xw types} with $\overline{b_1\cdots b_{4g}}\in\RR$. Let 
  $$N_i:=\#\{\mathrm{words} ~\overline{Xw} ~\mathrm{in ~item ~(i) ~of ~Lemma} ~\ref{lem for Xw types}\}, \quad i=1, \ldots, 5.$$
\begin{enumerate}
    \item $w=x_n^{-1}$. Then $N_1=\groupelement(n)$.
    \item $\overline{x_{n-2g+1}\cdots x_nw}=\overline{b_{4g}b_1\cdots b_{2g-1}b_{2g}}$. Note that for every irreducible word $\overline{x_1\cdots x_{n-2g}}$, there is always at least one choice for $\overline{b_{4g}b_1\cdots b_{2g}}$, and for every irreducible word $\overline{x_1\cdots x_{n-2g+1}}$, there is at most two choices for $\overline{b_1\cdots b_{2g}}$ with $x_{n-2g+1}=b_{4g}$. (In most cases, there are two choices available for $x_{n-2g+2}$, whereas in a minority of cases, only one choice exists or none at all. Then, once $x_{n-2g+1}$ and $x_{n-2g+2}$ are determined, $\overline{x_1\cdots x_nw}$ is also determined). Therefore,
$$\groupelement(n-2g)\leq N_2\leq 2\groupelement(n-2g+1).$$
    \item $\overline{x_{n-(2g-1)t}\cdots x_nw}=\overline{b_{4g}(b_1\cdots b_{2g-1})^tb_{2g}}$ for $t\geq 2$. For a given $t\geq 2$, by similar arguments as in item (2), the number $N_{3,t}$ of such $\overline{Xw}$ satisfies
$$\groupelement(n-(2g-1)t-1)\leq N_{3,t}\leq 2\groupelement(n-(2g-1)t).$$
Moreover, we have $N_3=\sum_{t=2}^{\lfloor \frac{n-1}{2g-1}\rfloor}N_{3, t}$.
    \item $\overline{x_{n-(2g-1)t+1}\cdots x_nw}=\overline{(b_1\cdots b_{2g-1})^tb_{2g}}$ with $b_1\succ b_{2g},t\geq 1$ maximal and $x_{n-(2g-1)t}\neq b_{4g}$. Note that for every irreducible word $\overline{x_1\cdots x_{n-(2g-1)t-1}}$, we choose $x_{n-(2g-1)t}=x_{n-(2g-1)t-1}$ and then there is at least one choice for $\overline{(b_1\cdots b_{2g-1})^tb_{2g}}$ such that $$\overline{x_1\cdots x_{n-(2g-1)t-1}x_{n-(2g-1)t}(b_1\cdots b_{2g-1})^tb_{2g}}$$ satisfies the conditions: $\overline{x_1\cdots x_{n-(2g-1)t}(b_1\cdots b_{2g-1})^t}$ is irreducible, $b_1\succ b_{2g}$, $t$ is maximal and $x_{n-(2g-1)t}\neq b_{4g}$.
Then for a given $t\geq 1$, the number $N_{4,t}$ of such $\overline{Xw}$ satisfies
$$\groupelement(n-(2g-1)t-1)\leq N_{4,t} \leq 2\groupelement(n-(2g-1)t+1).$$
and $N_4=\sum_{t=1}^{\lfloor \frac{n}{2g-1}\rfloor}N_{4, t}$.
    \item $\overline{Xw}$ is irreducible. Then $N_5=\groupelement(n+1)$.
\end{enumerate}

Since for every irreducible word $\overline{x_1\cdots x_n}$ there are $4g$ choices for $w\in\gs^\pm$, we have $$4g\groupelement(n)=N_1+N_2+N_3+N_4+N_5.$$
Thus
\begin{eqnarray}
    \groupelement(n+1)& \leq &(4g-1)\groupelement(n)- \sum_{t=1}^{\lfloor \frac{n}{2g-1}\rfloor}\groupelement(n-(2g-1)t-1)-\sum_{t=1}^{\lfloor\frac{n-1}{2g-1}\rfloor}\groupelement(n-(2g-1)t-1),\label{eq 1 in sect. growth}\\
    \groupelement(n+1)&\geq &(4g-1)\groupelement(n)-2\sum_{t=1}^{\lfloor \frac{n}{2g-1}\rfloor}\groupelement(n-(2g-1)t+1)-2\sum_{t=1}^{\lfloor \frac{n-1}{2g-1}\rfloor}\groupelement(n-(2g-1)t).\label{eq 2 in sect. growth}
\end{eqnarray}
Note when $1\leq n\leq 2g-2$, we clearly have $\groupelement(n+1)= (4g-1)\groupelement(n)>(4g-2)\groupelement(n)$. Then by induction on $n$, we can obtain the following.
\begin{lem}\label{lem for e(n+1)}
\begin{enumerate}
 \item $(4g-2)\groupelement(n)+1\leq  \groupelement(n+1)\leq  (4g-1)\groupelement(n)$ for every $n\geq 1$; 

 \item  $4g-2< \dfrac{E(n+1)}{E(n)}< 4g-1$ for every $n\geq 2g$.
\end{enumerate} 
\end{lem}

Moreover, by Eq. (\ref{eq 1 in sect. growth}) and Eq. (\ref{eq 2 in sect. growth}), for $k=0,1,\dots,2g-2$ and $\groupelement(-1):=0$, we have
\begin{eqnarray}
\label{eq 3 in sect. growth}\\
\groupelement(n+1-k) \leq (4g-1)\groupelement(n-k)- \sum_{t=1}^{\lfloor \frac{n-k}{2g-1}\rfloor}\groupelement(n-(2g-1)t-1-k)-\sum_{t=1}^{\lfloor \frac{n-1-k}{2g-1}\rfloor}\groupelement(n-(2g-1)t-1-k)\notag,\\
\label{eq 4 in sect. growth}\\
\groupelement(n+1-k)\geq (4g-1)\groupelement(n-k)-2\sum_{t=1}^{\lfloor \frac{n-k}{2g-1}\rfloor}\groupelement(n-(2g-1)t+1-k)-2\sum_{t=1}^{\lfloor \frac{n-1-k}{2g-1}\rfloor}\groupelement(n-(2g-1)t-k)\notag.
\end{eqnarray}
By calculating $\sum_{k=0}^{2g-2}\mbox{Eq. (\ref{eq 3 in sect. growth})}$ and $\sum_{k=0}^{2g-2}\mbox{Eq. (\ref{eq 4 in sect. growth})}$, we have
\begin{eqnarray*}
  E(n+1)-E(n-2g+2)&\leq &(4g-1)(E(n)- E(n-2g+1))-2E(n-2g),\\
E(n+1)-E(n-2g+2)&\geq &(4g-1)(E(n)- E(n-2g+1))-2E(n-2g+2)-2E(n-2g+1).
\end{eqnarray*}
Furthermore, combining the above two equations with Lemma \ref{lem for e(n+1)}(2), for $n\geq 2g$, we have
$$\dfrac{E(n+1)}{E(n)}\leq 4g-1-\dfrac{2E(n-2g)}{E(n)}\leq 4g-1-\dfrac{2}{(4g-1)^{2g}},$$
and
$$\dfrac{E(n+1)}{E(n)}\geq 4g-1 -\dfrac{(4g+1)E(n-2g+1)+E(n-2g+2)}{E(n)}
\geq 4g-1-\dfrac{8g-1}{(4g-2)^{2g-1}}.$$
Therefore, we have obtained a rough estimate of the standard growth rate as follows.
\begin{prop}\label{exam standard growth}
For the surface group $G=\pi_1(\Sigma_g) ~(g\geq 2)$ under the symmetric presentation (\ref{symmetric presentation}),
$$4g-1-\dfrac{8g-1}{(4g-2)^{2g-1}}\leq \limsup_{n\to\infty}\sqrt[n]{E(n)}\leq 4g-1-\dfrac{2}{(4g-1)^{2g}}.$$
\end{prop}
\begin{rem}\label{rem ...}
 Indeed, the standard growth rate can be computed precisely using this method. To avoid tedious discussions, we give a rough estimation here.
\end{rem}

\subsection{Conjugacy growth and primitive conjugacy growth}\label{subsect conjugacy growth} 
Let $\overline{X}=\overline{x_1\cdots x_n}\in\W(\gs)$ be an irreducible word with $n$ sufficiently large. Denote all the positive divisors of $n$ as follows: $$1=d_1<d_2<\cdots<d_s<d_{s+1}=n.$$
We discuss in two cases.

Case (1). $\overline{X}$ is cyclically irreducible. Denote the number of such $\overline{X}$ as $A_n$. 

Let $\overline{Z}$ be the primitive root of $\overline{X}$. Suppose $\overline{X}=\overline{Z^{n/d_i}}$. Then $\overline{Z}$ has length $d_i$ and is cyclically irreducible. Denote the number of such words $\overline{X}$ as $A_{n,i}$. Note that if $\nf([y^t])=\overline{y_1\cdots y_m}$ for some $t,m>0$ and $y\in G$, then $\nf([y])=\overline{y_1\cdots y_k}$ and $\overline{y_1\cdots y_m}=\overline{(y_1\cdots y_k)^t}$ with $m=kt$, and hence 
\begin{equation}\label{eq C(n)=P(n)+...}
\generalconj(n)=\primitiveconj(n)+\sum_{i=1}^s\primitiveconj(d_i).
\end{equation}
Moreover, according to Theorem \ref{main thm3 conj. class}, there is exactly one exception between conjugacy class with length $n$ and cyclically irreducible word class with length $n$ if and only if $(2g-1)\mid n$. In particular, if $(2g-1)\mid n$, then there are exactly $4g(2g-1)=8g^2-4g$ additional, cyclically irreducible words representing $4g$ conjugacy classes. 
Therefore, 
\begin{equation*}
  A_{n,i}=\left\{\begin{array}{ll}
    d_i\primitiveconj(d_i)   &\mbox{when }d_i\neq 2g-1  \\
    d_i\primitiveconj(d_i)+8g^2-4g   & \mbox{when }d_i=2g-1
  \end{array}\right.\quad,  
\end{equation*}
and thus
\begin{equation}\label{eq A_n=nP(n)+...}
  A_n=\sum_{i=1}^{s+1}A_{n,i}=\delta(n)+n\primitiveconj(n)+\sum_{i=1}^{s}d_i\primitiveconj(d_i),  
\end{equation}
where
\begin{equation}\label{eq function delta}
    \delta(n)=\left\{\begin{array}{ll}
        8g^2-4g & \mbox{when } (2g-1)\mid n \\
        0 & \mbox{when } (2g-1)\nmid n
    \end{array}\right.\quad.
\end{equation}

Case (2). $\overline{X}$ is irreducible but not cyclically irreducible. Denote the number of such $\overline{X}$ as $B_n$.

In this case, $\overline{X}$ has a cyclic permutation $\overline{X'}$ which is reducible. We now have the following disjoint classification.

Subcase (2.1). $\overline{X'}$ is reducible by $S_{(1)}$. Then $\overline{X}=\overline{x_{1}\cdots x_{n-1}x_1^{-1}}$ and every such word $\overline{X}$ is determined by its irreducible subword $\overline{x_1\cdots x_{n-1}}$. Denote the number of such $\overline{X}$ as $B_{n,1}$. Then $$B_{n,1}\leq \groupelement(n-1).$$
To get a lower bound of $B_{n,1}$, we first fix an irreducible word $\overline{Y}=\overline{y_1\cdots y_{n-1}}$, and then have the following cases:
\begin{enumerate}
    \item $\overline{y_1\cdots y_{n-1}y_1^{-1}}$ is irreducible. The number of such words $\overline{Y}$ is $B_{n,1}$.
    \item $\overline{y_1\cdots y_{n-1}y_1^{-1}}$ is reducible by $S_{(1)}$, i.e. $y_{n-1}=y_1$. The number of such words $\overline{Y}$ denoted as $D_1$ satisfies
    $$0\leq D_1\leq \groupelement(n-2).$$
    \item $\overline{y_1\cdots y_{n_1}y_1^{-1}}$ is reducible by $S_{(i)}(i=2,3,4)$. Let $y_1=b_1$ for some $\overline{b_1\cdots b_{4g}}\in\RR$. Then, $y_1^{-1}=b_{2g+1}$ and 
    $$\overline{y_1\cdots y_{n-1}y_1^{-1}}=\overline{y_1\cdots y_{n-2g}b_2\cdots b_{2g}b_{2g+1}}\mbox{ or }\overline{y_1\cdots y_{n-2g}b_{4g}\cdots b_{2g+2}b_{2g+1}}.$$
For each case, such $\overline{Y}$ is determined by $\overline{y_1\cdots y_{n-2g}}$. Hence, the number of such words $\overline{Y}$ denoted by $D_2$ satisfies 
$$0\leq D_2\leq 2\groupelement(n-2g).$$
\end{enumerate}

Since $\groupelement(n-1)=B_{n,1}+D_1+D_2$, by the above analysis, we have
\begin{equation}\label{eq upper and lower bound of Bn,1}
   \groupelement(n-1)- \groupelement(n-2)-2\groupelement(n-2g)\leq B_{n,1}\leq \groupelement(n-1).
\end{equation}

Subcase (2.2). $\overline{X'}$ is reducible by $S_{(2,2g+1)}$ or $S_{(3,t)}(t\geq 2)$. 
Denote the number of such $\overline{X}$ as $B_{n,2}$. Then 
$$\overline{X}=\overline{b_{i+1}\cdots b_{2g}\underbrace{(b_2\cdots b_{2g})^{t_1}b_{2g+1}x_{n_3}\cdots x_{n_4} b_1(b_2\cdots b_{2g})^{t_2}}_W b_2\cdots b_i},$$ 
for some $\overline{b_1\cdots b_{4g}}\in\RR$, $2\leq i\leq 2g$ and $t_1,t_2\geq 0.$
Since every $2g-1$ such words
$$\left\{\overline{Wb_2\cdots b_{2g}}, ~\overline{b_{2g}Wb_2\cdots b_{2g-1}}, \ldots, \overline{b_3\cdots b_{2g}Wb_2}\right\}$$ 
are determined by the irreducible word $\overline{W}$ with $|\overline{W}|=n-2g-1$, we have
\begin{equation}\label{eq upper and lower bound of Bn,2}
  0\leq B_{n,2}\leq (2g-1)\groupelement(n-2g+1).  
\end{equation}

Subcase (2.3). $\overline{X'}$ is reducible by $S_{(4,t)}$ and no cyclic permutation of $\overline{X}$ is reducible by $S_{(2,2g+1)}$ or $S_{(3,t')}(t'\geq 2)$. Denote the number of such $\overline{X}$ as $B_{n,3}$. Then
$$\overline{X}=\overline{b_{i+1}\cdots \underbrace{b_{2g}x_{n_5}\cdots x_{n_6}}_W b_1\cdots b_{i}}$$
for some $1\leq i\leq 2g-1; x_{n_5}\neq b_{2g+1};x_{n_6}\neq b_{4g};b_1\succ b_{2g}$.
Since every $(2g-1)$ such words
$$\left\{\overline{Wb_1\cdots b_{2g-1}}, ~\overline{b_{2g-1}Wb_1\cdots b_{2g-2}}, \ldots, \overline{b_2\cdots b_{2g-2}Wb_1} \right\}$$ 
are determined by the irreducible word $\overline{W}$ with $|\overline{W}|=n-2g-1$, we again have
\begin{equation}\label{eq upper and lower bound of Bn,3}
  0\leq B_{n,3}\leq (2g-1)\groupelement(n-2g+1).  
\end{equation}\\

Since $\groupelement(n)=A_n+B_{n,1}+B_{n,2}+B_{n,3}$, by Eqs. (\ref{eq upper and lower bound of Bn,1})(\ref{eq upper and lower bound of Bn,2})(\ref{eq upper and lower bound of Bn,3}) and Lemma \ref{lem for e(n+1)}, we have
\begin{equation}\label{eq upper and lower bound of An}
\groupelement(n)-\groupelement(n-1)-(4g-2)\groupelement(n-2g+1) \leq A_n \leq \groupelement(n)-\frac{2g-2}{2g-1}\groupelement(n-1)-\ell\groupelement(n-2g),
\end{equation}
where $\ell=(4g-2)^{2g-2}-2\geq 34$.
Note that every primitive conjugacy class with length $k\geq 1$ is determined by at least $k$ cyclically irreducible words with length $k$ representing it. Therefore,
$$k\primitiveconj(k)\leq\groupelement(k).$$
Furthermore, by Lemma \ref{lem for e(n+1)}, we have
\begin{equation}\label{eq upper bound of sum di primitive conj. classes}
    \sum_{i=1}^{s}d_i\primitiveconj(d_i)\leq \sum_{i=1}^{s}\groupelement(d_i)\leq 2\groupelement(\lfloor \frac{n}{2}\rfloor).
\end{equation}
Therefore, for sufficiently large $n$, by Eqs. (\ref{eq A_n=nP(n)+...})(\ref{eq C(n)=P(n)+...}), we have
\begin{eqnarray*}
 &\generalconj(n)&=\dfrac{1}{n}(A_n-\delta(n)+\sum_{i=1}^{s}(n-d_i)\primitiveconj(d_i))\\
 &&\leq \dfrac{1}{n}(A_n-\delta(n)+n\sum_{i=1}^sd_i\primitiveconj(d_i))\\
 \mbox{by Eq. (\ref{eq upper bound of sum di primitive conj. classes})}&&\leq \dfrac{1}{n}(A_n-\delta(n)+2n\groupelement(\lfloor\frac{n}{2}\rfloor))\\
 \mbox{by Eq. (\ref{eq upper and lower bound of An})}&&\leq \dfrac{1}{n}(\groupelement(n)-\dfrac{2g-2}{2g-1}\groupelement(n-1)).
\end{eqnarray*}
On the other hand, by Eqs. (\ref{eq A_n=nP(n)+...})(\ref{eq upper bound of sum di primitive conj. classes}) and  Eq. (\ref{eq upper and lower bound of An}),
$$\primitiveconj(n)\geq \dfrac{1}{n}(A_n-\delta(n)-2\groupelement(\lfloor\frac{n}{2}\rfloor))
    \geq \dfrac{1}{n}(\groupelement(n)-\dfrac{4g-1}{4g-2}\groupelement(n-1)).$$
By Eqs. (\ref{eq C(n)=P(n)+...})(\ref{eq upper bound of sum di primitive conj. classes}),
$$\lim_{n\to \infty}\dfrac{\primitiveconj(n)}{\generalconj(n)}=\lim_{n\to\infty}\dfrac{\primitiveconj(n)}{\primitiveconj(n)+\sum_{i=1}^s\primitiveconj(d_s)}=1.$$

In conclusion, we have obtained the following proposition.

\begin{prop}\label{prop for primitive conjugacy}
    For a surface group $G=\pi_1(\Sigma_g)$ with a symmetric presentation (\ref{symmetric presentation}) and sufficiently large $n$, we have 
    \begin{enumerate}
        \item $\groupelement(n)-\dfrac{4g-1}{4g-2}\groupelement(n-1)\leq n\primitiveconj(n)\leq n\generalconj(n) \leq \groupelement(n)-\dfrac{2g-2}{2g-1}\groupelement(n-1)$;
        \item
        $(1-\dfrac{4g-1}{(4g-2)^2})\groupelement(n)\leq n\primitiveconj(n)\leq n\generalconj(n)\leq (1-\dfrac{2g-2}{(4g-1)(2g-1)})\groupelement(n)$;
        \item $\lim_{n\to\infty}\dfrac{\primitiveconj(n)}{\generalconj(n)}=1;$
        \item $\limsup_{n\to\infty}\sqrt[n]{E(n)}=\limsup_{n\to\infty}\sqrt[n]{C(n)}=\limsup_{n\to\infty}\sqrt[n]{P(n)}.$
    \end{enumerate}
\end{prop}
\begin{rem}\label{rem conj.}
For items (1) and (2) in the above proposition, Antol\'in and Ciobanu \cite[Theorem 1.2]{AC17}, as well as Gekhtman and Yang \cite[Main Theorem]{GY22}, successively proved that the (primitive) conjugacy growth can be dominated by the standard growth in a broader class of groups. We present here a rough estimate of the coefficient. Indeed, this coefficient can likewise be computed precisely through a more extensive discussion. Item (3) was also proven by Gekhtman and Yang for groups admitting a statistically convex-cocompact action on some proper geodesic metric spaces.
\end{rem}

\noindent\textbf{Acknowledgements.} The authors would like to thank Dr. Renxing Wan for bringing the recent results in (primitive) conjugacy growth \cite{GY22} to our attention.

\appendix
\section{\texorpdfstring{Proof of Proposition \ref{prop classification of LLFR of X^2}(\ref{XX reducible}) ($|\overline{T}|\in\{2g,\dots,4g-1\}$)}{Proof of Proposition \ref{prop classification of LLFR of X^2}(\ref{XX reducible}) (|overline{T}| in \{2g,...,4g-1\})}}\label{sect appendix for some detailed proofs of props}

Note that in Proposition \ref{prop classification of LLFR of X^2}(\ref{XX reducible}),  $\overline{X}=\overline{x_1\cdots x_n}(n\geq 1)$ is cyclically freely reduced while $\overline{X^2}$ is reducible, then according to the discussion following Proposition \ref{prop classification of LLFR of X^2}, there exists a reducible subword  $\overline{V}$ satisfying
$$\overline{x_nx_1}\subset \overline{V}\subset \overline{X^2}.$$
Recall that $\overline{T}$ is the $\llfr$ of $\overline{X^2}$ containing $\overline{x_nx_1}$ in this section. 

\subsection{\texorpdfstring{Proof of Proposition \ref{prop classification of LLFR of X^2}(\ref{XX reducible}) ($|\overline{T}|=2g$)}{Proof of Proposition \ref{prop classification of LLFR of X^2}(\ref{XX reducible}) (|overline{T}|=2g)}}\label{Appendix A1}

Since $|\overline{T}|=2g$, the reducible subword $\overline{V}$ can only be of types $S_{(3,t)}~(t\geq 2)$ or $S_{(4,t_0)}~(t_0\geq 1)$. Now we have two cases.

	\textbf{Case (1). $\overline{V}$ is of type $S_{(3,t)}(t\geq 2)$.} Then $\overline{V}$ has two possible positions in $\overline{X^2}$:
\begin{eqnarray*}
\overline{V}&=&\overline{b_1(b_2\cdots b_{2g})^tb_{2g+1}}\\
&=&
\begin{cases}
  \overline{T (b_2\cdots b_{2g})^{t-1}b_{2g+1}} &\mbox{for}~\overline{T}=\overline{b_1\cdots b_{2g}}\\
  \overline{b_1(b_2\cdots b_{2g})^{t-1}T} & \mbox{for}~\overline{T}=\overline{b_2\cdots b_{2g+1}}
\end{cases}\\
&\xrightarrow{S_{(3,t)}}&\overline{(b_{2g}\cdots b_2)^t},
\end{eqnarray*}
where $\overline{b_1\cdots b_{4g}}\in\RR$. Since these two positions are symmetric, without loss of generality, let $\overline{V}$ be in the first position. Then $$\overline{x_nx_1}\subset\overline{T}=\overline{b_1\cdots b_{2g}}=\overline{x_{n-r+1}\cdots x_{n}x_1\cdots x_{2g-r}}\subset \overline{X^2},$$
where $1\leq r\leq 2g-1<n=|\overline{X}|$ since $\overline{T}$ does not contain a letter appearing more than once. Therefore,
\begin{eqnarray}\label{Eq. X}
		\overline{X}&=& \overline{b_{r+1}\cdots b_{2g}x_{2g-r+1}\cdots x_{n-r}b_1\cdots b_{r}}\nonumber\\
&=&\overline{\underbrace{b_{r+1}\cdots b_{2g}(b_2\cdots b_{2g})^{t-1}b_{2g+1}}_{x_1\cdots x_{n_1-1}=:A}~\underbrace{x_{n_1}\cdots x_{n-r}b_1\cdots b_r}_{x_{n_1}\cdots x_n=:B}},
        \end{eqnarray}
where $x_{n-r}\neq b_{2g+1},b_{4g}$ and  $x_{n_1}\neq b_{1},b_{2g+2}$ for
 \begin{equation}\label{position argument 1 |LLFR|=2g}
      n_1:=(2g-1)t-r+3\leq n-r.
    \end{equation}
(Proof of Eq. (\ref{position argument 1 |LLFR|=2g}): Since 
$x_{n-r+1}=b_1\notin \overline{b_{r+1}\cdots b_{2g}(b_2\cdots b_{2g})^{t-1}b_{2g+1}}=\overline{x_1\cdots x_{n_1-1}}\subset{\overline{V}},$ we have $n-r\geq n_1-1$. However, if $n-r=n_1-1$, then by Eq. (\ref{Eq. X}),
 $$\overline{x_{n-r}x_{n-r+1}}=\overline{x_{n_1-1}x_{n-r+1}}=\overline{b_{2g+1}b_{1}}\subset\overline{X},$$ which contradicts the irreducibility of $\overline{X}$. Therefore, $n_1\leq n-r$.)      
Hence, by Eq. (\ref{Eq. X}),
\begin{eqnarray}\label{Eq. H}
\overline{YX^2Y^{-1}}&=&\overline{Yx_1\cdots  x_{n_1}\cdots x_{n-r}Vx_{n_1}\cdots x_{n}Y^{-1}}\notag\\
&\xrightarrow{S_{(3,t)}}&\overline{Yx_1\cdots \underbrace{x_{n_1}\cdots  x_{n-r}(b_{2g}\cdots b_2)^t}_Wx_{n_1}\cdots x_{n}Y^{-1}}\notag\\
&=&\overline{Y\underbrace{b_{r+1}\cdots b_{2g}(b_2\cdots b_{2g})^{t-1}b_{2g+1}}_A W \underbrace{x_{n_1}\cdots x_{n-r}b_1\cdots b_r}_BY^{-1}}\notag\\
&=:&\overline{Z}.
\end{eqnarray}

	 If $\overline{Z}$ is irreducible, then by the same arguments as in the paragraph preceding Eq. (\ref{length=0 word2}), we have the same normal form of $x^k$ as in Eq. (\ref{length=0 word2}) for every $k\geq 1$:
\begin{equation}\label{eq. nf 4.2.31}
\nf(x^k)=\overline{Yx_1\cdots x_{n_1-1}W^{k-1}x_{n_1}\cdots x_{n}Y^{-1}},
\end{equation} 
where 
$$\overline W=\overline{x_{n_1}\cdots x_{n-r}(b_{2g}\cdots b_2)^t}=\nf(x_{n_1}\cdots x_{n}x_1\cdots x_{n_1-1}).$$ 

If $\overline{Z}$ is reducible (note that $\overline{Z}$ is freely reduced), then it contains a subword $\overline{U}$ of types $S_{(2)},S_{(3,t)}$ or $S_{(4,1)}$ containing $\overline{x_{n-r}b_{2g}}$ or $\overline{b_{2}x_{n_1}}$ in Eq. (\ref{Eq. H}). If $\overline{U}$ contains $\overline{b_{2}x_{n_1}}$, then in one hand, by Lemma \ref{lem frequently used}(2), 
$$\overline{U}=\overline{b_2 x_{n_1}\cdots}\subset\overline{(b_{2g}\cdots b_2)^tx_{n_1}\cdots x_nY^{-1}}$$ is reducible; in the other hand, since  $\overline{b_2\cdots b_{2g+1}x_{n_1}\cdots x_nY^{-1}} \subset \overline{XY^{-1}}$ is irreducible, by Corollary \ref{very useful cor}, $\overline{(b_{2g}\cdots b_2)^tx_{n_1}\cdots x_nY^{-1}}$ is irreducible, a contradiction. Thus, the reducible word $\overline{U}$ must contain $\overline{x_{n-r}b_{2g}}$ and we have
\begin{equation}
  \overline{U}=\overline{\cdots x_{n-r}b_{2g}}\subset \overline{Yx_1\cdots x_{n-r}b_{2g}}\subset\overline{Z},
\end{equation}
by Lemma \ref{lem frequently used}(2). Note that $\overline{Yx_1\cdots x_{n-r}}\subset \overline{YX}$ is irreducible. Then, according to Lemma \ref{lem for Xw types}, there are three cases for $\overline{U}$, which correspond to items (2, 3, 4) of Lemma \ref{lem for Xw types}, respectively.
     
	 Subcase (1.1). $\overline{U}=\overline{b_{4g}b_1\cdots b_{2g}}$ is of type $S_{(2,2g+1)}$. By Eq. (\ref{Eq. X}), we have $x_{n_1-1}=b_{2g+1}\notin \overline{U}$ and hence Eq. (\ref{Eq. H}) becomes
	 \begin{eqnarray}\label{eq 1.1 for Z in |LLFR|=2g}     
	 	\overline{Z}&=&\overline{Yx_1\cdots x_{n_2}Ub_{2g-1}\cdots b_2(b_{2g}\cdots b_2)^{t-1}x_{n_1}\cdots x_{n}Y^{-1}}\notag\\
	 	&\xrightarrow{S_{(2,2g+1)}}&\overline{Yx_1\cdots x_{n_2}b_{2g-1}\cdots b_2 [b_1b_{2g-1}]b_{2g-2}\cdots b_2(b_{2g}\cdots b_{2})^{t-1}x_{n_1}\cdots x_{n}Y^{-1}}\notag\\
        &=:&\overline{Z_1Z_2},
	 \end{eqnarray}
	 where $\overline{Z_1}=\overline{Yx_1\cdots x_{n_2}b_{2g-1}\cdots b_2 b_{1}}(n_2=n-r-2g\geq n_1-1)$ and
$$\overline{Z_2}=\overline{b_{2g-1}\cdots b_2(b_{2g}\cdots b_2)^{t-1}x_{n_1}\cdots x_nY^{-1}}\subset\overline{(b_{2g}\cdots b_2)^tx_{n_1}\cdots x_nY^{-1}}$$ 
are irreducible by Corollary \ref{very useful cor}. Furthermore, note that there is no reducible subword of $\overline{Z_1Z_2}$ of types $S_{(1)},S_{(2,2g+1)},S_{(3,t')}$ or $S_{(4,1)}$ that contains the bracketed (by $[\cdot]$) subword $\overline{b_1b_{2g-1}}$ in Eq. (\ref{eq 1.1 for Z in |LLFR|=2g}). Thus, $\overline{Z_1Z_2}$ and its subword $\overline{Wx_{n_1}\cdots x_{n_2}}\subset \overline{Z_1Z_2}$ are both irreducible, where
$$\overline{W}:=\overline{x_{n_1}\cdots x_{n_2}b_{2g-1}\cdots b_1b_{2g-1}\cdots b_2(b_{2g}\cdots b_2)^{t-1}}=\nf(x_{n_1}\cdots x_{n}x_{1}\cdots x_{n_{1}-1}),$$
for  $n_2\geq n_1-1$ ($\overline{x_{n_1}\cdots x_{n_2}}=1$ provided $n_2=n_1-1$).
Then by Lemma \ref{lem types of W}(2) and Lemma \ref{lem from 2-nd power to n-th power}, we obtain
	$$\nf(x^k)=\overline{Yx_1\cdots x_{n_1-1}W^{k-1}x_{n_1}\cdots x_{n}Y^{-1}}(k\geq 1).$$

Subcase (1.2). $\overline{U}=\overline{b_{4g}(b_1\cdots b_{2g-1})^{t_1}b_{2g}}$ is of type $S_{(3,t_1)}(t_1\geq 2)$. Since $x_{n_1-1}=b_{2g+1}\notin\overline{U}$, Eq. (\ref{Eq. H}) becomes
	 \begin{eqnarray}
	 	\overline{Z}&=&\overline{Yx_1\cdots x_{n_2}Ub_{2g-1}\cdots b_2(b_{2g}\cdots b_2)^{t-1}x_{n_1}\cdots x_{n}Y^{-1}}\notag\\
	 	&\xrightarrow{S_{(3,t_1)}}&\overline{Yx_1\cdots x_{n_2}(b_{2g-1}\cdots b_1)^{t_1}b_{2g-1}\cdots b_2(b_{2g}\cdots b_{2})^{t-1}x_{n_1}\cdots x_{n}Y^{-1}}\notag\\
        &=:&\overline{Z_1Z_2},
	 \end{eqnarray}
	 where $n_2=n-r-(2g-1)t_1-1\geq n_1-1$. After a similar argument following Eq. (\ref{eq 1.1 for Z in |LLFR|=2g}), we also obtain
	$$\nf(x^k)=\overline{Yx_1\cdots x_{n_1-1}W^{k-1}x_{n_1}\cdots x_{n}Y^{-1}}(k\geq 1),$$
  where 
  $$\overline{W}=\overline{x_{n_1}\cdots x_{n_2}(b_{2g-1}\cdots b_1)^{t_1}b_{2g-1}\cdots b_2(b_{2g}\cdots b_2)^{t-1}}=\nf(x_{n_1}\cdots x_{n}x_{1}\cdots x_{n_{1}-1}).$$
	
Subcase (1.3). \label{subcase1.3} $\overline{U}=\overline{(b_1\cdots b_{2g-1})^{t_2}b_{2g}}(b_1\succ b_{2g})$ is of type $S_{(4,t_2)}(t_2\geq 1)$ (not of types $S_{(2,2g+1)}$ or $S_{(3,t_1)}$). Since $x_{n_1-1}=b_{2g+1}\notin\overline{U}$, Eq. (\ref{Eq. H}) becomes
\begin{eqnarray}
	 	\overline{Z}&=&\overline{Yx_1\cdots x_{n_2}Ub_{2g-1}\cdots b_2(b_{2g}\cdots b_2)^{t-2}x_{n_1}\cdots x_{n}Y^{-1}}\notag\\
	 &\xrightarrow{S_{(4,t_2)}}&\overline{Yx_1\cdots x_{n_2}b_{2g}(b_{2g-1}\cdots b_1)^{t_2}b_{2g-1}\cdots b_2(b_{2g}\cdots b_{2})^{t-1}x_{n_1}\cdots x_{n}Y^{-1}}\notag\\
     &=:&\overline{Z_1Z_2},
	 \end{eqnarray}
	 where $n_2=n-r-(2g-1)t_2\geq n_1-1$ and $x_{n_2}\neq b_{4g}$. If $n_2=n_1-1$, then $$\overline{b_{2g+1}b_{1}}=\overline{x_{n_1-1}x_{n_2+1}}=\overline{x_{n_1-1}x_{n_1}}\subset\overline{X},$$
     which contradicts the irreducibility of $\overline{X}$. Therefore, $n_2\geq n_1$.
     After verifying the irreducibility of $\overline{Z_1Z_2}$, by Lemma \ref{lem types of W}(5) and Lemma \ref{lem from 2-nd power to n-th power}, we also obtain
	$$\nf(x^k)=\overline{Yx_1\cdots x_{n_1-1}W^{k-1}x_{n_1}\cdots x_{n}Y^{-1}}(k\geq 1),$$
where
$$\overline{W}=\overline{x_{n_1}\cdots x_{n_2}b_{2g}(b_{2g-1}\cdots b_1)^{t_2}b_{2g-1}\cdots b_2(b_{2g}\cdots b_2)^{t-1}}=\nf(x_{n_1}\cdots x_{n}x_{1}\cdots x_{n_{1}-1}).$$
	
\textbf{Case (2). $\overline{V}$ is of type $S_{(4,t_0)}(t_0\geq 1)$ (not of type $S_{(3)}$).}	Then $\overline{V}$ also has two possible positions as follows: 
	\begin{equation}\label{two position of V |LLFR|=2g, case2}
	\overline{V}=\left\{\begin{array}{lllll}
	     \overline{b_1(b_2\cdots b_{2g})^{t_0}}&=&\overline{T(b_2\cdots b_{2g})^{t_0-1}} &\xrightarrow{S_{(4,t_0)}}\overline{(b_{2g}\cdots b_2)^{t_0}b_1}\\
	     \overline{(b_1\cdots b_{2g-1})^{t_0}b_{2g}}&=&\overline{(b_1\cdots b_{2g-1})^{t_0-1}T}&\xrightarrow{S_{(4,t_0)}}\overline{b_{2g}(b_{2g-1}\cdots b_1)^{t_0}}
	\end{array}\right. ,
    \end{equation}
where $b_1\succ b_{2g}$,  $\overline{T}=\overline{b_1\cdots b_{2g}}$ and $\overline{b_1\cdots b_{4g}}\in\RR$. Since these two positions are symmetric, without loss of generality, let $\overline{V}$ be in the first position. Then, 
    $$\overline{x_nx_1}\subset\overline{T}=\overline{b_1\cdots b_{2g}}=\overline{x_{n-r+1}\cdots x_nx_1\cdots x_{2g-r}}\subset\overline{X^2},$$
    where $1\leq r\leq 2g-1<n=|\overline{X}|$ since $\overline{T}$ does not contain a letter appearing more than once. Therefore,
\begin{eqnarray}\label{Eq. X case 2 in |LLFR|=2g}
    \overline{X}&=&\overline{b_{r+1}\cdots b_{2g}x_{2g-r+1}\cdots x_{n-r}b_1\cdots b_{r}}\notag\\
    &=&\overline{\underbrace{b_{r+1}\cdots b_{2g}(b_2\cdots b_{2g})^{t_0-1}}_{x_1\cdots x_{n_1-1}=:A}\underbrace{x_{n_1}\cdots x_{n-r}b_1\cdots b_r}_{x_{n_1}\cdots x_n=:B}},
\end{eqnarray}
and
\begin{eqnarray}
    \overline{YX^2Y^{-1}}&=&\overline{Yx_1\cdots x_{n-r}Vx_{n_1}\cdots x_nY^{-1}}\label{Eq. YX2Y-1 case 2 in |LLFR|=2g}\\
   &\xrightarrow{S_{(4,t_0)}}&\overline{Yx_1\cdots x_{n-r}(b_{2g}\cdots b_2)^{t_0}b_1x_{n_1}\cdots x_nY^{-1}}\notag\\    &=&\overline{YAx_{n_1}\cdots x_{n-r}(b_{2g}\cdots b_2)^{t_0}b_1BY^{-1}}\notag\\
    &=:&\overline{Z},\label{Eq. H case 2 in |LLFR|=2g}
\end{eqnarray}
where $x_{n-r}\neq b_{2g+1},b_{4g}$ and $x_{n_1}\neq b_{4g},b_{2g+1}$ for
\begin{equation}\label{position argument 2 |LLFR|=2g}
    n_1:=(2g-1)t_0-r+2\leq n-r-1.
\end{equation}
Note that Eq. (\ref{position argument 2 |LLFR|=2g}) holds by $x_{n-r+1}=b_1\notin \overline{b_{r+1}\cdots b_{2g}(b_2\cdots b_{2g})^{t_0-1}}=\overline{x_1\cdots x_{n_1-1}}\subset \overline{V}$, and $\overline{Z}$ is freely reduced. 

If $\overline{Z}$ is irreducible, we have the following two cases: 
\begin{enumerate}
\item[$1^\circ$] When $\overline{x_{n_1}\cdots x_{n-r}}\neq \overline{b_2\cdots b_{2g-1}}$, by Lemma \ref{lem types of W}(3) and Lemma \ref{lem from 2-nd power to n-th power}, $\nf(x^k)(k\geq 1)$ is as follows:
    \begin{equation}\label{|LLFR|=2g word for subcase2.2.2}
        \nf(x^k)=\overline{Yx_{1}\cdots x_{n_1-1}W^{k-1}x_{n_1}\cdots x_nY^{-1}},
    \end{equation}
where $$\overline{W}=\overline{x_{n_1}\cdots x_{n-r}(b_{2g}\cdots b_2)^{t_0}b_1}=\nf(x_{n_1}\cdots x_nx_1\cdots x_{n_1-1}).$$
            
\item[$2^\circ$] When $\overline{x_{n_1}\cdots x_{n-r}}=\overline{b_2\cdots b_{2g-1}}$, then $\overline{X}=\overline{AB}$ is of type $\mathfrak{A}$ (see Definition \ref{defn word of types ABC}), $\nf(x^2)=\overline{YACDBY^{-1}}$ and $\nf(x^3)=\overline{YACWDBY^{-1}}$,
where $\overline{CD}=\nf(BA), \overline{W}=\nf(DC)$ and 
\begin{equation}\label{LLFR=2g special word 1}
		\left\lbrace \begin{array}{lll}
		\overline{A}&=&\overline{b_{r+1}\cdots b_{2g}(b_2\cdots b_{2g})^{t_0-1}}\\
		\overline{B}&=&\overline{b_2\cdots b_{2g-1}b_1\cdots b_r}\\
        \overline{C}&=&\overline{b_2\cdots b_{2g}}\\
        \overline{D}&=&\overline{b_{2g-1}\cdots b_2(b_{2g}\cdots b_2)^{t_0-1}b_1}\\
        \overline{W}&=&\overline{b_{2g-1}\cdots b_2(b_{2g}\cdots b_2)^{t_0} b_1}\\
	\end{array}\right. (t_0\geq 1).
    \end{equation}
Note that $|\overline{W}|=|\overline{CD}|$ and $\overline{W}$ is cyclically irreducible by Lemma \ref{lem types of W}(3). By Lemma \ref{lem from 2-nd power to n-th power}, $\nf(x^k)(k\geq 2)$ is as follows:
\begin{equation}\label{item(a) 1 for |LLFR|=2g}
    \nf(x^k)=\overline{YACW^{k-2}DBY^{-1}}.
\end{equation}
\end{enumerate}

If $\overline{Z}$ is reducible, then it contains a reducible subword $\overline{U}$ of types $S_{(2)},S_{(3,t_1)}(t_1\geq 2)$ or $S_{(4,t_2)}(t_2\geq 1)$. By Lemma \ref{lem frequently used}(3), there are two possible positions for $\overline{U}$:
\begin{equation*}
\overline{U}=\left\{\begin{array}{lll}
\overline{b_1x_{n_1}\cdots}&\subset&\overline{b_1x_{n_1}\cdots x_nY^{-1}}  \\
    \overline{\cdots x_{n-r}b_{2g}}&\subset&\overline{Yx_1\cdots x_{n-r}b_{2g}}     
\end{array}\right..  
\end{equation*}
        
Subcase (2.1).  $\overline{U}=\overline{b_1x_{n_1}\cdots}\subset \overline{Z}$. By Lemma \ref{lem for Xw types}, we have the following subcases. 

(a). If $\overline{U}$ is of type $S_{(2,2g+1)}$, then $\overline{U}=\overline{b_1b_2\cdots b_{2g}b_{2g+1}}$ and Eq. (\ref{Eq. YX2Y-1 case 2 in |LLFR|=2g}) becomes
\begin{eqnarray*}
    \overline{YX^2Y^{-1}}&=&\overline{Yx_1\cdots x_{n-r}b_1(b_2\cdots b_{2g})^{t_0}b_2\cdots b_{2g}b_{2g+1}\cdots }\\
    &\xrightarrow{S_{(3,t_0+1)}}&\overline{Yx_1\cdots x_{n-r}(b_{2g}\cdots b_{2})^{t_0+1}\cdots }.
\end{eqnarray*}
Therefore, $\overline{YX^2Y^{-1}}$ can be reduced by $S_{(3,t_0+1)}$ first, i.e., it should be discussed in Case (1). There are also some similar arguments in the following part of this proof. We will not give detailed analysis again.

 (b). If $\overline{U}$ is of type $S_{(3,t_1)}$, then $\overline{U}=\overline{b_1(b_2\cdots b_{2g})^{t_1}b_{2g+1}}$ and Eq. (\ref{Eq. YX2Y-1 case 2 in |LLFR|=2g}) becomes
$$\overline{YX^2Y^{-1}}=\overline{Yx_1\cdots x_{n-r}b_1(b_2\cdots b_{2g})^{t_0}(b_2\cdots b_{2g})^{t_1}b_{2g+1}\cdots }.$$
The argument is the same as that in the above subcase (a). 
			
(c). If $\overline{U}$ is of type $S_{(4,t_2)}$, then $\overline{U}=\overline{b_1(b_2\cdots b_{2g})^{t_2}}$ and Eq. (\ref{Eq. YX2Y-1 case 2 in |LLFR|=2g}) becomes
$$\overline{YX^2Y^{-1}}=\overline{Yx_1\cdots x_{n-r}b_1(b_2\cdots b_{2g})^{t_0}(b_2\cdots b_{2g})^{t_2}\cdots}.$$
Then, $\overline{V}$ should be of type $S_{(4,t_0+t_2)}$, which contradicts the maximality of $t_0$ (see Remark \ref{rem of maximality}). 
        
       Subcase (2.2).  $\overline{Z}$ does not contain a reducible subword of the form $\overline{b_1x_{n_1}\cdots }$. Then 
       $$\overline{U}=\overline{\cdots  x_{n-r}b_{2g}}\subset \overline{Z}.$$
       By Lemma \ref{lem for Xw types}, we have the following subcases.
        
	(a). If $\overline{U}$ is of type $S_{(2,2g+1)}$, then $\overline{U}=\overline{b_{4g}b_1\cdots b_{2g-1}b_{2g}}$ and Eq. (\ref{Eq. YX2Y-1 case 2 in |LLFR|=2g}) becomes
        \begin{eqnarray*}
                \overline{YX^2Y^{-1}}&=&\overline{\cdots b_{4g}(b_1\cdots b_{2g-1})^2b_{2g}(b_2\cdots b_{2g})^{t_0-1}x_{n_1}\cdots x_{n}Y^{-1}}\\
                &\xrightarrow{S_{(3,2)}}&\overline{\cdots (b_{2g-1}\cdots b_1)(b_{2g-1}\cdots b_2)b_1(b_2\cdots b_{2g})^{t_0-1}x_{n_1}\cdots x_nY^{-1}}\\
				&\xrightarrow{S_{(4,t_0-1)}}&\overline{\cdots(b_{2g-1}\cdots b_1)(b_{2g-1}\cdots b_2)(b_{2g}\cdots b_2)^{t_0-1}b_1x_{n_1}\cdots x_nY^{-1}}.
        \end{eqnarray*}
    Therefore, $\overline{YX^2Y^{-1}}$ can be reduced first by $S_{(3,2)}$  and then by $S_{(4,t_0-1)}$, i.e., it should be discussed in Subcase (1.3). (Actually, this case is symmetric to Subcase (1.3). )
			
(b). If $\overline{U}$ is of type $S_{(3,t_1)}$, then $\overline{U}=\overline{b_{4g}(b_1\cdots b_{2g-1})^{t_1}b_{2g}}$ and Eq. (\ref{Eq. YX2Y-1 case 2 in |LLFR|=2g}) becomes
    $$\overline{YX^2Y^{-1}}=\overline{\cdots b_{4g}(b_1\cdots b_{2g-1})^2b_{2g}(b_2\cdots b_{2g})^{t_0-1}x_{n_1}\cdots x_{n}Y^{-1}}.$$
    The argument is the same as that in the above subcase (a).

(c). If $\overline{U}$ is of type $S_{(4,t_2)}$ (not of type $S_{(3,t_2)}$), then $\overline{U}=\overline{(b_1\cdots b_{2g-1})^{t_2}b_{2g}}$ and Eq. (\ref{Eq. YX2Y-1 case 2 in |LLFR|=2g}) becomes
    $$\overline{YX^2Y^{-1}}=\overline{\cdots(b_1\cdots b_{2g-1})^{t_2}b_1(b_2\cdots b_{2g})^{t_0}x_{n_1}\cdots x_nY^{-1}}.$$
    Since $x_{n_1-1}=b_{2g}\notin\overline{(b_1\cdots b_{2g-1})^{t_2}}$, we get $n_2:=n-r-(2g-1)t_2\geq n_1-1$. Then Eqs. (\ref{Eq. X case 2 in |LLFR|=2g})(\ref{Eq. H case 2 in |LLFR|=2g}) become
        \begin{eqnarray}
                \overline{X}&=&\overline{\underbrace{b_{r+1}\cdots b_{2g}(b_2\cdots b_{2g})^{t_0-1}}_A~\underbrace{x_{n_1}\cdots x_{n_2}(b_1\cdots b_{2g-1})^{t_2}b_1\cdots b_r}_B},\label{X specail form a}\\
                \overline{Z}&=&\overline{Yx_1\cdots x_{n_2}(b_1\cdots b_{2g-1})^{t_2}(b_{2g}\cdots b_2)^{t_0}b_1 x_{n_1}\cdots x_nY^{-1}}\notag\\
                &\xrightarrow{S_{(4,t_2)}}&\overline{Yx_1\cdots x_{n_2}b_{2g}(b_{2g-1}\cdots b_1)^{t_2}b_{2g-1}\cdots b_2(b_{2g}\cdots b_2)^{t_0-1}b_1 x_{n_1}\cdots x_nY^{-1}}\notag\\
                &=:&\overline{H},
        \end{eqnarray}
    where $t_0,t_2\geq 1$ and $x_{n_2}\neq b_{2g+1},b_{4g}$. Then, we have:
   \begin{claim}\label{claim 1}
      $\overline{H}$ is irreducible.
   \end{claim}  
\begin{proof}[Proof of Claim \ref{claim 1}]       
Since $\overline{Z}$ does not contain a reducible subword of the form $\overline{b_1x_{n_1}\cdots }$, we can easily obtain the irreducibility of $\overline{Z_2}:=\overline{(b_{2g}\cdots b_2)^{t_0}b_1x_{n_1}\cdots x_nY^{-1}}$. Then, $\overline{Z}=\overline{Z_1Z_2}$ with $\overline{Z_1}=\overline{Yx_1\cdots x_{n-r}}$ irreducible. By Lemma \ref{lem frequently used}(3), there are two possible positions of reducible subwords in $\overline{K}$ of types $S_{(2,2g+1)},S_{(3)}$ or $S_{(4)}$: 

(i) ~$\overline{\cdots x_{n_2}b_{2g}}\subset\overline{Yx_1\cdots x_{n_2}b_{2g}}$; 

(ii) $\overline{b_1 b_{2g-1}\cdots }\subset\overline{ b_1 b_{2g-1}\cdots b_2(b_{2g}\cdots b_2)^{t_0-1}b_1x_{n_1}\cdots x_n Y}$. 

For (i), we will see it is impossible after a discussion similar to Subcase (2.1). For (ii), it is also impossible because the latter word is clearly irreducible. Hence, $\overline{H}$ is irreducible.
    \end{proof}
    
    Therefore, we have the following two cases.
        \begin{enumerate}
            \item[$1^\circ$] When $\overline{x_{n_1}\cdots x_{n_2}}\neq \overline{b_2\cdots b_{2g-1}}$, by Lemma \ref{lem types of W}(6) and Lemma \ref{lem from 2-nd power to n-th power},  $\nf(x^k)(k\geq 1)$ is as follows:
			$$\nf(x^k)= \overline{Yx_1\cdots x_{n_1-1}W^{k-1}x_{n_1}\cdots x_nY^{-1}},$$
            where the cyclically irreducible word $$\overline{W}=\overline{x_{n_1}\cdots x_{n_2}b_{2g}(b_{2g-1}\cdots b_1)^{t_2}b_{2g-1}\cdots b_2(b_{2g}\cdots b_2)^{t_0-1}b_1 }=\nf(x_{n_1}\cdots x_nx_1\cdots x_{n_1-1}).$$
            
            \item[$2^\circ$]  When $\overline{x_{n_1}\cdots x_{n_2}}=\overline{b_2\cdots b_{2g-1}}$, then by Eq. (\ref{X specail form a}), $\overline{X}=\overline{AB}$ is of type $\mathfrak{A}$ (see Definition \ref{defn word of types ABC}), $\nf(x^2)=\overline{YACDBY^{-1}}$ and $\nf(x^3)=\overline{YACWDBY^{-1}}$, 
            where 
            $$\overline{CD}=\nf(BA), \quad\overline{W}=\nf(DC)$$ and
            \begin{equation}\label{LLFR=2g special word 2}
				\left\lbrace \begin{array}{lll}
					\overline{A}&=&\overline{b_{r+1}\cdots b_{2g}(b_2\cdots b_{2g})^{t_0-1} }\\
					\overline{B}&=&\overline{b_{2}\cdots b_{2g-1}(b_1\cdots b_{2g-1})^{t_2}b_1\cdots b_r}\\
                    \overline{C}&=&\overline{b_2\cdots b_{2g}}\\
                    \overline{D}&=&\overline{(b_{2g-1}\cdots b_1)^{t_2}b_{2g-1}\cdots b_2 (b_{2g}\cdots b_2)^{t_0-1}b_1}\\
                    \overline{W}&=&\overline{(b_{2g-1}\cdots b_1)^{t_2}b_{2g-1}\cdots b_2(b_{2g}\cdots b_2)^{t_0}b_1}
				\end{array}\right.~(t_0, t_2\geq 1).
			\end{equation}   
            Note that $|\overline{W}|=|\overline{CD}|$ and $\overline{W}$ is cyclically irreducible by Lemma \ref{lem types of W}(7). Then, by Lemma \ref{lem from 2-nd power to n-th power},  for every $k\geq 2$,
            \begin{equation}\label{item(a) 2 for |LLFR|=2g}
                \nf(x^k)=\overline{YACW^{k-2}DBY^{-1}}.
            \end{equation} 
            \end{enumerate}  

Now we obtain Proposition \ref{prop classification of LLFR of X^2}(\ref{XX reducible}a) by concluding Eqs. (\ref{item(a) 1 for |LLFR|=2g})(\ref{item(a) 2 for |LLFR|=2g}) and their symmetric cases. The proof of Proposition \ref{prop classification of LLFR of X^2}(\ref{XX reducible}) ($|\overline{T}|=2g$) is complete by concluding all the above cases.

\subsection{\texorpdfstring{Proof of Proposition \ref{prop classification of LLFR of X^2}(\ref{XX reducible}) ($|\overline{T}|\in\{2g+1,\dots,4g-2\}$)}{Proof of Proposition \ref{prop classification of LLFR of X^2}(\ref{XX reducible}) (|overline{T}| in \{2g+1,...,4g-2\})}}\label{Appendix A2}

Clearly, the $\llfr$ $\overline{T}\subset\overline{YX^2Y^{-1}}$ containing $\overline{x_nx_1}$ is a reducible subword of type $S_{(2)}$ and $\overline{T}$ does not contain a letter appearing more than once. Therefore, we assume 
\begin{eqnarray}\label{Eq. T=bMb}
    \overline{T}&=&\overline{x_{n-r+1}\cdots x_nx_{1}\cdots x_{s}}\notag\\
    &=&\overline{b_1\cdots b_rb_{r+1}\cdots b_{r+s}}\notag\\    &\xrightarrow{S_{(2,r+s)}}&\overline{b_{2g}\underbrace{b_{2g-1}\cdots b_{r+s-2g+2}}_{M}b_{r+s-2g+1}}\notag\\
    &=&\overline{b_{2g}Mb_{k-2g+1}} \qquad(k:=r+s=|\overline{T}|\leq n)
\end{eqnarray}
for some $\overline{b_1\cdots b_{4g}}\in\RR$, where $1\leq r,s\leq 2g;~2g+1\leq k \leq 4g-2$ and $s\leq  n-r$. Note that $\overline{M}=1$ provided $k= 4g-2$. Then we have 
\begin{eqnarray}\label{Eq. Z om 4.22T}
\overline{YX^2Y^{-1}}&=& \overline{\underbrace{Yb_{r+1}\cdots b_{k}x_{s+1}\cdots x_{n-r}}_{Z_{1}}~ T~\underbrace{x_{s+1}\cdots x_{n-r} b_1\cdots b_rY^{-1}}_{Z_2}}\notag\\
&\xrightarrow{S_{(2,k)}}&
\overline{Z_1b_{2g}Mb_{k-2g+1}Z_2}\notag\\
&=:&\overline{Z},
\end{eqnarray}
where 
\begin{equation}\label{condition 1 for |LLFR|=2g+1 to 4g-2}
    x_{n-r}\neq b_{4g},b_{2g+1};\quad x_{s+1}\neq b_{k+1},b_{k-2g}  
\end{equation}
and 
$$\overline{Z_1}=\overline{Yx_1\cdots x_{n-r}}, \quad \overline{Z_2}=\overline{x_{s+1}\cdots x_nY^{-1}}$$
are both irreducible.
 
According to Lemma \ref{lem frequently used}(1), every potential reducible subword of $\overline{Z}$ of type $S_{(i)}$ must be one of the following two forms: 
\begin{eqnarray*}
\overline{\cdots x_{n-r}b_{2g}} &\subset& \overline{Z_1b_{2g}},\\
 \overline{b_{k-2g+1}x_{s+1}\cdots}&\subset& \overline{ b_{k-2g+1}Z_2}.
 \end{eqnarray*}
 Therefore, we only need to consider that $\overline{Z_1b_{2g}}$ and $\overline{b_{k-2g+1}Z_2}$  belong to which items of Lemma \ref{lem for Xw types} respectively. To do this, we use a pair $(i,j) $ to denote each combination:
$$(i,j) ~\mbox{means} ~\overline{Z_1b_{2g}}~ \mbox{and} ~\overline{b_{k-2g+1}Z_2} ~\mbox{satisfying items} ~(i) ~\mbox{and} ~(j)~ \mbox{of Lemma \ref{lem for Xw types}}, \mbox{respectively},$$
where $2\leq i,j\leq 5$. (Notice that $\overline{Z}$ is freely reduced. It does not contain a subword of type $S_{(1)}$, and thus item (1) of Lemma \ref{lem for Xw types} is excluded). Then, we have 16 combinations to examine. Fortunately, because of symmetry, it suffices to assume $i\leq  j$ and hence we only have the following 10 items to consider. 
For the convenience of subsequent discussions, we denote $\overline{YX}$ and $\overline{XY^{-1}}$ as follows. See Figure \ref{fig: D,E,X,Y} for a direct impression.
\begin{align*}
   \overline{D}&:=\overline{d_1\cdots d_{m+n}}=\overline{y_1\cdots y_mx_1\cdots x_n}\quad=\overline{YX},\\
   \overline{E}&:=\overline{e_1\cdots e_{m+n}}=\overline{x_1\cdots x_ny_m^{-1}\cdots y_1^{-1}}=\overline{XY^{-1}}. 
\end{align*}
\begin{figure}[ht]
    \centering
\includegraphics[width=0.85\linewidth]{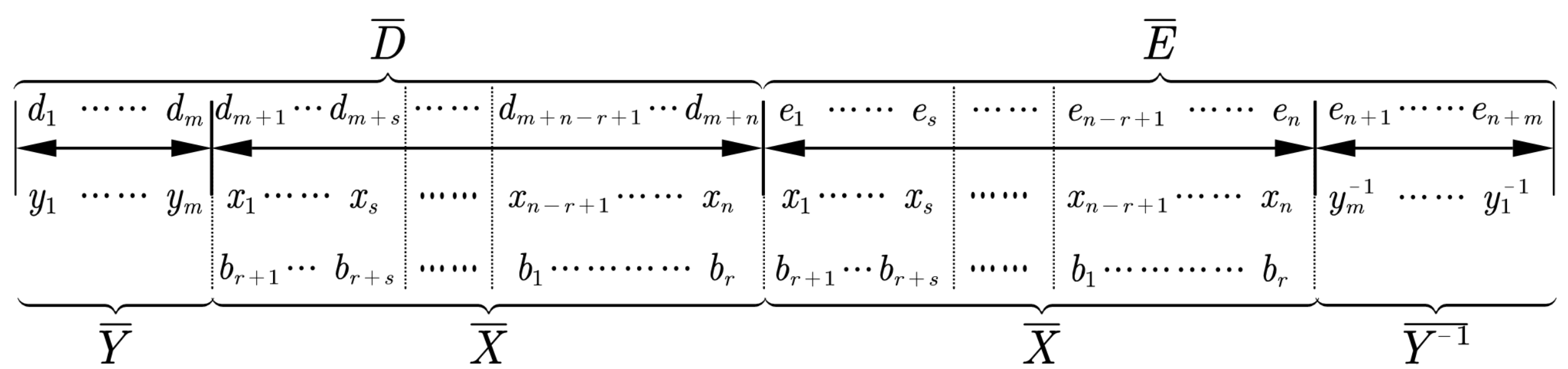}
    \caption{The correspondence among $\overline{D},\overline{E},\overline{X},\overline{Y},\overline{Y^{-1}}$ and their subscripts.}
    \label{fig: D,E,X,Y}
\end{figure}

In the following $10$ reduction cases, the integers $n_1$ and $n_2$ only depend on $j$ and $i$, respectively. See Table \ref{tab: n_1,n_2 in |LLFR| in 2g+1...4g-2}. (For instance, in the case of $(2,3)$, we have $n_1=s+(2g-1)t'+2$ and $n_2=m+n-r-2g$.)

\begin{table}[ht]
    \centering
\begin{tabular}{|l|l|}
\hline
    $j$ & $n_1$ \\
    \hline
    $2$ & $s+2g+1$\\
     \hline
    $3$ &$s+(2g-1)t'+2$ for $t'\geq 2$ \\
     \hline
    $4$& $s+(2g-1)t'+1$ for $t'\geq 1$\\
     \hline
    $5$ & $s+1$\\
     \hline
\end{tabular}
    \begin{tabular}{|l|l|}
    \hline
    $i$  & $n_2$  \\
    \hline
    $2$ & $m+n-r-2g$ \\
    \hline 
    $3$ & $m+n-r-(2g-1)t-1$ for $t\geq 2$ \\
    \hline
    $4$  &$m+n-r-(2g-1)t$ for $t\geq 1$ \\
    \hline
    $5$&$m+n-r$\\
    \hline
    \end{tabular}
    \newline
    \caption{The correspondences between $n_1$ and $j$; $n_2$ and $i$.}
    \label{tab: n_1,n_2 in |LLFR| in 2g+1...4g-2}
\end{table}

Then by Eq. (\ref{Eq. T=bMb}), Eq. (\ref{Eq. Z om 4.22T}) becomes:

\begin{enumerate}
\item[$(2,2)$.]
\begin{eqnarray*}
\overline{YX^2Y^{-1}}&=&\overline{d_1\cdots d_{n_2}  b_{4g}b_1\cdots b_{2g-1}Tb_{k-2g+2}\cdots b_{k}b_{k+1}e_{n_1}\cdots e_{m+n}},\\
\overline{Z}&=&\overline{d_1\cdots d_{n_2}  b_{4g}b_1\cdots b_{2g-1}b_{2g}M b_{k-2g+1} b_{k-2g+2}\cdots b_{k}b_{k+1}e_{n_1}\cdots e_{m+n}}\\
&\xrightarrow{2\cdot S_{(2)}}&\overline{\underbrace{d_1\cdots d_{n_2}b_{2g-1}\cdots b_{1}}_{L_2}M\underbrace{b_{k}\cdots b_{k-2g+2}e_{n_1}\cdots e_{m+n}}_{R_2}}.
\end{eqnarray*}
	
\item[$( 2,3 )$.] 
\begin{eqnarray*}
\overline{YX^2Y^{-1}}&=&\overline{d_1\cdots d_{n_2}b_{4g}b_1 \cdots b_{2g-1} T (b_{k-2g+2}\cdots b_{k})^{t'}b_{k+1}e_{n_1}\cdots e_{m+n}},\\
\overline{Z}&=&\overline{d_1\cdots d_{n_2}b_{4g}b_1 \cdots b_{2g-1} b_{2g}M b_{k-2g+1} (b_{k-2g+2}\cdots b_{k})^{t'}b_{k+1}e_{n_1}\cdots e_{m+n}}\\
&\xrightarrow{S_{(2)}+S_{(3)}}&\overline{L_2M\underbrace{(b_{k}\cdots b_{k-2g+2})^{t'}e_{n_1}\cdots e_{m+n}}_{R_3}}.
\end{eqnarray*}	

\item[$( 2,4 )$.] 
\begin{eqnarray*}
\overline{YX^2Y^{-1}}&=&\overline{d_1\cdots d_{n_2}b_{4g}b_1\cdots b_{2g-1}T(b_{k-2g+2}\cdots b_{k})^{t'}e_{n_1}\cdots e_{m+n}},\\
\overline{Z}&=&\overline{d_1\cdots d_{n_2}b_{4g}b_1\cdots b_{2g-1}b_{2g}M b_{k-2g+1}(b_{k-2g+2}\cdots b_{k})^{t'}e_{n_1}\cdots e_{m+n}}\\
&\xrightarrow{S_{(2)}+S_{(4)}}&\overline{L_2M\underbrace{(b_{k}\cdots b_{k-2g+2})^{t'}b_{k-2g+1}e_{n_1}\cdots e_{m+n} }_{R_4}}.
\end{eqnarray*}

\item[$( 2,5 )$.] 
\begin{eqnarray}\label{Eq. MR5}
\overline{YX^2Y^{-1}}&=&\overline{d_1\cdots d_{n_2}b_{4g} b_1\cdots b_{2g-1}Te_{n_1}\cdots e_{m+n}},\notag\\
\overline{Z}&=&\overline{d_1\cdots d_{n_2}b_{4g} b_1\cdots b_{2g-1}b_{2g}M \underbrace{b_{k-2g+1} e_{n_1}\cdots e_{m+n}}_{R_5}}\notag\\
&\xrightarrow{S_{(2)}}&\overline{L_2MR_5}.
\end{eqnarray}

\item[$( 3,3 ) $.] 
\begin{eqnarray*}
\overline{YX^2Y^{-1}}&=&\overline{d_1\cdots d_{n_2}b_{4g}(b_1\cdots b_{2g-1})^tT (b_{k-2g+2}\cdots b_{k})^{t'}b_{k+1}e_{n_1}\cdots e_{m+n}},\\
\overline{Z}&=&\overline{d_1\cdots d_{n_2}b_{4g}(b_1\cdots b_{2g-1})^tb_{2g}M b_{k-2g+1} (b_{k-2g+2}\cdots b_{k})^{t'}b_{k+1}e_{n_1}\cdots e_{m+n}}\\
&\xrightarrow{2\cdot S_{(3)}}&\overline{\underbrace{d_1\cdots d_{n_2} (b_{2g-1}\cdots b_1)^t}_{L_3}MR_3}.
\end{eqnarray*}	

\item[$( 3,4 ) $.] 
\begin{eqnarray*}
\overline{YX^2Y^{-1}}&=&\overline{d_1\cdots d_{n_2}b_{4g}(b_{1}\cdots b_{2g-1})^tT(b_{k-2g+2} \cdots b_{k})^{t'}e_{n_1}\cdots e_{m+n}},\\
\overline{Z}&=&\overline{d_1\cdots d_{n_2}b_{4g}(b_{1}\cdots b_{2g-1})^tb_{2g}M b_{k-2g+1}(b_{k-2g+2} \cdots b_{k})^{t'}e_{n_1}\cdots e_{m+n}}\\
&\xrightarrow{S_{(3)}+S_{(4)}}&\overline{L_3MR_4}.
\end{eqnarray*}	

\item[$( 3,5 )$.] 
\begin{eqnarray*}
\overline{YX^2Y^{-1}}&=&\overline{ d_1\cdots d_{n_2} b_{4g}(b_{1}\cdots b_{2g-1})^t  Te_{n_1}\cdots e_{m+n}},\\
\overline{Z}&=&\overline{ d_1\cdots d_{n_2} b_{4g}(b_{1}\cdots b_{2g-1})^t  b_{2g}MR_5}\\
&\xrightarrow{S_{(3)}}&\overline{L_3MR_5}.
\end{eqnarray*}

\item[$( 4,4 ) $.] 
\begin{eqnarray*}
\overline{YX^2Y^{-1}}&=&\overline{d_1\cdots d_{n_2}(b_1\cdots b_{2g-1})^t   T (b_{k-2g+2}\cdots b_{k})^{t'}e_{n_1}\cdots e_{m+n}},\\
\overline{Z}&=&\overline{d_1\cdots d_{n_2}(b_1\cdots b_{2g-1})^t  b_{2g}Mb_{k-2g+1}(b_{k-2g+2}\cdots b_{k})^{t'}e_{n_1}\cdots e_{m+n}}\\
&\xrightarrow{2\cdot S_{(4)}}&\overline{\underbrace{d_1\cdots d_{n_2} b_{2g}(b_{2g-1}\cdots b_1)^t}_{L_4}MR_4}.
\end{eqnarray*}

\item[$( 4,5 ) $.] 
\begin{eqnarray*}
\overline{YX^2Y^{-1}}&=&\overline{ d_1\cdots d_{n_2} (b_{1}\cdots b_{2g-1})^t  Te_{n_1}\cdots e_{m+n}},\\
\overline{Z}&=&\overline{d_1\cdots d_{n_2}(b_{1}\cdots b_{2g-1})^t b_{2g}MR_5}\\
&\xrightarrow{S_{(4)}}&\overline{L_4MR_5}.
\end{eqnarray*}

\item[$( 5,5 ) $.] 
\begin{eqnarray*}
\overline{YX^2Y^{-1}}&=&\overline{d_1\cdots d_{n_2} Te_{n_1}\cdots e_{m+n} }~,\\
\overline{Z}&=&\overline{ \underbrace{d_1\cdots d_{n_2}b_{2g}}_{L_5}MR_5}.
\end{eqnarray*}
\end{enumerate}

Note that in the above 10 cases, by  Lemma \ref{lem for Xw types}, $L_i$, $R_i$ and $M$ are irreducible for all $2\leq i\leq 5$. Through some arguments as before, we can verify:
\begin{claim}\label{claim 2}
    $\overline{H_{i,j}}:=\overline{L_iMR_j}$ are irreducible for all $2\leq i\leq j\leq 5$.
\end{claim}
\begin{proof}[Proof of Claim \ref{claim 2}]
Suppose $\overline{H_{i,j}}$ contains a reducible subword $\overline{V}$ of types $S_{(2,2g+1)},S_{(3,t_0)}$ or $S_{(4,1)}$.
In the case $(5,5)$, by Lemma \ref{lem frequently used}(1), $\overline{V}$ is a subword of $\overline{L_5}$ or $\overline{R_5}$, which contradicts the irreducibility of $\overline{L_5}$ and $\overline{R_5}$. Hence, $\overline{H_{5,5}}$ is irreducible.

In the cases $i\neq 5$, we have $\overline{H_{i,j}}=\overline{L_iMR_j}$ with the  irreducible words
$$\overline{L_i}= \overline{\cdots b_{2g-1}\cdots b_{1}}, \quad \overline{M}= \overline{ b_{2g-1}\cdots b_{k-2g+2}}.$$
Now we have three cases.
    \begin{enumerate}
    \item  If $j=5$, by Eq. (\ref{Eq. MR5}),
            $$ \overline{MR_5}=\overline{b_{2g-1}\cdots b_{k-2g+1}e_{n_1}\cdots e_{m+n}}\subset\overline{Z}.$$
Therefore, if there  is a reducible subword $\overline{V'}\subset\overline{MR_5}$ of type $S_{(i)}$, then $\overline{V'}$ is also a reducible subword of $\overline{Z}$ of type $S_{(i)}$, and hence has a form $\overline{\cdots d_{m+n-r}b_{2g}}(\nsubseteq \overline{MR_5})$ or $\overline{b_{k-2g+1}e_{s+1}\cdots}(\subset \overline{R_5})$ by Lemma \ref{lem frequently used}(1), which contradicts the irreducibility of $\overline{R_5}$. Hence, $\overline{MR_5}$ is irreducible, and thus we have $\overline{V}=\overline{\cdots b_1b_{2g-1}\cdots}\subset\overline{L_iMR_5}$. It implies $$\overline{V}=\overline{b_{2g}(b_{2g-1}\cdots b_1)^{t_0}b_{4g}}(t_0\geq 2),$$
which leads to $\overline{MR_5}=\overline{(b_{2g-1}\cdots b_{1})^{t_1}b_{4g}\cdots}~(t_1\geq 1)$ and thus $b_{k-2g}=e_{n_1}=e_{s+1}=x_{s+1}$ contradicts Eq. (\ref{condition 1 for |LLFR|=2g+1 to 4g-2}).

    \item If $j\neq 5$ and $k\leq 4g-3$, we have $\overline{H_{i,j}}=\overline{L_iMR_j}$ with $\overline{L_i},\overline{M}$ and $\overline{R_j}$ irreducible, where
   $$\overline{R_j}= \overline{b_{k}\cdots b_{k-2g+2}\cdots}.$$
    Then, there are two possible positions for 
    $$\overline{V}\subset\overline{L_iMR_j}=\overline{\cdots b_{2g-1}\cdots b_1b_{2g-1}\cdots b_{k-2g+2}b_{k}\cdots b_{k-2g+2}\cdots}$$ as follows:
    \begin{equation*}
        \overline{V}=\left\{\begin{array}{lll}
             \overline{\cdots b_1b_{2g-1}\cdots}\\
              \overline{\cdots b_{k-2g+2}b_{k}\cdots}
        \end{array}\right. ~.
    \end{equation*}
    Therefore, we have
    \begin{equation*}
        \overline{V}=\left\{\begin{array}{l}
             \overline{b_{2g}(b_{2g-1}\cdots b_1)^{t_1}b_{4g}} \\
             \overline{b_{k+1}(b_{k}\cdots b_{k-2g+2})^{t_1}b_{k-2g+1}}
        \end{array}\right. ~,
    \end{equation*}
which lead to $b_{k}=b_{k-2g+1}$ or $b_1=b_{2g}$. It's contradictory.

\item If $j\neq 5$ and $k= 4g-2$, then $\overline{M}=1$, and hence we have $\overline{H_{i,j}}=\overline{L_iR_j}$ with $\overline{L_i}$ and $\overline{R_j}$ irreducible, where
  $$\overline{R_j}=\overline{b_{4g-2}\cdots b_{2g}\cdots}.$$
Then, we have $\overline{V}$ as follows:
    $$\overline{V}=\overline{\cdots b_1b_{4g-2}\cdots}\subset\overline{\cdots b_{2g-1}\cdots b_1b_{4g-2}\cdots b_{2g}\cdots}.$$
Note that there is no such a reducible word of types $S_{(2)},S_{(3)}$ or $S_{(4)}$ containing $\overline{b_1b_{4g-2}}$. It is contradictory.
\end{enumerate}

In conclusion, $\overline{H_{i,j}}$ contains no reducible subwords and hence is irreducible for all $2\leq i\leq j\leq 5$.  
\end{proof}

Finally, by seeing Figure \ref{fig: D,E,X,Y}, in all the above 10 cases, we can obtain $n\geq n-r+1>n_1-1$ by verifying
$$e_{n-r+1}=x_{n-r+1}=b_1\notin\overline{e_1\cdots e_{n_1-1}}.$$
Then we have $e_{n_1-1}\notin \overline{Y^{-1}}$ and thus $e_{n_1-1}=x_{n_1-1}=d_{m+n_1-1}\in\overline{X}$. Furthermore, by verifying
$$d_{m+n_1-1}=e_{n_1-1}\notin\overline{d_{n_2+1}\cdots d_{m+n}},$$
we can obtain $1\leq n_1-1\leq n_2-m<n$ and thus $e_{n_1-1}=x_{n_1-1}$ and $d_{n_2}=x_{n_2-m}$. Hence, all of these $\overline{H_{i,j}}=\overline{L_iMR_j}$ have a uniform form:
\begin{eqnarray*}
    \overline{H_{i,j}}&=&\overline{\underbrace{d_1\cdots d_{n_2}L'_i}_{L_i}M\underbrace{R'_je_{n_1}\cdots e_{m+n}}_{R_j}}\\
    &=&\overline{Yx_1\cdots x_{n_2-m}L'_iMR'_j x_{n_1}\cdots  x_nY^{-1}}\\
    &=&\overline{Yx_1\cdots x_{n_1-1}\underbrace{x_{n_1}\cdots x_{n_2-m}L'_iMR'_j}_{W_{i,j}}x_{n_1}\cdots  x_nY^{-1}},
\end{eqnarray*}
where $\overline{M}=\overline{b_{2g-1}b_{2g-2}\cdots b_{k-2g+2}}$ ($\overline{M}=1$ when $k= 4g-2$), and $\overline{L'_i},\overline{R'_j}$ are of the forms corresponding to $(i,j)$ in tables in Definition \ref{defn irreducible word pair of type (i,j)}, and
 $$\overline{W_{i,j}}:=\overline{\underbrace{x_{n_1}\cdots x_{n_2-m}}_UL'_iMR'_j}=\nf(x_{n_1}\cdots x_nx_1\cdots x_{n_1-1}).$$

Note that in the case of $(i,j)$ with $2\leq i\leq j\leq 3$ and $k=4g-2$,  we have 
$$\overline{X}=\overline{b_{r+1}\cdots b_{4g-2}(b_{2g}\cdots b_{4g-2})^{t'}b_{4g-1}\underbrace{x_{n_1}\cdots x_{n_2-m}}_Ub_{4g}(b_1\cdots b_{2g-1})^t b_1\cdots b_r },$$
where $t,t'\geq 1$, and $\overline{U}\neq 1$. (Indeed, if $\overline{U}=1$, then $\overline{X}$ contains a reducible subword $\overline{b_{2g}\cdots b_{4g-2}b_{4g-1}b_{4g}}$, which contradicts the irreducibility of $\overline{X}$.) Therefore, $\overline{U}$  satisfies the special condition ``$\overline{U}\neq 1$  when $k=4g-2$ and $ 2\leq i\leq j\leq 3$'' in Definition \ref{defn irreducible word pair of type (i,j)}.
Furthermore, in each case $(i,j)$($2\leq i\leq j\leq 5$), the irreducible word $\overline{W_{i,j}}$ is of type $\overline{U}(i,j)$ (see Definition \ref{defn irreducible word pair of type (i,j)}), and $$\overline{W_{i,j}\underbrace{x_{n_1}\cdots x_{n_2-m}}_U}\subset\overline{H_{i,j}}$$ is also irreducible. Therefore, by Lemma \ref{lem types of W for lengh =2g+1 to 4g-2}, $\overline{W_{i,j}^2}$ is irreducible for all cases. Now by Lemma \ref{lem from 2-nd power to n-th power}, for any $k\geq 1$, we obtain a uniform expression
 $$\nf(x^k)=\overline{Yx_1\cdots x_{n_1-1}W^{k-1}x_{n_1}\cdots x_nY^{-1}},$$
where $\overline{W}=\overline{W_{i,j}}=\nf(x_{n_1}\cdots x_nx_1\cdots x_{n_1-1})$.
The proof of Proposition \ref{prop classification of LLFR of X^2}(\ref{XX reducible}) ($|\overline{T}|\in\{2g+1,\dots,4g-2\}$) is complete.

\subsection{\texorpdfstring{Proof of Proposition \ref{prop classification of LLFR of X^2}(\ref{XX reducible}) ($|\overline{T}|=4g-1$)}{Proof of Proposition \ref{prop classification of LLFR of X^2}(\ref{XX reducible}) (|overline{T}| = 4g-1)}} \label{Appendix A3}
Since this proof is quite lengthy, we first present an outline of the proof for the sake of readability.
Recall that $x$ is a nontrivial element of the surface group $G$ and $$\nf(x)=\overline{YXY^{-1}}$$ with $\overline{X}=\overline{x_1\cdots x_n}(n\geq 1)$ cyclically freely reduced and $\overline{Y}=\overline{y_1\cdots y_m}(m\geq 0)$ possibly empty.

\subsubsection{Outline of the proof.}
Our target is to obtain the normal form $\nf(x^k)$ of the element $x^k$.
\begin{enumerate}
    \item We first obtain the normal form $\nf(X^k)$ by the following steps.
    \begin{enumerate}
        \item Reduce $\overline{X^2}$ by the maximum times (denoted $t$) $S_{(2,4g-1)}$, then after choosing one of two symmetric cases, we obtain Eq. (\ref{Eq X with maximal t for LLFR=4g-1}):
        $$\overline{X}=\overline{b_1(b_2\cdots b_{2g})^tx_{n_5}\cdots x_{n_6}(b_{2g+1}\cdots b_{4g})^t},$$
where $\overline{X_0}:=\overline{b_1x_{n_5}\cdots x_{n_6}}$ is irreducible and cyclically freely reduced. Moreover, the $\llfr$ of $\overline{X_0^2}$ containing $\overline{x_{n_6}b_1}$ has length $\leq 4g-2$ or $\overline{x_{n_6}b_1}$ is not a fractional relator.
    
    \item If $\overline{X_0^2}$ is irreducible, then we obtain $\nf(X^k)~(k\geq 1)$ in Claim \ref{claim 3}.
    
    \item\label{outline of LLFR=4g-1, reducible and of type A} If $\overline{X_0^2}$ is reducible and $\overline{X_0}$ is of type $\mathfrak{A}$ (see Definition \ref{defn word of types ABC}), then we obtain $\nf(X^k)~(k\geq 2)$ in Claim \ref{claim 4}.
    
    \item\label{outline of LLFR=4g-1, reducible and not of type A} If $\overline{X_0^2}$ is reducible but $\overline{X_0}$ is not of type $\mathfrak{A}$, then for every $k\geq 1$ and some $n_5-1\leq i\leq n_6$, we obtain Eq. (\ref{Eq nf of X_0^k for general case before Claim 5}):
    $$\nf(X_0^k)=\overline{b_1x_{n_5}\cdots x_i W^{k-1}x_{i+1}\cdots x_{n_6}}.$$ Furthermore, we obtain $\nf(X^k)~(k\geq 2)$ in Claim \ref{claim 5} through a discussion by classification:
    \begin{enumerate}
        \item[Case](\hyperref[Case 1 for LLFR=4g-1]{1}) If $\overline{x_{n_5}\cdots x_i}=1$, then we obtain $\nf(X^k)$ in Eq. (\ref{NF of X^k LLFR=4g-1 Case 1}).
        \item[Case](\hyperref[Case 2 for LLFR=4g-1]{2}) If $\overline{x_{n_5}\cdots x_i}\neq1$, then we have two subcases:
        \begin{enumerate}
            \item[Subcase](\hyperref[Subcase 2.1 for LLFR=4g-1]{2.1}) If $\overline{x_{n_6}b_1}$ is not a fractional relator, then we obtain $\nf(X^k)$ in Eq. (\ref{NF of X^k LLFR=4g-1 Subcase 2.1}).
            \item[Subcase](\hyperref[Subcase 2.2 for LLFR=4g-1]{2.2}) If the $\llfr$ of $\overline{X_0^2}$ containing $\overline{x_{n_6}b_1}$ exists, then we obtain $\nf(X^k)$ in Eq. (\ref{NF of X^k LLFR=4g-1 Subcase 2.2}).
        \end{enumerate}
    \end{enumerate}
    \item By concluding Claim \ref{claim 3}, \ref{claim 4} and \ref{claim 5}, we obtain $\nf(X^k)~(k\geq 2)$.
    \end{enumerate}
     
    \item Finally, we obtain $\nf(x^k)=\overline{Y\nf(X^k)Y^{-1}}$, and prove Proposition \ref{prop classification of LLFR of X^2}(\ref{XX reducible}b, \ref{XX reducible}c) by concluding the above item (\ref{outline of LLFR=4g-1, reducible and of type A}) and its symmetric case respectively. 
\end{enumerate}

\subsubsection{Proof of Proposition \ref{prop classification of LLFR of X^2}(\ref{XX reducible}) ($|\overline{T}|=4g-1$)} 

Since $\overline{T}$ does not contain a letter appearing more than once, we have 
\begin{eqnarray*}
    \overline{T}&=&\overline{x_{n-r+1}\cdots x_nx_{1}\cdots x_{4g-r-1}}\\
    &=&\overline{b_{4g-r+1}\cdots b_{4g}b_{1}\cdots b_{4g-r-1}},
\end{eqnarray*}
for some $\overline{b_1\cdots b_{4g}}\in\RR$, where $1\leq r,4g-r-1\leq 2g$ and $4g-r-1< n-r+1$. Thus, $r=2g-1$ or $2g$, and we have two possible cases for the irreducible word $\overline{X}$:
\begin{equation*}
   \overline{X}=\left\{\begin{array}{ll}\overline{b_1\cdots b_{2g}x_{2g+1}\cdots x_{n-2g+1}b_{2g+2}\cdots b_{4g}} ~~(b_1\prec b_{2g})& \mbox{ when }r=2g-1 \\ \overline{b_1\cdots b_{2g-1}x_{2g}\cdots x_{n-2g}b_{2g+1}\cdots b_{4g}}~~(b_{2g+1}\prec b_{4g})& \mbox{ when }r=2g\end{array}\right.~. \end{equation*}
Because of symmetry, we assume $r=2g-1$ and hence
\begin{eqnarray}
\overline{X}&=&\overline{b_1\cdots b_{2g}x_{n_1}\cdots x_{n_2}b_{2g+2}\cdots b_{4g}},\label{Eq. A3 X}\\
\overline{T}&=&\overline{b_{2g+2}\cdots b_{4g}b_1\cdots b_{2g}}\xrightarrow{S_{(2,4g-1)}}b_1,\notag\\
  \overline{X^2}&=&\overline{x_1\cdots x_{n_2}Tx_{n_1}\cdots x_n}\notag\\
        &\xrightarrow{S_{(2,4g-1)}}&\overline{x_1\cdots x_{n_2}b_1x_{n_1}\cdots x_n}\notag\\
        &=:&\overline{H},
\end{eqnarray}
where   $n_1=2g+1$,  $n_2=n-2g+1$ (note that $n_1\leq n_2+1$ and $\overline{x_{n_1}\cdots x_{n_2}}=1$ provided $n_1= n_2+1$) and
    \begin{equation}\label{condition 1 |LLFR|=4g-1}
        b_1\prec b_{2g};\quad x_{n_2}\neq b_{2g+1},b_{2};\quad x_{n_1}\neq b_{2g+1},b_{4g}.
    \end{equation}

In the following, for every $k\geq 2$, we will first obtain the normal form $\nf(X^k)$ of $X^k$ and then obtain $\nf(x^k)=\nf(YX^kY^{-1})$ by showing 
$$\nf(YX^kY^{-1})=\overline{Y\nf(X^k)Y^{-1}}.$$
	
     Note that $\overline{H}$ may still contain a reducible subword $\overline{T'}$ of type $S_{(2,4g-1)}$. If so, $\overline{T'}$ is of the following form
\begin{eqnarray}\label{Eq. A3 T'}
\overline{T'}&=&\overline{x_{n_2-r'+1}\cdots x_{n_2}b_1x_{n_1}\cdots x_{n_1+4g-r'-3}}\\
&=&\overline{b_{4g-r'+1}\cdots b_{4g}b_1b_2\cdots b_{4g-r'-1}},\notag
\end{eqnarray}
 where $0\leq r',4g-r'-2\leq 2g$. Thus, $2g-2\leq r'\leq 2g$. By combining Eqs. (\ref{Eq. A3 X})(\ref{Eq. A3 T'}), we shall show $$r'=4g-r'-2=2g-1.$$ Indeed, 
 \begin{enumerate}
     \item if $r'=2g$, then
     $$\overline{x_{n-4g+2}\cdots x_n}=\overline{b_{2g+1}(b_{2g+2}\cdots b_{4g})^2}$$ is a reducible subword of $\overline{X}$ because $b_{2g+1}=b_{1}^{-1}\succ b_{2g}^{-1}=b_{4g}$ (see Eq. (\ref{condition 1 |LLFR|=4g-1})), which contradicts the irreducibility of $\overline{X}$;
     \item if $r'=2g-2$, then $$\overline{x_1\cdots x_{4g}}=\overline{b_1(b_2\cdots b_{2g})^2b_{2g+1}}$$ is a reducible subword of $\overline{X}$, which is also contradictory.
 \end{enumerate}
 Therefore, $r'=4g-r'-2=2g-1$ and Eqs. (\ref{Eq. A3 T'})(\ref{Eq. A3 X}) become:
 \begin{eqnarray*}
     \overline{T'}&=&\overline{x_{n_2-2g+2}\cdots x_{n_2}b_1x_{n_1}\cdots x_{n_1+2g-2}}\\&=&\overline{b_{2g+2}\cdots b_{4g}b_1b_2\cdots b_{2g}},\\
     \overline{X}&=&\overline{b_1(b_2\cdots b_{2g})^2x_{n_3}\cdots x_{n_4}(b_{2g+2}\cdots b_{4g})^2},
 \end{eqnarray*}
where $n_3\leq n_4+1$. Similarly, after $t\geq 1$ times this process (with $t$ maximal), we have
\begin{eqnarray}\label{Eq X with maximal t for LLFR=4g-1}
\overline{X}&=&\overline{\underbrace{b_1(b_2\cdots b_{2g})^t}_A\underbrace{x_{n_5}\cdots x_{n_6}}_B\underbrace{(b_{2g+2}\cdots b_{4g})^t}_D},\label{Eq. A3 Xt}
\end{eqnarray}
where $n_5\leq n_6+1; ~t\geq 1$ and
\begin{equation}\label{|LLFR|=4g-1, condition for n_5,n_6}
   b_1\prec b_{2g};\quad x_{n_5}\neq b_{2g+1},b_{4g}; \quad x_{n_6}\neq b_{2g+1},b_{2}.
\end{equation}
Note that 
\begin{eqnarray}\label{DA=b1}
\overline{DA}&=&\overline{x_{n_6+1}\cdots x_nx_1\cdots x_{n_5-1}}\notag\\
&=&\overline{(b_{2g+2}\cdots b_{4g})^tb_1(b_2\cdots b_{2g})^t}\notag\\
&\xrightarrow{t\cdot S_{(2,4g-1)}}&\overline{b_1}, 
\end{eqnarray}
then we have
\begin{eqnarray}
	\overline{X^2}&=&\overline{AB(DA)BD}\notag\\
    &\xrightarrow{t\cdot S_{(2,4g-1)}}&\overline{ABb_1BD} =:\overline{Z}.\label{Eq. A3 Z}
\end{eqnarray}

After analyzing every item of Lemma \ref{lem for Xw types}, we can obtain 
\begin{eqnarray}\label{Eq. A.3 X0}
   \overline{X_0}:=\overline{b_1B}=\overline{b_1x_{n_5}\cdots x_{n_6}}=\nf(x_{n_6+1}\cdots x_nx_1\cdots x_{n_6})
\end{eqnarray}
is irreducible. (For convenience, we only take one item as an example: if $\overline{X_0}$ contains a reducible subword of type $S_{(4,1)}$, then $\overline{X_0}=\overline{b_1(b_2\cdots b_{2g})x_{n_{5}+2g-1}\cdots x_{n_6}}$ where $b_1\succ b_{2g}$. This contradicts Eq. (\ref{|LLFR|=4g-1, condition for n_5,n_6}).) Moreover, $\overline{X_0}$ is cyclically freely reduced because $x_{n_6}\neq b_{2g+1}=b_{1}^{-1}$, and the length of the $\llfr$ $\overline{F}\subset\overline{X_0^2}$ containing $\overline{x_{n_6}b_1}$ (if $\overline{F}$ exists) is not $4g-1$ (otherwise, $\overline{X^2}$ can be reduced by $S_{(2,4g-1)}$ $(t+1)$ times).

If $n_5=n_6+1$, then $\overline{B}=\overline{x_{n_5}\cdots x_{n_6}}=1$, $\overline{X}= \overline{AD}$, and hence for any $k\geq 1$,
$$\nf(X^k)=\overline{Ab_1^{k-1}D}.$$
Therefore, we suppose $n_5\leq n_6$, i.e., $|\overline{X_0}|\geq 2$ in the following.
\begin{claim}\label{claim 3}
    If $\overline{X_0^2}$ is irreducible, then $\overline{Z}$ in Eq. (\ref{Eq. A3 Z}) is irreducible, and for any $k\geq 1$, we have
\begin{eqnarray*}
\nf(X^k)=\overline{ABX_0^{k-1}D}=\overline{x_1\cdots x_{n_6}X_0^{k-1}x_{n_6+1}\cdots x_n},
\end{eqnarray*}
    where $\overline{X_0}=\nf(x_{n_6+1}\cdots x_n x_1\cdots x_{n_6}).$

\end{claim}
 \begin{proof}[Proof of Claim \ref{claim 3}]
Suppose there is a reducible subword $$\overline{V}\subset \overline{Z}=\overline{ABb_1BD}$$
of types $S_{(2,2g+1)}$, $S_{(3,t')}~(t'\geq 2)$ or $S{(4,1)}$.  
Note that $\overline{AB},\overline{Bb_1},\overline{b_1B},\overline{Bb_1B}$ and $ \overline{BD}$ are subwords of  the irreducible words $\overline{X}$ or $\overline{X_0^2}=\overline{b_1Bb_1B}$, they are all irreducible. Hence $\overline{V}$ has the following possible positions.
\begin{enumerate}
	\item $\overline{V}$ starts in $\overline{A}$ and ends in $\overline{b_1}$, that is, $$\overline{V}=\overline{\cdots b_{2g}Bb_1}=\overline{\cdots b_{2g}x_{n_5}\cdots x_{n_6}b_1},$$
where  $x_{n_6}\neq b_2$ by Eq. (\ref{|LLFR|=4g-1, condition for n_5,n_6}).
If $\overline{V}$ is of type $S_{(2,2g+1)}$, $S_{(4,1)}$ or $S_{(3,t')}$, then it is of the following forms respectively:
    $$\overline{b_{2g+1}b_{2g}b_{2g-1}\cdots b_2b_1}, \quad\overline{b_{2g}b_{2g-1}\cdots b_{2}b_1}, \quad \overline{b_{2g+1}(b_{2g}\cdots b_2)^{t'}b_1}.$$
It implies $x_{n_6}= b_2$, a contradiction. By symmetry, the case when $\overline{V}$ starts in $\overline{b_1}$ and ends in $\overline{D}$ can be proved impossible in a similar manner. 
	
	\item $\overline{V}$ starts in $\overline{A}$ and ends in $\overline{BD}$, that is, $$\overline{V}=\overline{\cdots b_{2g}Bb_1x_{n_5}\cdots}=\overline{\cdots b_{2g}x_{n_5}\cdots x_{n_6}b_1x_{n_5}\cdots}.$$ 
    Since $x_{n_5}$ appears more than once in $\overline{V}$,  we have $\overline{V}$ is not of type $S_{(2,2g+1)}$ or $S_{(4,1)}$. Thus, $\overline{V}$ is of type $S_{(3,t')}(t'\geq 2)$. Let $\overline{V'}$ be the word obtained by deleting $\overline{Bb_1}$ from $\overline{V}$. Then by Lemma \ref{lem some observations}(7), $\overline{V'}$ is again reducible, which contradicts that $\overline{V'}=\overline{\cdots b_{2g}x_{n_5}\cdots}\subset  \overline{X}$ is irreducible.  By symmetry, the case when $\overline{V}$ starts in $\overline{AB}$ and ends in  $\overline{D}$ can also be proven impossible in a similar manner. 
\end{enumerate}
Therefore, $\overline{Z}=\overline{ABX_0D}$ is irreducible. Furthermore, since $\overline{X_0^2}$ is irreducible, by Lemma \ref{lem from 2-nd power to n-th power}, we obtain
	\begin{eqnarray*}
\nf(X^k)&=&\overline{ABX_0^{k-1}D}\\
		&=&\overline{x_1\cdots x_{n_6}X_0^{k-1}x_{n_6+1}\cdots x_n}.
	\end{eqnarray*}
  for any $k\geq 1$.  \end{proof}

We now suppose $\overline{X_0^2}$ is reducible. 

\begin{claim}\label{claim 4}
   If $\overline{X_0^2}$ is reducible and $\overline{X_0}$ is of type $\mathfrak{A}$ (see Definition \ref{defn word of types ABC}), 
   then $\overline{X}$ is of type $\mathfrak{B}$, and for any $k\geq 2$,
\begin{eqnarray*}
\nf(X^k)&=&\overline{Ax_{n_5}\cdots x_i (x_1'\cdots x_{j}'W^{k-2} x_{j+1}'\cdots x_{n'}') x_{i+1}\cdots x_{n_6}D}\\
&=&\overline{x_{1}\cdots x_i (x_1'\cdots x_{j}'W^{k-2} x_{j+1}'\cdots x_{n'}') x_{i+1}\cdots x_{n}},
\end{eqnarray*}
where $\overline{x_1'\cdots x_{n'}'}=\nf(x_{i+1}\cdots x_nx_1\cdots x_i)$ and $\overline{W}=\nf(x_{j+1}'\cdots x_{n'}'x_1'\cdots x_{j}')$ with $|\overline{W}|=n'$.
\end{claim}

\begin{proof}[Proof of Claim \ref{claim 4}]
Since $\overline{X_0}$ is of type $\mathfrak{A}$, by Eq. (\ref{Eq. A.3 X0}) and combining Eqs. (\ref{LLFR=2g special word 1})(\ref{LLFR=2g special word 2}), we have $$\overline{X_0}=\overline{b_1x_{n_5}\cdots x_{n_6}}=\overline{A_0B_0},\quad \nf(X_0^2)=\overline{A_0C_0D_0B_0},\quad \nf(X_0^3)=\overline{A_0C_0WD_0B_0},$$
where $\overline{b_{1}\cdots b_{4g}}\in\RR,~x_{n_5}\neq b_{4g},~x_{n_6}\neq b_2$ (see Eq. (\ref{|LLFR|=4g-1, condition for n_5,n_6})) and 
$$\overline{C_0D_0}=\nf(B_0A_0),\qquad \overline{W}=\nf(D_0C_0),$$
\begin{equation*}
    \left\{\begin{array}{lll}
					\overline{A_0}&=&\overline{b_{1}\underbrace{\cdots b_{2g-r_0}(b_{2-r_0}\cdots b_{2g-r_0})^{t_1}}_{A'_0} }\\
					\overline{B_0}&=&\overline{b_{2-r_0}\cdots b_{2g-1-r_0}(b_{1-r_0}\cdots b_{2g-1-r_0})^{t_2}b_{1-r_0}\cdots b_{4g}}\\
                    \overline{C_0}&=&\overline{b_{2-r_0}\cdots b_{2g-r_0}}\\
                    \overline{D_0}&=&\overline{(b_{2g-1-r_0}\cdots b_{1-r_0})^{t_2}b_{2g-1-r_0}\cdots b_{2-r_0} (b_{2g-r_0}\cdots b_{2-r_0})^{t_1}b_{1-r_0}}\\
                    \overline{W}&=&\overline{(b_{2g-1-r_0}\cdots b_{1-r_0})^{t_2}b_{2g-1-r_0}\cdots b_{2-r_0}(b_{2g-r_0}\cdots b_{2-r_0})^{t_1+1}b_{1-r_0}}
    \end{array}\right.\left(\begin{array}{c}  
        1\leq r_0\leq 2g-1, \\
        b_{1-r_0}\succ b_{2g-r_0}, \\
        t_1,t_2\geq 0
    \end{array}\right).
\end{equation*}
Note that all subscripts here are$\mod{4g}$, i.e. $b_{-i}=b_{4g-i}$, and $\overline{A_0'}=\overline{(b_{2g+3}\cdots b_{4g} b_{1})^{t_1}}$ provided $r_0=2g-1$. Then, Eq. (\ref{Eq. A3 Xt}) becomes
\begin{eqnarray*}
\overline{X}&=&\overline{\underbrace{b_1(b_2\cdots b_{2g})^t}_A\underbrace{A_0'B_0}_B\underbrace{(b_{2g+2}\cdots b_{4g})^t}_D}.
\end{eqnarray*}
Hence, $\overline{X}$ is of type $\mathfrak{B}$. 
Furthermore, let $$\overline{E}:=\overline{AA_0'C_0D_0B_0D}, \quad \overline{E'}:=\overline{AA_0'C_0WD_0B_0D}.$$ 
We can directly verify the irreducibility of $\overline{E}$ and $\overline{E'}$. Note as elements (not words) in the surface group $G$ in Eq. (\ref{symmetric presentation}), $DA=b_1,~B=A_0'B_0$ and $B_0b_1A_0'=C_0D_0$. Then in the surface group $G$, 
\begin{eqnarray*}
X^2&=&(ABD)^2=AA_0'C_0D_0B_0D=E,\\
X^3&=&(ABD)^3=ABb_1Bb_1BD=AA_0'C_0WD_0B_0D=E'.
\end{eqnarray*}
Thus, $\nf(X^2)=\overline{E}$ and $\nf(X^3)=\overline{E'}$. By Lemma \ref{lem types of W}(7), $\overline{W}$ is cyclically irreducible. 

Finally, by Lemma \ref{lem from 2-nd power to n-th power}, we obtain $\nf(X^k)(k\geq 2)$ as follows:
$$\nf(X^k)=\overline{AA_0'C_0W^{k-2}D_0B_0D},$$
where $|\overline{W}|=|\overline{C_0D_0}|$ and
\begin{equation*}
    \left\{\begin{array}{lll}
        \overline{C_0D_0} &=&\nf(B_0b_1A_0')=\nf(B_0DAA_0')  \\
         \overline{W}&=&\nf(D_0C_0)
    \end{array}\right..
\end{equation*}
The proof of Claim \ref{claim 4} is complete.
\end{proof}

We now suppose $\overline{X_0^2}$ is reducible, but the irreducible word $$\overline{X_0}=\overline{b_1B}=\overline{b_1\underbrace{x_{n_5}\cdots x_{n_6}}_B}$$ in Eq. (\ref{Eq. A.3 X0}) is not of type $\mathfrak{A}$, i.e., $\overline{X}$ is not of type $\mathfrak{B}$ (see Definition \ref{defn word of types ABC}). 

Let $\overline{F}$ be the $\llfr$ of $\overline{X_0^2}$ containing $\overline{x_{n_6}b_1}$. Then the length $|\overline{F}|\leq 4g-2$ if $\overline{F}$ exists. Therefore, by Proposition \ref{prop classification of LLFR of X^2}(\ref{XX reducible}) ($\overline{x_nx_1}$ is not a fractional relator, $|\overline{T}|\in\{2g-1,\dots,4g-2\}$), we can obtain 
\begin{equation}\label{Eq nf of X_0^k for general case before Claim 5}
    \nf(X_0^k)=\overline{b_1x_{n_5}\cdots x_{i}W^{k-1}x_{i+1}\cdots x_{n_6}},\quad n_5-1\leq i\leq n_6,
\end{equation}
for every $k\geq 1$, where
\begin{equation}\label{Eq LLFR=4g-1 W is cyclically irreducible}
    \overline{W}=\nf(x_{i+1}\cdots x_{n_6}b_1x_{n_5}\cdots x_i)\xlongequal{\mbox{by Eq. (\ref{DA=b1})}} \nf(x_{i+1}\cdots x_nx_1\cdots x_i):=\overline{w_1\cdots w_{m_0}}
\end{equation}
is cyclically irreducible and ``$i=n_5-1$'' means $\overline{x_{n_5}\cdots x_i}=1$ and $\overline{W}=\nf(x_{n_5}\cdots x_{n_6}b_1)$. Note that $$\overline{x_{n_5}\cdots x_iWx_{i+1}\cdots x_{n_6}}=\nf(x_{n_5}\cdots x_{n_6}b_1x_{n_5}\cdots x_{n_6})=\nf(Bb_1B).$$ Then, by Eq. (\ref{Eq. A3 Z}),
\begin{equation}\label{length=4g-1 word1}
\overline{Z}=\overline{ABb_1BD}\xrightarrow{S_{(i)}\mbox{'s}}\overline{\underbrace{b_1(b_2\cdots b_{2g})^t}_{A}\underbrace{x_{n_5}\cdots x_i}_{B'}W\underbrace{x_{i+1}\cdots x_{n_6}}_{C'}\underbrace{(b_{2g+2}\cdots b_{4g})^t}_{D}}=:\overline{E},
\end{equation}
which is freely reduced.

To prove the irreducibility of $\overline{E}$, we need to analysis the potential reducible subwords of $\overline{E}$.
Note that $i+1\leq n_6$, i.e., $\overline{C'}\neq 1$ (otherwise, $\overline{X_0^2}$ is irreducible), and the subwords $\overline{AB'},\overline{B'WC'},\overline{C'D}$ of $\overline{X}$ or $\nf(X_0^2)$ are irreducible. Then the position of a reducible subword (if it exists) $\overline{V}\subset\overline{E}$ of types $S_{(2,2g+1)},S_{(3)}$ or $S_{(4,1)}$ can only be the following:
	\begin{enumerate}
		\item $\overline{V}$ starts in $\overline{A}$ and ends in $\overline{W}$;
		\item $\overline{V}$ starts in $\overline{A}$ and ends in $\overline{C'D}$;
		\item $\overline{V}$ starts in $\overline{AB'}$ and ends in $\overline{D}$;
		\item $\overline{V}$ starts in $\overline{W}$ and ends in $\overline{D}$.
	\end{enumerate}

\begin{claim}\label{claim 5}
If $\overline{X_0^2}$ is reducible but $\overline{X_0}$ is not of type $\mathfrak{A}$, then $\overline{E}$ in Eq. (\ref{length=4g-1 word1}) is irreducible and thus for every $k\geq 1$,
$$\nf(X^k)=\overline{Ax_{n_5}\cdots x_iW^{k-1}x_{i+1}\cdots x_{n_6}D}=\overline{x_1\cdots x_iW^{k-1}x_{i+1}\cdots x_n}~,$$
where $\overline{W}=\nf(x_{i+1}\cdots x_nx_1\cdots x_i)$.
\end{claim}
\begin{proof}[Proof of Claim \ref{claim 5}]
We shall discuss in two cases.

\textbf{Case (1).}\label{Case 1 for LLFR=4g-1}
$\overline{B'}=\overline{x_{n_5}\cdots x_{i}}=1$, i.e. $i=n_5-1$. Then Eq. (\ref{Eq LLFR=4g-1 W is cyclically irreducible}) becomes $$\overline{W}=\nf(x_{n_5}\cdots x_{n_6}b_{1}).$$
By a specific analysis and noting $x_{n_6}\neq b_2$ (see Eq. (\ref{|LLFR|=4g-1, condition for n_5,n_6})), we know that $\overline{x_{n_5}\cdots x_{n_6}b_1}$ satisfies Lemma \ref{lem for Xw types}(4), i.e., 
	$$\overline{x_{n_5}\cdots x_{n_6}b_1}=\overline{x_{n_5}\cdots x_{n_7}(b_{2g+2}\cdots b_{4g})^{t'}b_1}\xrightarrow{S_{(4,t')}}\overline{x_{n_5}\cdots x_{n_7}b_1(b_{4g}\cdots b_{2g+2})^{t'}}=\overline{W},$$
	where $n_7:=n_6-(2g-1)t'(t'\geq 1)$ and
    \begin{equation}\label{eq. xn7+1... xn6}
        \overline{x_{n_7+1}\cdots x_{n_6}}=\overline{(b_{2g+2}\cdots b_{4g})^{t'}}, \quad x_{n_5}\neq b_{2g+1},b_{4g}, \quad x_{n_7}\neq b_2, b_{2g+1}, \quad b_{2g+2}\succ b_{1}.
    \end{equation}
    Then, Eq. (\ref{length=4g-1 word1}) becomes
    \begin{equation}\label{LLFR=4g-1 Case 1 Eq E}
        \overline{E}=\overline{\underbrace{b_{1}(b_2\cdots b_{2g})^t}_{A}\underbrace{x_{n_5}\cdots x_{n_7}b_1(b_{4g}\cdots b_{2g+2})^{t'}}_{W}\underbrace{x_{n_5}\cdots x_{n_7}(b_{2g+2}\cdots b_{4g})^{t'}}_{C'}\underbrace{(b_{2g+2}\cdots b_{4g})^t}_{D}}.
    \end{equation} 
    Note there is another partition of $\overline{E}$ with irreducible words $\overline{A'}(\subset \overline{X}),~\overline{P}$ and $\overline{D'}(\subset \overline{X})$ :
    $$\overline{E}=\overline{\underbrace{b_{1}(b_2\cdots b_{2g})^tx_{n_5}\cdots x_{n_7}}_{A'}\underbrace{b_1(b_{4g}\cdots b_{2g+2})^{t'}}_{P}\underbrace{x_{n_5}\cdots x_{n_7}(b_{2g+2}\cdots b_{4g})^{t'+t}}_{D'=C'D}}.$$
    Hence, if $\overline{E}$ contains a reducible subword $\overline{V}$ of types $S_{(2,2g+1)},S_{(3,t_0)}(t_0\geq 2)$ or $S_{(4,1)}$, then $\overline{V}$ contains $\overline{x_{n_7}b_1}$ or $\overline{b_{2g+2}x_{n_5}}$. However, 
	\begin{enumerate}
		\item If $\overline{V}$ is of type $S_{(2,2g+1)}$, then we have two possible cases:
		\begin{eqnarray}
			\overline{V}&=&\overline{x_{n_7-2g+1}\cdots x_{n_7}b_1}=\overline{b_{2g+1}\cdots b_{4g} b_1}\qquad\qquad\mbox{or}\label{length=4g-1 word2}\\
			\overline{V}&=& \overline{b_{2g+2}x_{n_5}\cdots x_{n_5+2g-1} }=\overline{b_{2g+2}b_{2g+3}\cdots b_{4g}b_1b_2}.\label{length=4g-1 word3}
		\end{eqnarray} 
		In the case (\ref{length=4g-1 word2}), since $b_{2g+1}=b_1^{-1}\succ b_{2g}^{-1}=b_{4g}$ (see Eq. (\ref{condition 1 |LLFR|=4g-1})), by Eqs. (\ref{eq. xn7+1... xn6})(\ref{length=4g-1 word2}), the irreducible word $\overline{X}$ has a reducible subword
        $$\overline{x_{n_7-2g+1}\cdots x_{n_7}x_{n_7+1}\cdots x_{n_6} x_{n_6+1}\cdots  x_n}=\overline{b_{2g+1}(b_{2g+2}\cdots b_{4g})^{t+t'+1}}.$$
		In the case (\ref{length=4g-1 word3}), note that $\overline{V}$ starts in $\overline{W}$, thus $\overline{V}$ must end in $\overline{D}$ in Eq. (\ref{LLFR=4g-1 Case 1 Eq E}) but $b_2\notin \overline{D}$. Therefore, both of the above two cases are impossible.
		
		\item If $\overline{V}$ is of type $S_{(3,t_0)}(t_0\geq 2)$, then we have three possible cases:
		\begin{eqnarray}
			\overline{V}&=&\overline{x_{n_7-(2g-1)t_0}\cdots x_{n_7}b_1}=\overline{b_{2g+1}(b_{2g+2}\cdots b_{4g})^{t_0} b_1}\qquad \qquad\qquad\qquad\qquad~\mbox{ or}\label{length=4g-1 word4}\\
            \overline{V}&=&\overline{x_{n_7-(2g-1)(t_0-1)}\cdots x_{n_7}b_1b_{4g}\cdots b_{2g+2}}\ =\overline{b_2(b_1b_{4g}\cdots b_{2g+3})^{t_0} b_{2g+2}}\quad\ \mbox{or}\label{length=4g-1 word4'}\\
			\overline{V}&=& \overline{b_{2g+2}x_{n_5}\cdots x_{n_5+(2g-1)t_0} }=\overline{b_{2g+2}(b_{2g+3}\cdots b_{4g}b_1)^{t_0}b_2}.\label{length=4g-1 word5}
		\end{eqnarray} 
		In the cases (\ref{length=4g-1 word4})(\ref{length=4g-1 word4'}), the irreducible word $\overline{X}$ has the following reducible subwords respectively:
        \begin{eqnarray*}
            \overline{x_{n_7-(2g-1)t_0}\cdots x_n}&=&\overline{b_{2g+1}(b_{2g+2}\cdots b_{4g})^{t+t'+t_0}},\\
            \overline{x_{n_7-(2g-1)t_0}\cdots x_{n_7}x_{n_7+1}}&=&\overline{b_2(b_1b_{4g}\cdots b_{2g+3})^{t_0-1}b_{2g+2}}.
        \end{eqnarray*}
    In the case (\ref{length=4g-1 word5}), note that $\overline{V}$ starts in $\overline{W}$, thus $\overline{V}$ must end in $\overline{D}$ in Eq. (\ref{LLFR=4g-1 Case 1 Eq E}) but $b_2\notin \overline{D}$. Therefore, all of the above three cases are also impossible.
		
		\item If $\overline{V}$ is of type $S_{(4,1)}$, then
		\begin{eqnarray}
			\overline{V}&=&\overline{x_{n_7-2g+2}\cdots x_{n_7}b_1}=\overline{b_{2g+2}\cdots b_{4g} b_1}\qquad\qquad\mbox{ or}\label{length=4g-1 word6}\\
			\overline{V}&=& \overline{b_{2g+2}x_{n_5}\cdots x_{n_5+2g-1} }=\overline{b_{2g+2}b_{2g+3}\cdots b_{4g}b_1}.\label{length=4g-1 word7}
		\end{eqnarray} 
		In the cases (\ref{length=4g-1 word6})(\ref{length=4g-1 word7}), $\overline{V}$ ends (starts) in $\overline{W}$, thus $\overline{V}$ must start in $\overline{A}$ (end in $\overline{D}$ ) in Eq. (\ref{LLFR=4g-1 Case 1 Eq E}) respectively, but $b_{2g+2}\notin \overline{A}$ and $b_1\notin \overline{D}$.
		Therefore, both of them are impossible.
	\end{enumerate}
	 
	In conclusion, for Case (1), we have proven that $\overline{E}$ in Eq. (\ref{length=4g-1 word1}) is irreducible. Since $\overline{W}$ is cyclically irreducible (see Eq. (\ref{Eq LLFR=4g-1 W is cyclically irreducible})), by Lemma \ref{lem from 2-nd power to n-th power}, we can obtain $\nf(X^k)~(k\geq 1)$ as follows:
    \begin{equation}\label{NF of X^k LLFR=4g-1 Case 1}
            \nf(X^k)=\overline{x_1\cdots x_{n_5-1}W^{k-1}x_{n_5}\cdots x_n},
        \end{equation}
		where $\overline{W}=\overline{x_{n_5}\cdots x_{n_7}b_1(b_{4g}\cdots b_{2g+2})^{t'}}=\nf(x_{n_5}\cdots x_nx_1\cdots x_{n_5-1})$.\\
        
\textbf{Case (2).}\label{Case 2 for LLFR=4g-1}
$\overline{B'}=\overline{x_{n_5}\cdots x_i}\neq 1$. Then Eq. (\ref{length=4g-1 word1}) becomes
	\begin{equation}\label{E in case 2}
\overline{E}=\overline{\underbrace{b_1(b_2\cdots b_{2g})^t}_{A}\underbrace{x_{n_5}\cdots x_{i}}_{B'}W \underbrace{x_{i+1}\cdots x_{n_6}}_{C'}\underbrace{(b_{2g+2}\cdots b_{4g})^t}_{D}},
	\end{equation}
	where $n_5\leq i\leq n_6-1$. 
    According to the assumption, the $\llfr$ $\overline{F}$ satisfying
    \begin{equation}\label{LLFR F}
    \overline{x_{n_6}b_1}\in\overline{F}
\subset\overline{X_0^2}=\overline{b_1x_{n_5}\cdots x_{n_6}b_1x_{n_5}\cdots x_{n_6}}
    \end{equation}
    has the following two subcases.
	
Subcase (2.1).\label{Subcase 2.1 for LLFR=4g-1}
$\overline{F}$ does not exist, i.e., $\overline{x_{n_6}b_1}$ is not a $2$-fractional relator.

According to the discussion in the proof of Proposition \ref{prop classification of LLFR of X^2}(\ref{XX reducible})($\overline{x_nx_1}$ is not a fractional relator), $\overline{X_0}$ has a form as Eq. (\ref{length=0 condition 1}). Since $x_{n_5}\neq b_{4g}$ (see Eq. (\ref{|LLFR|=4g-1, condition for n_5,n_6})), we have
\begin{eqnarray*}
    \overline{X_0}&=&\overline{b_1x_{n_5}\cdots x_{n_6}}\\
    &=&\overline{b_1\underbrace{
    b_{2}\cdots b_{2g-1}(b_1\cdots b_{2g-1})^{t_1-1}b_{2g}x_{n_8}\cdots x_{n_9}}_{x_{n_5}\cdots x_{n_9}=B'}~\underbrace{b_{4g}(b_1\cdots b_{2g-1})^{t_2}}_{x_{n_9+1}\cdots x_{n_6}=C'}},\notag
\end{eqnarray*}
		where $t_1,t_2\geq 1$ and
        \begin{equation}\label{Eq condition for n_9 subcase 2.1 in LLFR=4g-1}
           x_{n_8}\neq b_{2g+1},b_{4g};\quad x_{n_9}\neq b_{2g+1},b_{2g}; \quad n_8:=n_5+(2g-1)t_1\leq n_9:=n_6-(2g-1)t_2-1.
        \end{equation}
        Therefore, $|\overline{x_{n_8}\cdots x_{n_9}}|\geq1$. 
Moreover, by Proposition \ref{prop classification of LLFR of X^2}(\ref{XX reducible})($\overline{x_nx_1}$ is not a fractional relator), we have 
   \begin{eqnarray*}
       \nf(X_0^2)&=&\overline{b_1B'\underbrace{(b_{2g-1}\cdots b_1)^{t_1+t_2}x_{n_8}\cdots x_{n_9}}_{W}C'},
   \end{eqnarray*}
   which means we can take $i=n_9$ in Eq. (\ref{Eq nf of X_0^k for general case before Claim 5}). Hence Eq. (\ref{E in case 2}) becomes
        \begin{eqnarray*}
           \overline{E}=\overline{AB'WC'D}
             \end{eqnarray*}
      where $\overline{B'},\overline{W}$ and $\overline{C'}$ as above with length
        \begin{equation*}
            \left\{\begin{array}{llll}
                |\overline{B'}| & =&|\overline{b_2\cdots b_{2g-1}(b_1\cdots b_{2g-1})^{t_1-1}b_{2g}x_{n_8}\cdots x_{n_9}}|&\geq 2g \\
                 |\overline{W}|&=&|\overline{(b_{2g-1}\cdots b_1)^{t_1+t_2}x_{n_8}\cdots x_{n_9}}| &\geq 4g-1\\
                 |\overline{C'}|&=&|\overline{b_{4g}(b_1\cdots b_{2g-1})^{t_2}}|&\geq 2g
            \end{array}\right.~.
        \end{equation*}

We shall show that $\overline{E}$ is irreducible. Otherwise, suppose there is a reducible subword $\overline{V}\subset\overline{E}$ of types $S_{(2,2g+1)},S_{(3)}$ or $S_{(4,1)}$, by combining the paragraph following Eq. (\ref{length=4g-1 word1}), we obtain $|\overline{V}|\geq 2g+2$. Therefore, $\overline{V}$ is of type $S_{(3)}$. However, 
		 \begin{enumerate}
		 	\item If $\overline{V}$ starts in $\overline{A}$ and ends in $\overline{W}$, then
		 	\begin{eqnarray*}
		 		\overline{V}&=&\overline{\cdots b_{2g}\underbrace{b_2\cdots b_{2g-1}(b_1\cdots b_{2g-1})^{t_1-1}b_{2g}x_{n_8}\cdots x_{n_9}}_{B'}b_{2g-1}\cdots }\subset\overline{AB'W}.
		 	\end{eqnarray*}
		 	Clearly, it can not be of type $S_{(3)}$. Thus, it's contradictory.
            
		 	\item If $\overline{V}$ starts in $\overline{A}$ and ends in $\overline{C'D}$, then by an argument similar to that of $(1)$, we can also get a contradiction.
		 	
		 	\item If $\overline{V}$ starts in $\overline{AB'}$ and ends in $\overline{D}$, then
		 	$$\overline{V}=\overline{\cdots x_{n_9}\underbrace{(b_{2g-1}\cdots b_1)^{t_1+t_2}x_{n_8}\cdots x_{n_9}}_W\underbrace{b_{4g}(b_1\cdots b_{2g-1})^{t_2}}_{C'}b_{2g+2}\cdots }.$$
		 	Clearly, it can not be of type $S_{(3)}$. Thus, it's contradictory.
		 	\item If $\overline{V}$ starts in $\overline{W}$ and ends in $\overline{D}$, then the argument is similar to that of $(3)$.
		 \end{enumerate}
		Therefore, for Subcase (2.1), we conclude that $\overline{E}$ is irreducible and obtain $\nf(X^k)~(k\geq 1)$ as follows:
        \begin{equation}\label{NF of X^k LLFR=4g-1 Subcase 2.1}
            \nf(X^k)=\overline{x_1\cdots x_{n_9}W^{k-1}x_{n_9+1}\cdots x_n},
        \end{equation}
		where $\overline{W}=\nf(x_{n_9+1}\cdots x_nx_1\cdots x_{n_9})$.
        
Subcase (2.2).\label{Subcase 2.2 for LLFR=4g-1}
$\overline{F}$ does exist. 
Then
\begin{eqnarray*}
\overline{F}&=&\overline{x_{n_6-r_1+1}\cdots x_{n_6}b_1x_{n_5}\cdots x_{n_5+s_1-1}}\\
&=&\overline{b_{4g-r_1+1}\cdots b_{4g}b_1\cdots b_{s_1+1}}, 
\end{eqnarray*}
	where 
  \begin{equation}\label{eq condition of s1 r1 in subcase 2.2 |LLFR|=4g-1}
      1\leq s_1+1, r_1\leq 2g;\qquad 2g-1\leq |\overline{F}|=r_1+s_1+1\leq 4g-2.
  \end{equation}  
 Then $x_{n_6}=b_{4g}$. Hence, $\overline{X_0}$ and Eq. (\ref{E in case 2})  become
 \begin{eqnarray}\label{Eq xn5...xn6 for subcase 2.2 |LLFR|=4g-1}
 \overline{X_0}&=&\overline{b_1\underbrace{\cdots b_{s_1+1}x_{n_5+s_1}\cdots x_{n_6-r_1}b_{4g-r_1+1}\cdots b_{4g}}_{x_{n_5}\cdots x_{n_6}}},\\
     \overline{E}&=&\overline{\underbrace{b_1(b_2\cdots b_{2g})^t}_A\underbrace{x_{n_5}\cdots x_{i}}_{B'}W\underbrace{x_{i+1}\cdots x_{n_6-1} b_{4g}}_{C'}\underbrace{(b_{2g+2}\cdots b_{4g})^t}_D},\notag
 \end{eqnarray}
where $x_{n_5+s_1}\neq b_{s_1+2},b_{2g+s_1+1}$ and $x_{n_6-r_1}\neq b_{4g-r_1},b_{2g-r_1+1}$.

To prove that $\overline{E}$ is irreducible, we suppose there is a reducible subword $\overline{V}\subset\overline{E}$ of types $S_{(2,2g+1)},S_{(3)}$ or $S_{(4,1)}$. By the paragraph proceeding Claim \ref{claim 5}, we first show

\noindent (Assertion $*$).\textit{ If such $\overline{V}$  exists, it must start in $\overline{A}$ and end in $\overline{W}$ or $\overline{C'}$, and  $s_1=0$, i.e., $x_{n_5}\neq b_{2}$ in Eq. (\ref{Eq xn5...xn6 for subcase 2.2 |LLFR|=4g-1}).}

\begin{proof}[Proof of Assertion $*$]
	Indeed, if $\overline{V}$ ends in $\overline{D}$, then
	$$\overline{V}=\overline{\cdots x_{n_6}b_{2g+2}\cdots }=\overline{\cdots b_{4g}b_{2g+2}\cdots }\subset\overline{\cdots b_{4g}D}.$$ 
    It leads to $\overline{V}=\overline{b_{2g+1}(b_{2g+2}\cdots b_{4g})^{t'}b_1}$ for some $t'\geq 2$, and thus $b_1\in \overline{D}=\overline{(b_{2g+2}\cdots b_{4g})^t}$, which is contradictory. Therefore, $\overline{V}$ starts in $\overline{A}$ and ends in $\overline{W}$ or $\overline{C'}$. Furthermore,  if $s_1\geq 1$ , then $x_{n_5}=b_{2}$ and
    $$\overline{V}=\overline{\cdots b_{2g}x_{n_5}\cdots }=\overline{\cdots b_{2g}b_2\cdots }\subset\overline{Ab_2\cdots }.$$ Thus, $\overline{V}=\overline{b_1(b_2\cdots b_{2g})^{t+t'}b_{2g+1}}$ is of type $S_{(3)}$ where $t'\geq 1$. Since $\overline{V}$ starts in $\overline{A}$ and ends in $\overline{W}$ or $\overline{C'}$, we have $\overline{V}\subset \overline{AB'WC'}$ and thus $\overline{B'WC'}=\overline{(b_2\cdots b_{2g})^{t'}b_{2g+1}\cdots }$. Then, Eq. (\ref{Eq nf of X_0^k for general case before Claim 5}) with $k=2$ becomes $$\nf(X_0^2)=\overline{b_1B'WC'}=\overline{b_1(b_2\cdots b_{2g})^{t'}b_{2g+1}\cdots }$$ is reducible, which is contradictory.  Therefore, we have $s_1=0$, i.e., $x_{n_5}\neq b_{2}$.
\end{proof}

By Assertion $*$, we have
	  \begin{equation}\label{length=4g-1 last eq.}
	  	\overline{\underbrace{b_1(b_2\cdots b_{2g})^t}_Ax_{n_5}\cdots}\supset\overline{V}=\overline{\cdots b_{2g}x_{n_5}\cdots }=\left\lbrace 
	  	\begin{array}{lll}
	  		\overline{b_{2g}b_{2g-1}\cdots b_1b_{4g}}&&\overline{V} \mbox{ is of type }S_{(2,2g+1)}\\
	  		\overline{b_{2g}(b_{2g-1}\cdots b_{1})^{t''}b_{4g}}&&\overline{V} \mbox{ is of type }S_{(3, t''\geq 2)}\\
	  		\overline{b_{2g}b_{2g-1}\cdots b_1}&&\overline{V} \mbox{ is of type }S_{(4,1)}
	  	\end{array}\right.,
	  \end{equation}
where $x_{n_5}\neq b_2,b_{2g+1},b_{4g}$ (see Eq. (\ref{|LLFR|=4g-1, condition for n_5,n_6})). 
It implies $x_{n_5}=b_{2g-1}$ and Eq. (\ref{Eq xn5...xn6 for subcase 2.2 |LLFR|=4g-1}) becomes 
\begin{equation}\label{X0 A54}
\overline{X_0}=\overline{b_1\underbrace{b_{2g-1}x_{n_5+1}\cdots x_{n_{6}-r_1}b_{4g-r_1+1}\cdots b_{4g}}_{x_{n_5}\cdots x_{n_6}}}.    
\end{equation}
By Eq. (\ref{eq condition of s1 r1 in subcase 2.2 |LLFR|=4g-1}), we have $2g-2\leq|\overline{F}|-1=r_1\leq 2g$. If $r_1=2g$, then $\overline{b_{2g+1}\cdots b_{4g}}\subset \overline{X_0}$ is of type $S_{(4,1)}$, which contradicts the irreducibility of $\overline{X_0}$. If $r_1=2g-2$, then $\overline{X_0}$ is of the form as Eq. (\ref{X in |LLFR|=2g-1}),
$$\overline{X_0}=\overline{b_1\underbrace{(b_{2g+3}\cdots b_{4g}b_1)^{t_4}b_{2}x_{n_7}\cdots x_{n_8}b_{2g+2}(b_{2g+3}\cdots b_{4g}b_1)^{t_5}b_{2g+3}\cdots b_{4g}}_{x_{n_5}\cdots x_{n_6}}} ~(t_4,t_5\geq 1),$$
which contradicts that $x_{n_5}=b_{2g-1}$. Hence, 
$$|\overline{F}|-1=r_1=2g-1, \quad  x_{n_6-r_1}=x_{n_6-2g+1}\neq b_{2g+1},$$ and Eq. (\ref{X0 A54}) becomes
\begin{eqnarray*}\label{X0 A55}
    \overline{X_0}&=&
    \overline{b_1b_{2g-1}x_{n_5+1}\cdots x_{n_{6}-2g+1}b_{2g+2}\cdots b_{4g}}\\
    &\neq& \overline{\cdots b_{2g+1} (b_{2g+2}\cdots b_{4g})^{t_6}},
\end{eqnarray*}
where $t_6\geq 1$, $b_{2g+1}=b_1^{-1}\succ b_{2g}^{-1}=b_{4g}$ (see Eq. (\ref{|LLFR|=4g-1, condition for n_5,n_6}), and the ``$\neq$'' holds because $\overline{b_{2g+1}(b_{2g+2}\cdots b_{4g})^{t_6}}$ is reducible but $\overline{X_0}$ is irreducible. Therefore, $\overline{X_0^2}$ can not be reduced by $S_{(3)}$ but can be reduced by $S_{(4)}$ first. It implies
      \begin{equation}\label{eq last X_0 for subcase 2.2 LLFR=4g-1}
          \overline{X_0}=\overline{b_1\underbrace{b_{2g-1}x_{n_5+1}\cdots x_{i}}_{x_{n_5}\cdots x_i=:B'}\underbrace{(b_{2g+2}\cdots b_{4g})^{t_7}}_{x_{i+1}\cdots x_{n_6}=:C'}},
      \end{equation}
where $t_7\geq 1;~x_{i}\neq b_{2g+1},b_2$; $i:=n_6-(2g-1)t_7\geq n_5$, and hence
\begin{eqnarray*}\label{Eq55 X02}
    \overline{X_0^2}
    &=&\overline{b_1B'\underbrace{(b_{2g+2}\cdots b_{4g})^{t_7}}_{C'}b_1\underbrace{b_{2g-1}x_{n_5+1}\cdots x_{i}}_{B'}C'}\\
    &\xrightarrow{S_{(4,t_7)}}& \overline{b_1B'\underbrace{b_{1}(b_{4g}\cdots b_{2g+2})^{t_7}b_{2g-1}x_{n_5+1}\cdots x_i}_{W}C'}\notag\\
    &=&\nf(X_0^2),\notag
\end{eqnarray*}
where the last ``$=$'' holds because $\overline{X_0}$ is not of the symmetric form of Eq. (\ref{X specail form a}) and hence  $\overline{X_0^2}$ can be reduced just once. Then combining Eqs. (\ref{Eq nf of X_0^k for general case before Claim 5})( \ref{Eq LLFR=4g-1 W is cyclically irreducible} )(\ref{length=4g-1 word1}), we have
     \begin{equation*}
\overline{E}=\overline{AB'WC'D}.
     \end{equation*}
Recall that $\overline{V}$ starts in $\overline{A}$ and ends in $\overline{W}$ or $\overline{C'}$ (see Assertion $*$). Then Eq. (\ref{length=4g-1 last eq.}) becomes 
$$\overline{V}=\left\lbrace 
	  	\begin{array}{lll}
        \overline{b_{2g}\underbrace{b_{2g-1}\cdots b_2}_{B'=x_{n_5}\cdots x_i}b_1b_{4g}}&&\overline{V} \mbox{ is of type }S_{(2,2g+1)}\\
        \overline{b_{2g}\underbrace{(b_{2g-1}\cdots b_1)^{t''-1}b_{2g-1}\cdots b_{2}}_{B'=x_{n_5}\cdots x_i}b_1b_{4g}} &&\overline{V} \mbox{ is of type }S_{(3,t''\geq 2)}\\
        \overline{b_{2g}\underbrace{b_{2g-1}\cdots b_2}_{B'=x_{n_5}\cdots x_i}b_1}&&\overline{V} \mbox{ is of type }S_{(4,1)}
        \end{array}\right.~.$$ It leads to $x_{i}=b_2$ contradicting that $\overline{X_0}$ in Eq. (\ref{eq last X_0 for subcase 2.2 LLFR=4g-1}) is irreducible. 
	 
	Therefore, for Subcase (2.2), we have concluded that $\overline{E}$ is irreducible and obtain
\begin{equation}\label{NF of X^k LLFR=4g-1 Subcase 2.2}
    \nf(X^k)=\overline{Ax_{n_5}\cdots x_{i}W^{k-1}x_{i+1}\cdots x_{n_6}D}=\overline{x_1\cdots x_{i}W^{k-1}x_{i+1}\cdots x_n},
\end{equation}
where $\overline{W}=\nf(x_{i+1}\cdots x_nx_1\cdots x_{i})$.

The proof of Claim \ref{claim 5} is complete by concluding Eqs. (\ref{NF of X^k LLFR=4g-1 Case 1})(\ref{NF of X^k LLFR=4g-1 Subcase 2.1})(\ref{NF of X^k LLFR=4g-1 Subcase 2.2}).
\end{proof}	

	 In conclusion, by combining Claims \ref{claim 3}, \ref{claim 4} and \ref{claim 5}, for every $k\geq 2$, we obtain 
	 \begin{eqnarray*}
	    \nf(X^k)&=&\overline{x_1\cdots x_i(x'_1\cdots x'_jW_0^{k-2}x'_{j+1}\cdots x'_{n'})x_{i+1}\cdots x_n}\\
        &=&\overline{\underbrace{b_1(b_2\cdots b_{2g})^t}_A\underbrace{u_1\cdots u_l}_{U_k}\underbrace{(b_{2g+2}\cdots b_{4g})^t}_D},
	 \end{eqnarray*}
	 where $\overline{x'_1\cdots x'_{n'}}=\nf(x_{i+1}\cdots x_nx_1\cdots x_i),~\overline{W_0}=\nf(x_{j+1}'\cdots x_{n'}'x_1'\cdots x_j')$ with length $|\overline{W_0}|=n'$,  $t\geq 1$ is maximal and for any $t'\geq 0$,
\begin{equation}\label{eq.1 for proof of 4g-1 conjugate}
    \overline{U_k}~\neq~ \overline{(b_2\cdots b_{2g})^{t'}b_{2g+1}\cdots},\quad \overline{\cdots b_{2g+1}(b_{2g+2}\cdots b_{4g})^{t'}}.
\end{equation}
Otherwise, we get a contradiction that $\nf(X^k)$ contains the following reducible subwords $\overline{U'}$ respectively:
\begin{equation*}
\overline{U'}=\left\{\begin{array}{lll}
     \overline{Au_1...u_{(2g-1)t'+1}}&=&\overline{b_1(b_2\cdots b_{2g})^{t+t'}b_{2g+1}}\\
     \overline{u_{l-(2g-1)t'+1}\cdots u_lD}& =&\overline{b_{2g+1}(b_{2g+2}\cdots b_{4g})^{t+t'}}
\end{array}\right.,
\end{equation*}
where $\overline{U'}$ is of types $S_{(3,t+t')}$ or $S_{(4,t+t')}$ (since $b_{2g+1}=b_1^{-1}\succ b_{2g}^{-1}=b_{4g}$ see Eq. (\ref{|LLFR|=4g-1, condition for n_5,n_6})). Thus Eq. (\ref{eq.1 for proof of 4g-1 conjugate}) holds. Note that $\nf(x)=\overline{YXY^{-1}}=\overline{YA\cdots DY^{-1}}$ is irreducible, then for any $t'\geq 0$,
\begin{equation}\label{eq.2 for proof of 4g-1 conjugate}
    \overline{Y}=\overline{y_1\cdots y_m}\neq \overline{\cdots b_{2g+1}(b_{2g+2}\cdots b_{4g})^{t'}}~(y_m\neq b_{4g}, b_{2g+1}).
\end{equation}

Furthermore, note that $\overline{Y\nf(X^k)Y^{-1}}$ is freely reduced. If $\overline{Y\nf(X^k)Y^{-1}}$ contains a reducible subword $\overline{V}$ of types $S_{(2,2g+1)},S_{(3)}$ or $S_{(4,1)}$, then there are only two possible positions for $\overline{V}$:
$$\overline{V}=\overline{\cdots y_mb_1(b_2\cdots b_{2g})^tu_1\cdots}\quad \mbox{ or }\quad \overline{V}= \overline{\cdots u_l(b_{2g+2}\cdots b_{4g})^ty_m^{-1}\cdots}.$$
However, by Eqs. (\ref{eq.1 for proof of 4g-1 conjugate})(\ref{eq.2 for proof of 4g-1 conjugate}),  
$$y_m\neq b_{4g}, b_{2g+1};\quad y_{m}^{-1}\neq b_1,b_{2g};\quad u_1\neq b_{2g+1},b_{4g};\quad u_l\neq b_{2g+1},b_2.$$
The first position of $\overline{V}$ is impossible, because there is an $\llfr$ $\overline{b_1\cdots b_{2g}}$ not containing the first and the last letters of $\overline{V}$ and thus $\overline{V}$ is not of types $S_{(2,2g+1)},S_{(3)}$ or $S_{(4,1)}$. In the second position, there is an $\llfr$ $\overline{b_{2g+2}\cdots b_{4g}}$. Then, $\overline{V}=\overline{b_{2g+1}(b_{2g+2}\cdots b_{4g})^{t_0}b_1}$($t_0>t$) is of type $S_{(3)}$. Thus $$\overline{U_kD}=\overline{\cdots b_{2g+1}(b_{2g+2}\cdots b_{4g})^{t_0}}$$ is a reducible subword of $\nf(\overline{X^k})$. It is contradictory. 

Therefore, $\overline{Y\nf(X^k)Y^{-1}}$ is irreducible and thus $\nf(x^k)=\overline{Y\nf(X^k)Y^{-1}}$. By concluding Claim \ref{claim 4} and its symmetric case, we can obtain Proposition \ref{prop classification of LLFR of X^2}(\ref{XX reducible}b, \ref{XX reducible}c). The proof of Proposition \ref{prop classification of LLFR of X^2} is complete.

\section{Lists of reductions}\label{sect appendix tables for reductions}
In the following tables, we always assume $\overline{b_1\cdots b_{4g}}, ~\overline{b'_1\cdots b'_{4g}}\in \RR$, and use the top-left label $^*$ to indicate the presence of symmetry. Every word in the reducing-subword pairs is irreducible. Therefore, we will not provide any parameter declarations implied by the irreducibility. For example, in Table \ref{tab: reduction class 1},  the irreducibility of the two words in the reducing-subword pair implies $b_1\prec b_{2g}$ and $b_2\prec b_{2g+1}$, and these  will not be recorded in the parameter declaration. 

\begin{rem}\label{rem classification}
   Note that the proof of the key Proposition \ref{prop classification of LLFR of X^2} is divided into five parts according to the length of $\llfr$ at the junction. See the beginning of Section \ref{sect. proof of key Proposition}. The following five tables exactly corresponds to those five parts of the proof.
\end{rem}


\begin{table}[ht]
    \centering
    \begin{tabular}{|M{5.2cm}|M{2.7cm}|M{1.8cm}|M{2.0cm}|}
    \hline
        Reducing-subword pair $(\overline{C_1}, \overline{C_2})$ & $\nf(C_1C_2)$ & Reducing operation & Parameter declaration \\
    \hline
     $(\overline{b_1(b_2\cdots b_{2g})^{t}},\overline{(b_2\cdots b_{2g})^{t'}b_{2g+1} })$    & $\overline{(b_{2g}\cdots b_2)^{t+t'}}$ & $S_{(3,t+t')}$ & $t,t'\geq 1$\\
    \hline
    \end{tabular}
    \newline
    \caption{The junction has no fractional relator.}
    \label{tab: reduction class 1}
\end{table}

\begin{table}[ht]
    \centering
    \begin{tabular}{|M{4.7cm}|M{2.9cm}|M{1.8cm}|M{2.3cm}|}
    \hline
        Reducing-subword pair $(\overline{C_1},\overline{C_2})$& $\nf(C_1C_2)$ & Reducing operation & Parameter declaration \\
    \hline
       $(\overline{b_1(b_2\cdots b_{2g})^{t}b_2\cdots b_{r+1}},$ $ \overline{b_{r+2}\cdots b_{2g}(b_2\cdots b_{2g})^{t'}b_{2g+1}})$  & $\overline{(b_{2g}\cdots b_2)^{t+t'+1}}$ & $S_{(3,t+t'+1)}$ & $1\leq r\leq 2g-2;$  $~ t,t'\geq 1~$\\
    \hline
    \end{tabular}
\newline
\caption{The length of the $\llfr$ at the junction is $2g-1$.}
\label{tab: reduction class 2g-1}
\end{table}

{
\setlength{\LTleft}{0pt plus 1fill}
\setlength{\LTright}{0pt plus 1fill}
\addtolength{\LTleft}{-\parindent}
\addtolength{\LTright}{\parindent}

\begin{longtable}{|M{4.65cm}|M{3.9cm}|M{1.35cm}|M{1.8cm}|}  
  \captionsetup{width=0.82\textwidth}
  \hline  
  Reducing-subword pair $(\overline{C_1},\overline{C_2})$ & $\nf(C_1C_2)$& Reducing operation & Parameter declaration \\
  \hline  
  \endhead  
  \hline  
  \endfoot  
  \hline\noalign{\vspace{14pt}}
  \caption{The length of the $\llfr$ at the junction is $2g$.  In all the cases, $1\leq r\leq 2g-1 $.}
    \label{tab: reduction class 2g}
  \endlastfoot  
    $^*(\overline{b_1\cdots b_r},$ $\overline{b_{r+1}\cdots b_{2g}(b_2\cdots b_{2g})^{t-1}b_{2g+1}})$   & $\overline{(b_{2g}\cdots b_2)^t}$ & $S_{(3,t)}$ & $t\geq 2$\\
    \hline
    $^*(\overline{b_{4g}b_1\cdots b_{2g-1}b_1\cdots b_r},$ $\overline{b_{r+1}\cdots b_{2g}(b_2\cdots b_{2g})^{t-1}b_{2g+1}})$ & $\overline{b_{2g-1}\cdots b_1b_{2g-1}\cdots b_2}$ $\cdot\overline{(b_{2g}\cdots b_2)^{t-1}}$ &$S_{(3,t)}$ $S_{(2,2g+1)}$ & $t\geq 2$\\
    \hline
    $^*(\overline{b_{4g}(b_1\cdots b_{2g-1})^{t}b_1\cdots b_r},$ $\overline{b_{r+1}\cdots b_{2g}(b_2\cdots b_{2g})^{t'-1}b_{2g+1}})$ & $\overline{(b_{2g-1}\cdots b_1)^{t}b_{2g-1}\cdots b_2}$ $\cdot\overline{(b_{2g}\cdots b_2)^{t'-1}}$ &$S_{(3,t')}$  $S_{(3,t)}$ & $t, t'\geq 2$\\
    \hline
    $^*(\overline{(b_1\cdots b_{2g-1})^{t}b_1\cdots b_r},$ $\overline{b_{r+1}\cdots b_{2g}(b_2\cdots b_{2g})^{t'-1}b_{2g+1}}) $&$\overline{b_{2g}(b_{2g-1}\cdots b_1)^{t}}$ $\cdot\overline{b_{2g-1}\cdots b_2(b_{2g}\cdots b_2)^{t'-1}} $
    &$ S_{(3,t')}$  $S_{(4,t)}$&$b_1\succ b_{2g}$;  $t\geq 1;t'\geq 2$\\%
    \hline
    $^*(\overline{b_1\cdots b_r},$ $\overline{b_{r+1}\cdots b_{2g}(b_2\cdots b_{2g})^{t-1}}) $&$\overline{(b_{2g}\cdots b_2)^{t}b_1} $&$ S_{(4,t)}$&$b_1\succ b_{2g};$ \quad $t\geq 1$\\
    \hline
    $^*(\overline{(b_1\cdots b_{2g-1})^{t'}b_1\cdots b_r},$ $\overline{b_{r+1}\cdots b_{2g}(b_2\cdots b_{2g})^{t-1}}) $&$\overline{b_{2g}(b_{2g-1}\cdots b_1)^{t'}}$ $\cdot\overline{b_{2g-1}\cdots b_2(b_{2g}\cdots b_2)^{t-1}b_1} $
    &$ S_{(4,t)}$  $S_{(4,t')}$&$b_1\succ b_{2g};$ $t\geq 2;t'\geq 1$\\
    \hline  
\end{longtable} 
}

{\setlength{\LTleft}{0pt plus 1fill}
\setlength{\LTright}{0pt plus 1fill}
\addtolength{\LTleft}{-\parindent}
\addtolength{\LTright}{\parindent}

\begin{longtable}{|M{3.3cm}|M{4.3cm}|M{1.4cm}|M{2.7cm}|}  
  \captionsetup{width=0.82\textwidth}
  \hline  
  Reducing-subword pair $(\overline{C_1},\overline{C_2})$ & $\nf(C_1C_2)$& Reducing operation & Parameter declaration \\
  \hline  
  \endhead  
  
  \hline  
  \endfoot  
  
  \hline\noalign{\vspace{14pt}}
  \caption{The length of the $\llfr$ at the junction is $r+s\in\{2g+1,2g+2,\dots,4g-2\}(r,s\geq 1)$. In all the cases, $\overline{L}:=\overline{b_1\cdots b_{2g-1}},$ $\overline{R}:=\overline{b_{r+s-2g+2}\cdots b_{r+s}};$  $\overline{L'}:=\overline{b_{2g-1}\cdots b_1},$ $\overline{R'}:=\overline{b_{r+s}\cdots b_{r+s-2g+2}}.$}
    \label{tab: reduction class 2g+1 to 4g-2}
  \endlastfoot  
  
  $(\overline{b_{4g}L b_1\cdots b_r},$ $ \overline{b_{r+1}\cdots b_{r+s}Rb_{r+s+1}})$   & 
  $\overline{L'b_{2g-1}\cdots b_{r+s-2g+2}R'}$ & 
  $S_{(2,r+s)}$ $S_{(2,2g+1)}$ $S_{(2,2g+1)}$ & \\
    \hline
    $^*(\overline{b_{4g}L b_1\cdots b_r}$,\quad $\overline{b_{r+1}\cdots b_{r+s}R^{t'}b_{r+s+1}})$ & $\overline{L'b_{2g-1}\cdots b_{r+s-2g+2}(R')^{t'}}$&
    $S_{(2,r+s)}$ $S_{(2,2g+1)}$ $S_{(3,t')}$ &$t'\geq 2$ \\
    \hline 
    $^*(\overline{b_{4g}L b_1\cdots b_r},$ $\overline{b_{r+1}\cdots b_{r+s}R^{t'}})$   & $\overline{L'b_{2g-1}\cdots b_{r+s-2g+2}}$ $ \cdot \overline{(R')^{t'}b_{r+s-2g+1}}$ & $S_{(2,r+s)}$ $S_{(2,2g+1)}$ $S_{(4,t')}$ &$t'\geq 1;$ $b_{r+s}\prec b_{r+s-2g+1}$\\
    \hline
    $^*(\overline{b_{4g}L b_1\cdots b_r},$ $\overline{b_{r+1}\cdots b_{r+s}})$   & $\overline{L'b_{2g-1}\cdots b_{r+s-2g+1}}$ & $S_{(2,r+s)}$ $S_{(2,2g+1)}$ & \\
    \hline
    $(\overline{b_{4g}L^t b_1\cdots b_r},$ $\overline{b_{r+1}\cdots b_{r+s}R^{t'}b_{r+s+1}})$ & $\overline{(L')^tb_{2g-1}\cdots b_{r+s-2g+2}(R')^{t'}}$ & $S_{(2,r+s)}$ $S_{(3,t)}$ $S_{(3,t')}$ &$t,t'\geq 2$ \\
    \hline
    $^*(\overline{b_{4g}L^t b_1\cdots b_r},$ $\overline{b_{r+1}\cdots b_{r+s}R^{t'}})$   & 
    $\overline{(L')^tb_{2g-1}\cdots b_{r+s-2g+2}}$ $ \cdot \overline{(R')^{t'}b_{r+s-2g+1}}$ 
    & $S_{(2,r+s)}$ $S_{(3,t)}$ $S_{(4,t')}$ &$t\geq 2; t'\geq 1;$ $b_{r+s}\prec b_{r+s-2g+1}$\\
    \hline
    $^*(\overline{b_{4g}L^t b_1\cdots b_r},$ $\overline{b_{r+1}\cdots b_{r+s}})$   & $\overline{(L')^tb_{2g-1}\cdots b_{r+s-2g+1}}$ & $S_{(2,r+s)}$ $S_{(3,t)}$ &$t\geq 2$ \\
    \hline
    $(\overline{L^t b_1\cdots b_r},$ $\overline{b_{r+1}\cdots b_{r+s}R^{t'}})$   & $\overline{b_{2g}(L')^tb_{2g-1}\cdots b_{r+s-2g+2}}$ $ \cdot \overline{(R')^{t'}b_{r+s-2g+1}}$ & $S_{(2,r+s)}$ $S_{(4,t)}$ $S_{(4,t')}$ &$t,t'\geq 1;$ $ b_1\succ b_{2g};$ $b_{r+s}\prec b_{r+s-2g+1}$\\%
    \hline
    $^*(\overline{L^t b_1\cdots b_r},$ $\overline{b_{r+1}\cdots b_{r+s}})$   & $\overline{b_{2g}(L')^tb_{2g-1}\cdots b_{r+s-2g+1}}$ & $S_{(2,r+s)}$ $S_{(4,t)}$ &$b_1\succ b_{2g}$; ~$t\geq 1$ \\
    \hline
    $(\overline{b_1\cdots b_r},$ $\overline{b_{r+1}\cdots b_{r+s}})$   & $\overline{b_{2g}\cdots b_{r+s-2g+1}}$ & $S_{(2,r+s)}$  &  \\
\end{longtable} 
}

\begin{table}[ht]
\centering
    \begin{tabular}{|M{3.7cm}|M{1.8cm}|M{1.7cm}|M{4.5cm}|}
    \hline
        Reducing-subword pair $(\overline{C_1},\overline{C_2})$ & $\nf(C_1C_2)$ & Reducing operation & Parameter declaration \\
    \hline
        $^*(\overline{X_1(b_{2g+2}\cdots b_{4g})^t},$ $\overline{b_1(b_2\cdots b_{2g})^tX_2})$   & $\nf(X_1b_1X_2)$ & 
        $t\cdot S_{(2,4g-1)},$ $\cdots$  &
        $t\geq 1$; $\overline{X_1}=\overline{X_2}=1$ or $   
        (\overline{X_1}, ~\overline{b_1X_2})$ satisfies one of the previous cases in Tables \ref{tab: reduction class 1}, \ref{tab: reduction class 2g-1}, \ref{tab: reduction class 2g}, \ref{tab: reduction class 2g+1 to 4g-2} with $x_{p_2}\in\{ b_{4g}, b_{2g-1}\}$ and $x_{q_1}\in\{1, b_2,b_{2g+3}\}$\\
    \hline 
    \end{tabular}
   \caption{The length of the $\llfr$ at the junction is $4g-1$. In this table, $\overline{X_1}:=\overline{x_{p_1}x_{p_1+1}\cdots x_{p_2}}$ and $\overline{X_2}:=\overline{x_{q_1}x_{q_2+1}\cdots x_{q_2}}$. The parameter declaration of $x_{p_2},x_{q_1}$ is explained in the proof of Lemma \ref{lem observations of reductions}(1).}
    \label{tab: reduction class 4g-1} 
\end{table}



\end{document}